\documentclass[11pt,reqno]{amsart}

\usepackage{amsmath, amsfonts, amsthm, amssymb, color, cite}

\usepackage{latexsym,hyperref} 

\newtheorem{theorem}{Theorem}[section]
\newtheorem{lemma}{Lemma}[section]

\theoremstyle{definition}
\newtheorem{Definition}{Definition}[section]

\theoremstyle{remark}
\newtheorem{remark}{Remark}[section]

\numberwithin{equation}{section}
\allowdisplaybreaks

\newcommand{\R}{{\mathbb R}}

\def\f{\frac}

\def\hf1{^\f{1}{1-\xi^2}}

\def\be{\begin{equation}}
\def\en{\end{equation}}
\def\bs{\begin{split}}
\def\es{\end{split}}
\def\ba{\begin{align}}
\def\ea{\end{align}}

\newcommand{\supp}{{\rm supp}}

\author[Feimin Huang]{Feimin Huang}
\address{Institute of Applied Mathematics, Academy of Mathematics and Systems Science, Chinese Academy of Sciences, Beijing 100190, China.}  
\email{fhuang@amt.ac.cn}

\author[Dehua Wang]{Dehua Wang}
\address{Department of Mathematics, University of Pittsburgh, Pittsburgh, PA 15260, USA;  
Department of Mathematics, Harbin Engineering University, Harbin 150001,   China}
\email{dwang@math.pitt.edu}

\author[Difan Yuan]{Difan Yuan}
\address{School of Mathematical Sciences, University of Chinese Academy of Sciences, Beijing 100049, China; Institute of Applied Mathematics,
 AMSS, CAS, Beijing 100190, China.}
\email{yuandf@amss.ac.cn}

\title[vortex sheet: two-phase flow]
{Nonlinear stability and existence of vortex sheets for inviscid liquid-gas two-phase flow
}


\date{}

\begin{document}
\begin{abstract}
We are concerned with the vortex sheet solutions for the inviscid two-phase flow in two dimensions.
In particular, the nonlinear stability and existence of compressible vortex sheet solutions under small perturbations are established by using a modification of the Nash-Moser iteration technique, where a \textit{priori} estimates for the linearized equations have a loss of derivatives. Due to the jump of the normal derivatives of densities of liquid and gas, we obtain the normal estimates in the anisotropic Sobolev space, instead of the usual Sobolev space. New ideas and techniques are developed to close the energy estimates and derive the tame estimates for the two-phase flows.
\end{abstract}
\maketitle

\section{Introduction}
Two-phase flows or multi-phase flows are important in many industrial applications, for instance, in aerospace, chemical engineering, micro-technology and so on. They have attracted studies from many mechanical engineers, geophysicists and astrophysicists. In multi-phase flows,   two or more states of matters are considered simultaneously. We   focus  on the mixed flow of liquid and gas, which is an interesting phenomenon in nature that occurs in boilers, condensers, pipelines for oil and natural gas, etc. For more  physical background, we refer the readers to \cite{Pai1977}.
In this paper, we consider the following compressible inviscid liquid-gas two-phase isentropic flow of drift-flux type:
\begin{eqnarray}\label{t}
\left\{ \begin{split}
\displaystyle &\partial_tm+\nabla\cdot(m \mathbf{u_g})=0,\\
\displaystyle &\partial_tn+\nabla\cdot(n\mathbf{u_l})=0,\\
\displaystyle &\partial_t(m \mathbf{u_g}+n \mathbf{u_l})+\nabla\cdot(m \mathbf{u_g}\otimes\mathbf{u_g}+n\mathbf{u_l}\otimes\mathbf{u_l})+\nabla p(m,n)=0,
\end{split}
\right.
\end{eqnarray}
where, $m=\alpha_g\rho_g$ and $n=\alpha_l\rho_l$ represent the mass of gas and liquid respectively,
$\alpha_g,\alpha_l\in[0,1]$ are the gas and liquid volume fractions satisfying $\alpha_g+\alpha_l=1,$  $\rho_g$ and $\rho_l$ denote the density of gas and liquid respectively; 
$p$ is the pressure of two phases,
$\mathbf{u_g},\mathbf{u_l}$ denote the velocity of gas and liquid respectively.
Here, we study the simplified two-phase flows, assuming that the velocities of gas and liquid are equal in the fluid region, $\mathbf{u_g}=\mathbf{u_l}=\mathbf{u}.$  It is also quite interesting to study the slip case, i.e., the two phase velocities are different.
Due to the fact that the liquid phase is much heavier than the gas phase, the momentum of gas phase in the mixture momentum equations is small in comparison with that of liquid phase and thus neglected.
Hence, we write the simplified model as follows:
\begin{eqnarray}\label{t2}
\left\{ \begin{split}
\displaystyle &\partial_tm+\nabla\cdot(m \mathbf{u})=0,\\
\displaystyle &\partial_tn+\nabla\cdot(n\mathbf{u})=0,\\
\displaystyle &\partial_t(n\mathbf{u})+\nabla\cdot(n\mathbf{u\otimes u})+\nabla p(m,n)=0.\\
\end{split}
\right.
\end{eqnarray}
We will follow the simplification of pressure term $p(m,n)$ as in \cite{Evje2008}. It guarantees that the model is consistent in the sense of thermal-mechanics, which helps to get basic energy estimate. The pressure term $p$ is a smooth function of $(m,n)$ defined on $(0,+\infty)\times(0,+\infty).$ We take the pressure of the following form:
$$p(m,n)=(c_1m+c_2n)^2P'(c_1m+c_2n),$$
where $c_1,c_2$ are two positive constants and $P=P(\rho)$ is a smooth function. For example, the pressure satisfies $\gamma$-law $P(\rho)=\rho^{\gamma-1},\gamma>1.$ Without loss of generality, we take $c_1=c_2=1$ and take the pressure term as
\begin{equation}\label{pressure}
p(m,n)=(\gamma-1)(m+n)^{\gamma},
\end{equation}
and
\begin{equation}\label{pressure2}
p_m(m,n)=p_n(m,n)=\gamma(\gamma-1)(m+n)^{\gamma-1}.
\end{equation}
The system \eqref{t2} is a non-strictly hyperbolic system of conservation law in the region $(m,n)\in(0,\infty)\times(0,\infty).$

For the viscous two-phase flow, the well-posedness in various one-dimensional and multi-dimensional cases has been studied in \cite{EvjeJDE2011,EvjeSIAM2011,Evje2008,Evje2009,FriisSIAM20111,FriisSIAM20112,Hao2012}. For the inviscid two-phase flow, limited theoretic works have been done before. We are concerned with the vortex sheet problem for the two phase-flow  $\eqref{t2}$ in the two-dimensional space $\R^2.$  Vortex sheet is a surface across which there is a discontinuity in the fluid tangential velocity, such as in slippage of one layer of fluid over another. The tangential components of the flow velocity are discontinuous across the vortex sheet, however, the normal component of the flow velocity is continuous. For a brief introduction to the compressible vortex sheets in the applications of multi-dimensional hyperbolic conservation laws, please refer to \cite{Gavage2007,Ruan2016} and for the related studies of  incompressible flows, please refer to \cite{Jiu2006,Jiu2007,WangC2012,Wu2006} and their references.

Our paper is inspired by the stability of planar and rectilinear vortex sheets for the compressible isentropic Euler equations \cite{Coulombel2004, Coulombel2008}, the non-isentropic Euler equations \cite{Morando2008} and the ideal compressible MHD equations \cite{ChenG2008,Trakhinin2005}. The linear stability of compressible vortex sheets for the isentropic Euler equations in two dimensions was studied by Coulombel-Secchi \cite{Coulombel2004}. Using the linear analysis and Nash-Moser technique, the nonlinear stability for the compressible Euler equations was proved in \cite{Coulombel2008}. These two papers are the pioneering work of compressible vortex sheets. Later on, the result on the linear stability for the isentropic Euler equations was extended to the non-isentropic Euler equations in\cite{Morando2008} (see also \cite{MTW2018} for a very recent preprint on nonlinear stability). Recently, Morando-Secchi-Trebeschi in \cite{Morando2018} provided an alternative way to prove the linear stability of compressible vortex sheets in two space dimensions using the evolution equation for the discontinuity front of the vortex sheet.
For the MHD equations, the linear and nonlinear stability of compressible current-vortex sheets was proved   in \cite{ChenG2008,WangYARMA2013,Trakhinin2005}. Recently, Chen-Hu-Wang \cite{RChen2017,RChen2018} and Chen-Hu-Wang-Wang-Yuan \cite{RChen20182} studied the linear and nonlinear stability of compressible vortex sheets in two dimensional elastodynamics, using the method of upper triangularization, delicate spectrum analysis, and Nash-Moser technique. Some other models of contact discontinuities can be found and have been studied in \cite{ChenG2017,Morando20182,Trakhinin2009,Huang2018,WangYARMA2013,WangYJDE2013,WangY2015}.

In this paper, we establish the nonlinear stability and local existence of compressible vortex sheet solutions for the two-phase flow. Strongly motivated by  the paper of Coulombel and  Secchi 
\cite{Coulombel2008}, we extend the linear stability result in Ruan-Wang-Weng-Zhu \cite{Ruan2016} to the nonlinear stability using the estimates in anisotropic Sobolev spaces. The difficulty for the two-phase flow, compared with the Euler flow, is that the difference of the densities of the liquid and gas makes the analysis more complicated and much harder. We have one more equation and one more variable density than the Euler equations. Moreover, since the vortex sheet is a contact discontinuity, the free boundary is characteristic. The linearized problems could be solved with the price of a loss of derivatives. As usual, we are thus led to employing the Nash-Moser procedure to compensate the loss of derivatives with respect to the source terms in each iteration step. It requires to derive the tame estimates on the higher order derivatives of the solutions of the linear systems. If we adopt the usual method in \cite{Coulombel2008},   we also can  estimate the normal derivatives, using certain vorticity equation that is a transport equation with some source containing lower order derivatives of       the unknowns than the terms in certain form of transport equation. However, comparing this form of equations with the Euler equations, some new extra terms appear, which come from the jump of the normal derivatives of the density. These terms cannot be controlled in energy estimates.
In order to overcome this new obstacle occurring in the normal derivatives, we restrict our solution into the anisotropic Sobolev space, instead of the usual Sobolev space.  The key point is that the one-order gain of normal differentiation could be compensated by two-order loss of tangential differentiation. Using the weighted functions on the normal derivatives and the properties of the anisotropic space, we can close all the derivative estimates and derive the tame estimates. 
We can prove the weighted energy estimates for the equations with transport structures, even though the source terms contain some higher order derivatives. Applications of anisotropic spaces may be found in \cite{Shuxing,Secchi1995}.
We remark that in \cite{Ruan2016}, a good symmetric form of the linearized version of liquid-gas two-phase flow was introduced, which plays a crucial role in the linear stability analysis. This symmetrizer is   different from \cite{Coulombel2008} and \cite{Morando2008}.
Focusing on the linearized problem with variable coefficients, we can prove the well-posedness for the effective linearized equations in $L^2$ where the source terms are in $L^2(H^1)\times H^1.$  As a first step, we obtain the a \textit{priori} estimates of the tangential derivatives. Next, we prove the unweighted and weighted normal derivatives in anisotropic spaces instead of the usual Sobolev spaces  in \cite{Coulombel2008}.
Finally, we obtain the whole a \textit{priori} tame estimate  in the weighted anisotropic Sobolev spaces, which is linearly dependent on higher order norms  multiplied by the lower order norms. This is the key step to obtain the solvability of linearized equations in preparation for the Nash-Moser iteration. There is a fixed loss of derivatives with respect to the source terms and the coefficients.
The idea in our paper can be also extended to study the nonlinear stability of non-isentropic two-phase flow. 
We remark that in the recent preprint \cite{Ruan2018},   Ruan and   Trakhinin introduced a different symmetrization   for the  two-phase flow models to prove the local existence of compressible shock waves and 2D compressible vortex sheets by the  result in \cite{MTW2018}.

The rest of the paper is organized as follows. In Section \ref{vortex},
we introduce the set-up of the vortex sheet problem, and provide some definitions and notations. In Section \ref{tame}, we first formulate the weighted Sobolev spaces and norms and then derive the effective linearized equations. After that, we prove the well-posedness of the linearized problems.
In Section \ref{compatibility}, we construct the approximate solutions, which will be useful in Section \ref{nash}, where the Nash-Moser iteration scheme is introduced to get the local existence of the linearized problems. In Section \ref{tameestimate4}, we summarize all the tame estimates that will be needed in the proof of Theorem \ref{stability}.
In Section \ref{proof}, we complete the proof of our main Theorem \ref{stability} by proving the convergence of the Nash-Moser scheme.
We  emphasize  that in spite of the new difficulties and new techniques for the proof of the nonlinear stability of vortex sheets in this paper, we shall still follow closely both  the classical procedure and the classical presentation of the pioneering paper \cite{Coulombel2008} in the context of  two-phase flows \eqref{t2}.
In Appendices  A and B, some basic estimates in weighted Sobolev spaces, anisotropic Sobolev spaces and the proof of a \textit{priori} estimates of tangential derivatives are provided separately.


\section{Vortex Sheet Problem}\label{vortex}

In this section, we shall introduce the vortex sheet problem  and its formulation as an  initial boundary-value problem for hyperbolic conservation laws. We will consider the liquid-gas two-phase flow \eqref{t2} in the whole space $\R^2$. Let $(x_1,x_2)$ be the space variable in  $\R^2$ and let $v$ and $u$ represent the components of two velocities, $\textbf{u}=(v,u)\in\R^2.$ Then, for
$$U=(m,n,\textbf{u})^T\in(0,+\infty)\times(0,+\infty)\times\R^2,$$
we can define the following matrices:
\begin{equation}\label{A}      
 A_1(U)=\left(                 
 \begin{array}{cccc}
    v & 0 & m & 0\\
    0 & v & n & 0\\
    \frac{p_m}{n} & \frac{p_n}{n}& v & 0\\
    0 & 0& 0 & v\\
  \end{array}
\right), \quad
A_2(U)=\left(                 
 \begin{array}{cccc}
    u & 0 & 0 & m\\
    0 & u & 0 & n\\
   0 & 0& u & 0\\
    \frac{p_m}{n} & \frac{p_n}{n}& 0 & u\\
  \end{array}
\right).
\end{equation}
For simplicity, we write the spatial partial derivatives by
$\partial_1=\partial_{x_1},\,  \partial_2=\partial_{x_2}.$
When $(m,n,\textbf{u})$ is smooth, \eqref{t2} is equivalent to the following quasilinear form:
\begin{equation}\label{quasi}
\partial_tU+A_1(U)\partial_1U+A_2(U)\partial_2U=0.
\end{equation}
It is noted that the system \eqref{quasi} or \eqref{t2} is linearly degenerate with respect to the second characteristic field. Hence, we focus on the vortex sheet ~(contact discontinuity) solutions for the two-phase flow.  
\begin{Definition}\label{weak}
Suppose that $(m,n,\textbf{u})$ is a smooth function of $(t,x_1,x_2)$ on either side of a smooth surface $\Gamma:=\{x_2=\varphi(t,x_1),t>0,x_1\in\R\}.$ Then, $(m,n,\textbf{u})$ is a weak solution of \eqref{t2} if and only if $(m,n,\textbf{u})$ is a classical solution of \eqref{t2} on both sides of $\Gamma$ and the Rankine-Hugoniot conditions hold at each point of $\Gamma:$
\begin{eqnarray}\label{RH}
\left\{ \begin{split}
\displaystyle &\partial_t\varphi[m]-[m\textbf{u}\cdot\nu]=0,\\
\displaystyle &\partial_t\varphi[n]-[n\textbf{u}\cdot\nu]=0,\\
\displaystyle &\partial_t\varphi[n\textbf{u}]-[(n\textbf{u}\cdot\nu)\textbf{u}]-[p]\nu=0,\\
\end{split}
\right.
\end{eqnarray}
where $\nu:=(-\partial_1\varphi,1)$ is a spatial  normal vector to $\Gamma,$  and
$[q]=q^{+}-q^{-}$ denotes the jump of a quantity $q$ across the interface $\Gamma.$
\end{Definition}
%
 Following the definition of contact discontinuity solutions in the sense of Lax \cite{Lax1957}, we present the definition of a vortex sheet solution of the liquid-gas two-phase flow.
\begin{Definition}\label{vortexsheet}
A piecewise smooth function $(m,n,\textbf{u})$ is a vortex sheet solution of \eqref{t2} if $(m,n,\textbf{u})$ is a classical solution of \eqref{t2} on either side of the surface $\Gamma$ and the Rankine-Hugoniot conditions \eqref{RH} hold at each point of $\Gamma$ as follows:
\begin{equation}\label{RH2}
\partial_t\varphi=\textbf{u}^{+}\cdot\nu=\textbf{u}^{-}\cdot\nu, \quad p^{+}=p^{-}.
\end{equation}
\end{Definition}
Since $p$ is monotone, the previous relations in \eqref{RH2} read
\begin{equation}\label{RH3}
\partial_t\varphi=\textbf{u}^{+}\cdot\nu=\textbf{u}^{-}\cdot\nu, \quad m^++n^+=m^-+n^-.
\end{equation}
Then we can formulate the vortex sheet problem as the following free-boundary value problem: determine
$$U^{\pm}(t,x_1,x_2)=(m^{\pm},n^{\pm},v^{\pm},u^{\pm})^\top (t,x_1,x_2)\in(0,+\infty)\times(0,+\infty)\times\R^2$$
and a free boundary $\Gamma:=\{x_2=\varphi(t,x_1),t>0,x_1\in\R\}$ such that
\begin{eqnarray}\label{equation}
\left\{ \begin{split}
\displaystyle &\partial_tU^++A_1(U^+)\partial_1U^+
+A_2(U^+)\partial_2U^+=0, \, x_2>\varphi(t,x_1),\\
\displaystyle &\partial_tU^-+A_1(U^-)\partial_1U^-
+A_2(U^-)\partial_2U^-=0,\, x_2<\varphi(t,x_1),\\
\displaystyle &U^{\pm}(0,x_1,x_2)=\left\{ \begin{split}
\displaystyle &U^+_0(x_1,x_2), \, x_2>\varphi_0(x_1),\\
\displaystyle &U^-_0(x_1,x_2), \, x_2<\varphi_0(x_1),\\
\end{split}\right.
\end{split}
\right.
\end{eqnarray}
satisfying the Rankine-Hugoniot conditions on $\Gamma:$
\begin{equation}\label{RH4}
\partial_t\varphi=-v^+\partial_1\varphi+u^+=-v^-\partial_1\varphi+u^-,\quad m^++n^+=m^-+n^-,
\end{equation}
where $\varphi_0(x_1)=\varphi(0,x_1).$
To prove the existence of vortex sheet solutions for the free-boundary problem \eqref{equation} and \eqref{RH4}, we need to find a solution $U(t,x_1,x_2)$ and $\varphi(t,x_1)$ of the problem \eqref{equation} and \eqref{RH4} at least locally in time. In this paper, we establish the nonlinear stability of the linearized problem resulting from the linearization of \eqref{equation} and \eqref{RH4} around a background vortex sheet (piecewise constant) solution.
It is convenient to reformulate the problem in the fixed domain $\{t\in[0,T],x_1\in\R,x_2\geq0\}$ by introducing a change of variables which can be found in \cite{Coulombel2004}. After fixing the unknown front, we construct smooth solutions $U^{\pm}=(m^{\pm},n^{\pm},v^{\pm},u^{\pm})^{\top}, \, \Phi^{\pm}$ from the following system:
\begin{equation}\label{equation2}
\partial_tU^{\pm}+A_1(U^{\pm})\partial_1U^{\pm}+\frac{1}{\partial_2\Phi^{\pm}}[A_2(U^{\pm})-\partial_t\Phi^{\pm}I_{4\times4}-\partial_1\Phi^{\pm}A_1(U^{\pm})]\partial_2U^{\pm}=0
\end{equation}
in the interior domain $\{t\in[0,T],x_1\in\R,x_2>0\}$ with the boundary conditions:
\begin{equation}\label{boundary1}
\Phi^{+}|_{x_2=0}=\Phi^{-}|_{x_2=0}=\varphi,
\end{equation}
\begin{equation}\label{boundary2}
(v^+-v^-)|_{x_2=0}\partial_1\varphi-(u^+-u^-)|_{x_2=0}=0,
\end{equation}
\begin{equation}\label{boundary3}
\partial_t\varphi+v^+|_{x_2=0}\partial_1\varphi-u^+|_{x_2=0}=0,
\end{equation}
\begin{equation}\label{boundary4}
(m^++n^+)|_{x_2=0}-(m^-+n^-)|_{x_2=0}=0.
\end{equation}
Here, $I_{4\times4}$ denotes the $4\times4$ identity matrix in \eqref{equation2}.
For the initial data, we write:
\begin{equation}\label{initialdata}
(m^{\pm},n^{\pm},v^{\pm},u^{\pm})|_{t=0}=(m^{\pm}_0,n^{\pm}_0,v^{\pm}_0,u^{\pm}_0)(x_1,x_2),\quad \varphi|_{t=0}=\varphi_0(x_1).
\end{equation}
The functions $\Phi^{\pm}$ should satisfy
\begin{equation}\label{phi}
\pm\partial_2\Phi^{\pm}(t,x)\geq \kappa>0, \quad \forall (t,x_1,x_2)\in[0,T]\times\R\times\R^+,
\end{equation}
 for some $\kappa>0.$
For contact discontinuities, one can choose the change of variables $\Phi^{\pm}$ satisfying the eikonal equations:
\begin{equation}\label{eikonal}
\partial_t\Phi^{\pm}+v^{\pm}\partial_1\Phi^{\pm}-u^{\pm}=0
\end{equation}
in the whole closed half-space $\{x_2\geq0\}.$
We can rewrite the equation \eqref{equation2} into a compact form by defining the left side as a differential operator $\mathbb{L}$. Then, the system \eqref{equation2} becomes
\begin{equation}\label{L}
\mathbb{L}(U^+,\nabla\Phi^+)U^+=\mathbb{L}(U^-,\nabla\Phi^-)U^-=0,
\end{equation}
where $\nabla\Phi^{\pm}=(\partial_t\Phi^{\pm},\partial_1\Phi^{\pm},\partial_2\Phi^{\pm}).$ We can also write the system as
\begin{equation}\label{L2}
\mathbb{L}(U,\nabla\Phi)U=0,
\end{equation}
where $U$ denotes the vector $(U^+,U^-)$ and $\Phi$ for $(\Phi^+,\Phi^-).$ The boundary conditions \eqref{boundary1}-\eqref{boundary4} can be written into the following compact form:
\begin{equation}\label{boundary5}
\mathbb{B}(U^+|_{x_2=0},U^-|_{x_2=0},\varphi)=0.\\
\end{equation}
There are many simple solutions of \eqref{L}, \eqref{boundary5}, \eqref{eikonal}, and \eqref{phi}, that correspond to the stable background states:
\begin{eqnarray}\label{stationary}
(m,n,\textbf{u})=\begin{cases}
 (\bar{m}^+,\bar{n}^+,\bar{v}^+,0), &\text{if } x_2>0,\\
 (\bar{m}^-,\bar{n}^-,-\bar{v}^+,0), &\text{if } x_2<0,\\
\end{cases}
\end{eqnarray}
for the two-phase flow system  \eqref{t2} in the original variables, where $\bar{m}^\pm>0, \bar{n}^\pm>0, \bar{v}\in\R$ are constants. In the straightened variables, these stationary vortex sheets correspond to the following smooth stationary solution to \eqref{L}, \eqref{boundary5}, \eqref{eikonal}, and \eqref{phi}:
\begin{equation}\label{rectilinear}
\bar{U}^{\pm}\equiv(\bar{m}^{\pm},\bar{n}^{\pm},\pm\bar{v}^+,0)^\top, \, \bar{\Phi}^{\pm}(t,x)\equiv\pm x_2,\, \bar{\varphi}\equiv0.
\end{equation}
Here we assume $\bar{v}^{+}>0$ without loss of generality.
We write piecewise constant functions \eqref{rectilinear} into the following form:
\begin{equation}\label{rectilinear2}
\bar{U}_{r,l}\equiv(\bar{m}_{r,l},\bar{n}_{r,l},\pm\bar{v}_r,0)^\top, \, \bar{\Phi}_{r,l}(t,x)\equiv\pm x_2,\, \bar{\varphi}\equiv0.
\end{equation}
With the notations of Sobolev spaces and norms to be introduced in the next section, we state our main result as follows.

\begin{theorem}\label{stability}
Let $T>0$, $\alpha\in \mathbb{N},\alpha\geq15$,  and the stationary solution defined by \eqref{rectilinear2} satisfy the ``supersonic" condition:
\begin{equation}\label{supersonic}
\bar{v}_r-\bar{v}_l>\left(\bar{c}^{\frac{2}{3}}_r+\bar{c}^{\frac{2}{3}}_l\right)^{\frac{3}{2}},\quad \bar{v}_r-\bar{v}_l\neq\sqrt{2}(\bar{c}_r+\bar{c}_l),
\end{equation}
where $\bar{c}_{r,l}=\sqrt{(1+\frac{\bar{m}_{r,l}}{\bar{n}_{r,l}})p_n(\bar{m}_{r,l},\bar{n}_{r,l})}.$
Assume that the initial data $(U^{\pm}_0,\varphi_0)$ has the form $U^{\pm}_0=\bar{U}^{\pm}+\dot{U}^{\pm}_0,$ with $\dot{U}^{\pm}_0\in H^{2\alpha+15}_{\ast}(\R^2_+)$ and $\varphi_0\in H^{2\alpha+16}(\R)$   compatible up to order $\alpha+7$ in the sense of Definition \ref{c} and compactly supported.
Then, there exists $\delta>0,$ such that, if $[\dot{U}^{\pm}_0]_{2\alpha+15,\ast,T}+||\varphi_0||_{H^{2\alpha+16}}\leq\delta,$
\eqref{equation2}-\eqref{eikonal} has  a unique solution $U^{\pm}=\bar{U}^{\pm}+\dot{U}^{\pm},\Phi^{\pm}=\pm x_2+\dot{\Phi}^{\pm},\varphi$  on $[0,T],$  satisfying $(\dot{U}^{\pm},\dot{\Phi}^{\pm})\in H^{\alpha-1}_{\ast}((0,T)\times\R^2_+),$ and $\varphi\in H^{\alpha}((0,T)\times\R).$
\end{theorem}
\begin{remark}
It is noted that when the support of $\varphi_0$ is included in $[-1,1]$ and the support of $\dot{U}_0^{\pm}$ is included in $\{x_2\geq0,\sqrt{x^2_1+x^{2}_2}\leq1\},$ the corresponding solution $\dot{U}^{\pm},\dot{\Phi}^{\pm},\varphi$ will have a compact support:
$$\supp(\dot{U}^{\pm},\dot{\Phi}^{\pm})\subseteq \{t\in[0,T],x_2\geq0,\sqrt{x^2_1+x^2_2}\leq 1+\lambda_{max}T\}, $$ $$\supp\varphi\subseteq\{t\in[0,T],|x_1|\leq 1+\lambda_{max}T\},$$
 because of the finite speed of propagation of the two-phase flow equations and the eikonal equation,
where $\lambda_{max}$ is the maximum characteristic speed of the two-phase flow equations in a small neighborhood of the background state $\bar{U}^{\pm}.$
\end{remark}
\begin{remark}
For the critical case $\bar{v}_r-\bar{v}_l=\sqrt{2}(\bar{c}_r+\bar{c}_l),$  
the linear stability with variable coefficients of compressible vortex sheets is still open for the two-phase flows \cite{Ruan2016}
and non-isentropic Euler equations  \cite{Morando2008}, due to the existence of double roots of the Lopatinski$\breve{\rm i}$ determinant.
\end{remark}


\section{Tame Estimates for the Linearized Equations}\label{tame}

In this section, we shall follow \cite{Coulombel2008} to define the weighted Sobolev spaces and norms,   derive the effective linearized equations, and  prove the well-posedness of the linearized problems.

\subsection{Weighted spaces and norms}\label{weighted}
First, we introduce the weighted Sobolev spaces. Define the half-space
$$\Omega:=\{(t,x_1,x_2)\in \R^3, x_2>0\}=\R^2\times \R^+.$$
The boundary $\omega:=\partial\Omega$ is identified to be $\R^2.$ We denote the norm of $L^2(\R^2)$ and $L^2(\Omega)$ by $||\cdot||_{L^2(\R^2)}$ and $||\cdot||_{L^2(\Omega)}$ respectively. For all real numbers $s\in\R$ and $\lambda\geq1$, we define the weighted space:
$$H^{s}_{\lambda}(\R^2):=\{u\in \mathcal{D}'(\R^2):\; e^{-\lambda t}u\in H^s(\R^2)\}.$$
It is equipped with the $\lambda$-weighted norm: $$||u||_{H^s_{\lambda}(\R^2)}:=\sum_{|\alpha|\leq s}\lambda^{s-|\alpha|}||e^{-\lambda t}\partial^{\alpha}u||_{L^2(\R^2)}.$$ The space $L^2(\R^+;H^{s}_{\lambda}(\R^2))$ is equipped with the norm:
$$||u||^2_{L^2(H^s_{\lambda})}:=\int_0^{+\infty}||u(\cdot,x_2)||^2_{H^{s}_{\lambda}(\R^2)}dx_2.$$
For a fixed time $T>0,$ we define $\omega_T:=(-\infty,T)\times\R$ and $\Omega_{T}:=\omega_T\times\R^+.$ For an integer $\tilde{s}\ge 0$ and a real number $\lambda\geq1$,  we define the weighted Sobolev space $H^{\tilde{s}}_{\lambda}(\Omega_T)$ as:
$$H^{\tilde{s}}_{\lambda}(\Omega_T):=\{u\in \mathcal{D}'(\R^2): \; e^{-\lambda t}u\in H^{\tilde{s}}(\Omega_T)\}.$$ Similarly, we can define $H^{\tilde{s}}_{\lambda}(\omega_T).$ The norm on $H^{\tilde{s}}_{\lambda}(\Omega_T)$ may be defined by
\begin{equation}\label{norm}
||u||_{H^{\tilde{s}}_{\lambda}(\Omega_T)}:=\sum_{|\alpha|\leq \tilde{s}}\lambda^{\tilde{s}-|\alpha|}||e^{-\lambda t}\partial^{\alpha}u||_{L^2(\Omega_T)},
\end{equation}
which is equivalent to the norm $||e^{-\lambda t}u||_{H^{\tilde{s}}(\Omega_T)}.$ The constant in the equivalence is independent of $\lambda\geq1$ and $T.$ The norm on $H^{\tilde{s}}_{\lambda}(\omega_T)$ is defined in the same way.\\

Now, we introduce the weighted anisotropic Sobolev space as in \cite{Alinhac}.   For an integer $s\ge 0$ and  a real number $\lambda\geq1$,  we define
$$H^{s,\lambda}_{\ast}(\Omega_T):=\{u(t,x_1,x_2)\in \mathcal{D}'(\Omega_T):e^{-\lambda t}\partial^{\alpha}_{\ast}\partial^{k}_2u(t,x_1,x_2)\in L^2(\Omega_T), \text{ for } |\alpha|+2k\leq s\},$$
which is equipped with the norm
$$[u]_{s,\lambda,T}:=\sum_{|\alpha|+2k\leq   s}\lambda^{s-|\alpha|-2k}||e^{-\lambda t}\partial^{\alpha}_{\ast}\partial^k_2u||_{L^2(\Omega_T)}.$$
Here,   $\alpha=(\alpha_0,\alpha_1,\alpha_2)$, $\partial^{\alpha}_{\ast}:=\partial^{\alpha_0}_t\partial^{\alpha_1}_1(\sigma(x_2)\partial_2)^{\alpha_2}$ and $\sigma$ is a fixed smooth function such that $\sigma(0)=0,\sigma(x_2)=1  \text{ if } x_2>1.$
\begin{remark}
The weight $\sigma(x_2)$ can also be assumed to satisfy $\sigma(x_2)=x_2$ for sufficiently small $x_2>0.$
\end{remark}
For simplicity, we present the following trace theorem and extension theorem in the weighted anisotropic Sobolev spaces from \cite{Alinhac}.
\begin{lemma}
If $s>1$ and $u\in H^{s,\lambda}_{\ast}(\Omega_T),$ then $u|_{x_2=0}\in H^{s-1}_{\lambda}(\omega_T)$ and
$$||u|_{x_2=0}||_{H^{s-1}_{\lambda}(\omega_T)}\leq C[u]_{s,\lambda,T},$$
where $C>0$ is a constant. Moreover, if $v\in H^{s}_{\lambda}(\omega_T),$ for $s>0,$ then there is a function $u\in H^{s+1,\lambda}_{\ast}(\Omega_T),$ such that $u|_{x_2=0}=v$ and
$$[u]_{s+1,\lambda,T}\leq C||u|_{x_2=0}||_{H^s_{\lambda}(\omega_T)},$$
where $C>0$ is a constant. The same results hold for $\R^2_+$ and $\R$ instead of $\Omega_T$ and $\omega_T.$
\end{lemma}
It is noted that if $\lambda=0$, the   space $H^{s,0}_{\ast}(\Omega_T)$ becomes   the classical anisotropic space $H^{s}_{\ast}(\Omega_T)$:
$$H^s_{\ast}(\Omega_T):=\{u(t,x_1,x_2)\in \mathcal{D'}(\Omega_T):\partial^{\alpha}_{\ast}\partial^k_2u(t,x_1,x_2)\in L^2(\Omega_T) \text{ for } |\alpha|+2k\leq s\},$$
  equipped with the norm
$$[u]_{s,\ast,T}:=\sum_{|\alpha|+2k\leq s}\lambda^{s-|\alpha|-2k}||\partial^{\alpha}_{\ast}\partial^k_2u||_{L^2(\Omega_T)}.$$
It is easy to check that
$$[u]_{s,\lambda,T}=[e^{-\lambda t}u]_{s,\ast,T}.$$
Similarly, we can define $H^s_{\ast}(\R^2_+)$ and its norm $[\cdot]_{s,\ast}.$
Then, we define the following Sobolev spaces for the weighted normal derivatives and unweighted tangential derivatives:
$$W^{1,tan}(\Omega_T):=\{u(t,x_1,x_2)\in \mathcal{D'}(\Omega_T):||u||_{L^{\infty}(\Omega_T)}+\sum_{|\alpha|=1}||\partial^{\alpha}_{\ast}u||_{L^{\infty}(\Omega_T)}<\infty\},$$
$$W^{2,tan}(\Omega_T):=\{u(t,x_1,x_2)\in \mathcal{D'}(\Omega_T):||u||_{W^{1,\infty}(\Omega_T)}+\sum_{|\alpha|=1}||\partial^{\alpha}_{\ast}u||_{W^{1,\infty}(\Omega_T)}<\infty\},$$
with the norms
$$||u||_{W^{1,tan}(\Omega_T)}:=||u||_{L^{\infty}(\Omega_T)}+\sum_{|\alpha|=1}||\partial^{\alpha}_{\ast}u||_{L^{\infty}(\Omega_T)},$$
$$||u||_{W^{2,tan}(\Omega_T)}:=||u||_{W^{1,\infty}(\Omega_T)}+\sum_{|\alpha|=1}||\partial^{\alpha}_{\ast}u||_{W^{1,\infty}(\Omega_T)},$$
respectively.
We also introduce the following differential operators:
$$\nabla:=(\partial_t,\partial_1,\partial_2),\; \nabla^{tan}:=(\partial_t,\partial_1),\; \nabla^{tan}_{\ast}:=(\partial_t,\partial_1,\sigma\partial_2),$$
for the multi-index $\beta=(\alpha_0,\alpha_1,\alpha_2,k),$ we denote
$$D^{\beta}=\partial^{\alpha_0}_t\partial^{\alpha_1}_1(\sigma\partial_2)^{\alpha_2}\partial^{k}_2.$$
Here we write $\langle\beta\rangle=|\alpha|+2k=\alpha_0+\alpha_1+\alpha_2+2k.$
More details on weighted Sobolev spaces and anisotropic Sobolev spaces are provided in Appendix \ref{Appendix}.
\subsection{The effective linearized equations}\label{effective}
We introduce the linearized equations around a state that is given as a perturbation of the stationary solution \eqref{rectilinear2}. We consider the following functions:
\begin{equation}\label{perturbation}
U_r=\bar{U}_r+\dot{U}_r(t,x),\,U_l=\bar{U}_l+\dot{U}_l(t,x),\,\Phi_r=x_2+\dot{\Phi}_r(t,x),\,\Phi_l=-x_2+\dot{\Phi}_l(t,x),
\end{equation}
where  $U_{r,l}(t,x)\equiv(m_{r,l},n_{r,l},v_{r,l},u_{r,l})^\top(t,x),$ $\dot{U}_{r,l}(t,x)\equiv(\dot{m}_{r,l},\dot{n}_{r,l},\dot{v}_{r,l},\dot{u}_{r,l})^\top(t,x).$
The index $r,l$ represent the states on the right and left of the interface. We assume that  the perturbations $\dot{U}_{r,l},\dot{\Phi}_{r,l}$ have compact support:
\begin{equation}\label{compact}
\supp(\dot{U}_{r,l},\dot{\Phi}_{r,l})\subseteq\{t\in[-T,2T],x_2\geq0,\sqrt{x^2_1+x^2_2}\leq 1+2\lambda_{max}T\}.
\end{equation}
We also assume that these quantities satisfy the boundary conditions \eqref{boundary1}-\eqref{boundary4},  that is,
\begin{equation}\label{boundary6}
\begin{split}
&\Phi_r|_{x_2=0}=\Phi_l|_{x_2=0}=\varphi,\\
&(v_r-v_l)|_{x_2=0}\partial_1\varphi-(u_r-u_l)|_{x_2=0}=0,\\
&\partial_t\varphi+v_r|_{x_2=0}\partial_1\varphi-u_r|_{x_2=0}=0,\\
&(m_r+n_r)|_{x_2=0}-(m_l+n_l)|_{x_2=0}=0.
\end{split}
\end{equation}
We assume that the functions $\Phi_{r,l}$ satisfy the eikonal equations
\begin{equation}\label{eikonal2}
\begin{split}
&\partial_t\Phi_r+v_r\partial_1\Phi_r-u_r=0,\\
&\partial_t\Phi_l+v_l\partial_1\Phi_l-u_l=0,\\
\end{split}
\end{equation}
together with
\begin{equation}\label{bound}
\partial_2\Phi_r\geq\kappa_0,\; \partial_2\Phi_l\leq-\kappa_0,
\end{equation}
for a constant $\kappa_0\in(0,1)$ in the whole closed half-space $\{x_2\geq0\}.$
Let us consider some families $U^{\pm}_\kappa=U_{r,l}+\kappa V^{\pm},\Phi^{\pm}_\kappa=\Phi_{r,l}+\kappa\Psi^{\pm},$ where $\kappa$ is a small parameter. We compute the linearized equation around the state $U_{r,l},\Phi_{r,l}:$
\begin{equation}\label{linearize}
\mathbb{L}'(U_{r,l},\Phi_{r,l})(V^{\pm},\Psi^{\pm}):=\frac{d}{d\kappa}\mathbb{L}(U^{\pm}_\kappa,\Phi^{\pm}_\kappa)|_{\kappa=0}=f_{\pm}.
\end{equation}
Introducing the ``good unknowns" as in \cite{Alinhac}:
\begin{equation}\label{good}
\dot{V}_{+}=V^{+}-\frac{\Psi^+}{\partial_2\Phi_r}\partial_2U_r, \; \dot{V}_{-}=V^{-}-\frac{\Psi^-}{\partial_2\Phi_l}\partial_2U_l,
\end{equation}
we obtain that
\begin{equation}\label{linearize2}
\begin{split}
&\mathbb{L}'(U_r,\Phi_r)(V^+,\Psi^+)\\
&\quad =\mathbb{L}(U_r,\nabla\Phi_r)\dot{V}_++C(U_r,\nabla U_r,\nabla \Phi_r)\dot{V}_++\frac{\Psi^+}{\partial_2\Phi_r}\partial_2\{\mathbb{L}(U_r,\nabla\Phi_r)U_r\}=f_+,\\
&\mathbb{L}'(U_l,\Phi_l)(V^-,\Psi^-)\\
&\quad =\mathbb{L}(U_l,\nabla\Phi_l)\dot{V}_-+C(U_l,\nabla U_l,\nabla \Phi_l)\dot{V}_-+\frac{\Psi^-}{\partial_2\Phi_l}\partial_2\{\mathbb{L}(U_l,\nabla\Phi_l)U_l\}=f_-,\\
\end{split}
\end{equation}
where, for smooth functions $(U,\Phi),$ the matrix $C(U,\nabla U, \nabla\Phi)$ is defined as follows:
\begin{equation}\label{C}
C(U,\nabla U,\nabla \Phi)X:=[dA_1(U)X]\partial_1U+\frac{1}{\partial_2\Phi}\left\{[dA_2(U)X]-\partial_1\Phi[dA_1(U)X]\right\}\partial_2U.
\end{equation}
In particular, the matrices $C(U_{r,l},\nabla U_{r,l},\nabla \Phi_{r,l})$ are $C^{\infty}$ function of $(\dot{U}_{r,l},\nabla\dot{U}_{r,l},\nabla\Phi_{r,l})$  vanishing when $(\dot{U}_{r,l},\nabla\dot{U}_{r,l},\nabla\Phi_{r,l})=0.$
In view of the previous results in  \cite{Ruan2016}, we neglect the zeroth  order term in $\Psi^{\pm}$ in the linearized equation \eqref{linearize2} and thus we consider the effective linear operators:
\begin{equation}\label{effective2}
\begin{split}
&\mathbb{L}'_{e}(U_r,\Phi_r)\dot{V}_+:=\mathbb{L}(U_r,\nabla\Phi_r)\dot{V}_++C(U_r,\nabla U_r,\nabla{\Phi}_r)\dot{V}_+=f_+,\\
&\mathbb{L}'_{e}(U_l,\Phi_l)\dot{V}_-:=\mathbb{L}(U_l,\nabla\Phi_l)\dot{V}_++C(U_l,\nabla U_l,\nabla{\Phi}_l)\dot{V}_-=f_-.
\end{split}
\end{equation}
We shall solve the nonlinear problems \eqref{L}, \eqref{boundary5}, \eqref{eikonal}, and \eqref{phi} by constructing   a sequence of linear equations of the form \eqref{effective2}. The remaining terms in \eqref{linearize2} are regarded as error terms in each iteration step. It is important to obtain energy estimates and   define the dual problem.
From previous analysis \cite{Ruan2016}, we know that the effective linearized equation \eqref{linearize2} form a symmetrizable hyperbolic system. The Friedrichs symmetrizer for the operator  $\mathbb{L}'_{e}(U_{r,l},\Phi_{r,l})$ is
\begin{equation}\label{symmetrizer}      
S_{r,l}(t,x)=\left(                 
 \begin{array}{cccc}
    \frac{p_n}{m_{r,l}} & 0 & 0 & 0\\
    0 &  \frac{p_n}{n_{r,l}} & 0 & 0\\
   0 & 0 & n_{r,l} & 0\\
    0 & 0& 0 & m_{r,l}\\
  \end{array}
\right)(t,x).
\end{equation}
Combining the eikonal equation \eqref{eikonal2}, we obtain that
\begin{equation}\label{symmetrizer2}      
\frac{S_r}{\partial_2\Phi_r}[A_2(U_r)-\partial_t\Phi_rI_{4\times4}-\partial_1\Phi_rA_1(U_r)]=\frac{1}{\partial_2\Phi_r}\left(                 
 \begin{array}{cccc}
    0 & 0 & -p_n\partial_1\Phi_r & p_n\\
    0 &  0 & -p_n\partial_1\Phi_r &  p_n\\
   -p_n\partial_1\Phi_r & -p_n\partial_1\Phi_r &0 & 0\\
    p_n & p_n & 0 & 0\\
  \end{array}
\right).
\end{equation}
We thus expect to control the traces of the components $\dot{V}_{+,1}+\dot{V}_{+,2}$ and $\dot{V}_{+,4}-\partial_1\Phi_r\dot{V}_{+,3}$ on the boundary $\{x_2=0\}.$ Similarly, we expect to control the traces of the components $\dot{V}_{-,1}+\dot{V}_{-,2}$ and $\dot{V}_{-,4}-\partial_1\Phi_l\dot{V}_{-,3}$ on the boundary. Therefore, we introduce the following ``trace operator":
\begin{equation}\label{trace}
\mathbb{P}(\varphi)\dot{V}_{\pm}|_{x_2=0}:= \left(
 \begin{array}{c}
  \dot{V}_{\pm,1}+\dot{V}_{\pm,2}\\
     \dot{V}_{\pm,4}-\partial_1\Phi_{r,l}\dot{V}_{\pm,3}\\
  \end{array}
\right)\Big|_{x_2=0}.
\end{equation}
This trace operator is used in the energy estimates for the linearized equations, whose image can be regarded as the ``noncharacteristic" part of the vector $\dot{V}.$
We now turn to the linearized boundary conditions. Let us introduce the matrices:
\begin{equation}\label{bM}    
b(t,x_1)=                
 \left(\begin{array}{cc}
    0 & (v_r-v_l)|_{x_2=0}\\
    1 &  v_r|_{x_2=0}\\
    0 & 0\\
  \end{array}
\right), \,
M(t,x_1)=                
 \left(\begin{array}{cccccccc}
    0 & 0 & \partial_1\varphi & -1 & 0 & 0 & -\partial_1\varphi & 1 \\
    0 & 0 & \partial_1\varphi & -1 & 0 & 0 & 0 & 0 \\
    1 & 1 & 0& 0& -1 & -1 & 0 & 0 \\
  \end{array}
\right).
\end{equation}
Denote $\dot{V}=(\dot{V}_{+},\dot{V}_{-})^{T}$,  $\nabla\psi=(\partial_t\psi,\partial_1\psi)^T, g=(g_1,g_2,g_3)^T.$
Then, the linearized boundary conditions become equivalently
\begin{equation}\label{boundary7}
\begin{split}
&\Psi^+|_{x_2=0}=\Psi^-|_{x_2=0}=\psi,\\
&\mathbb{B}'_e(U_{r,l},\Phi_{r,l})(\dot{V}|_{x_2=0},\psi):=b\nabla \psi+\underbrace{M\left(\begin{array}{c}
  \frac{\partial_2U_r}{\partial_2\Phi_r}\\
   \frac{\partial_2U_l}{\partial_2\Phi_l}\\
  \end{array}
\right)|_{x_2=0}}_{b_{\sharp}}\psi+M\dot{V}|_{x_2=0}=g.
\end{split}
\end{equation}
We observe that the operator $\mathbb{B}'_e$ only involves $\mathbb{P}(\varphi)\dot{V}_{\pm}|_{x_2=0}.$ Here, $\mathbb{P}(\varphi)$ is defined by $\eqref{trace}.$
\subsection{The basic $L^2$ a \textit{priori} estimate}\label{priori}
In \cite{Ruan2016}, we have already derived $L^2$ a \textit{priori} estimate for \eqref{effective2} and \eqref{boundary7}. We assume that the perturbations satisfy
\begin{equation}\label{perturbation2}
\dot{U}_r,\dot{U}_l\in W^{2,\infty}(\Omega),\dot{\Phi}_r,\dot{\Phi}_l\in W^{3,\infty}(\Omega),||(\dot{U}_r,\dot{U}_l)||_{W^{2,\infty}(\Omega)}+||(\dot{\Phi}_r,\dot{\Phi}_l)||_{W^{3,\infty}(\Omega)}\leq K,
\end{equation}
where $K$ is a positive constant. Then, we have the following theorem as in \cite{Ruan2016}:
\begin{theorem}\label{aprior}
Assume that the  solution defined by \eqref{perturbation} satisfies supersonic condition \eqref{supersonic}, and the perturbations $\dot{U}_{r,l}, \nabla \dot{\Phi}_{r,l}$ satisfies \eqref{compact}
-\eqref{bound} and \eqref{perturbation2}. Then, there exist some constants $K_0>0, C_0>0$ and $\lambda_0>1,$ such that if $K\leq K_0$ and $\lambda\geq\lambda_0,$   the solution $(\dot{V},\psi)\in H^2_{\lambda}(\Omega)\times H^2_{\lambda}(\R^2)$ to the linearized problem \eqref{effective2} and \eqref{boundary7} satisfies the following estimates:
\begin{equation}\label{estimate}
\begin{split}
&\lambda||\dot{V}||^2_{L^2_{\lambda}(\Omega)}+||\mathbb{P}(\varphi)\dot{V}|_{x_2=0}||^2_{L^2_{\lambda}(\R^2)}+||\psi||^2_{H^1_{\lambda}(\R^2)}\\
&\leq C_0\left(\frac{1}{\lambda^3}||(\mathbb{L}'_e(U_r,\Phi_r)\dot{V}_+,\mathbb{L}'_e(U_l,\Phi_l)\dot{V}_-)||^2_{L^2(H^1_{\lambda})}+\frac{1}{\lambda^2}||\mathbb{B}'_e(U_{r,l},\Phi_{r,l})(\dot{V}|_{x_2=0},\psi)||^2_{H^{1}_{\lambda}(\R^2)}\right)\\
&:=C_0\left(\frac{1}{\lambda^3}||(f_+,f_-)||^2_{L^2(H^1_{\lambda})}+\frac{1}{\lambda^2}||g||^2_{H^{1}_{\lambda}(\R^2)}\right),
\end{split}
\end{equation}
where the operators $\mathbb{P}(\varphi), \mathbb{L}'_e, \mathbb{B}'_e$ are defined in \eqref{trace}, \eqref{effective2}, and \eqref{boundary7}.
\end{theorem}

\subsection{Well-posedness of the linearized equations}\label{wellposed}

In this section, we will show the well-posedness result of the linearized equations \eqref{effective2}, \eqref{boundary7} similar to \cite{Coulombel2005}. What we need to do is to check the existence of the dual problem that satisfies some a \textit{priori} estimate with a loss of one tangential derivative. Now, the definitions of the dual problem of \eqref{effective2}, \eqref{boundary7} are given below first.
On the boundary $\{x_2=0\},$ the matrix in front of the normal derivative $\partial_2$ in the operator $\mathbb{L}'_e(U_{r,l},\Phi_{r,l})$ is equal to:
\begin{equation}
\begin{split}
&\frac{1}{\partial_2\Phi_{r,l}}[A_2(U_{r,l})-\partial_t\Phi_{r,l}I_{4\times4}-\partial_1\Phi_{r,l}A_1(U_{r,l})]|_{x_2=0}\\
&=\frac{1}{\partial_2\Phi_{r,l}|_{x_2=0}}\left(\begin{array}{cccc}
    0 & 0 & -m\partial_1\varphi & m\\
    0 & 0 & -n\partial_1\varphi & n\\
    -\frac{p_m\partial_1\varphi}{n} & -\frac{p_m\partial_1\varphi}{n}  & 0 &0\\
    \frac{p_m}{n} & \frac{p_n}{n}  & 0 &0\\
  \end{array}
\right),
\end{split}
\end{equation}
where $m,n$ denote the common trace of $m_r,m_l,n_r,n_l$ and $\varphi$ is the common trace of $\Phi_r$ and $\Phi_l.$ Recall that the matrix $M$ is defined by \eqref{bM}, which is in the Sobolev space $W^{2,\infty}(\R^2).$ Now, we define the following matrices:
\begin{equation}\label{NM}
\begin{split}
&N:=               
 \left(\begin{array}{cccccccc}
    0 & 0 & \partial_{1}\varphi & -1 & 0 & 0 & -\partial_{1}\varphi & 1  \\
    0 & 0 & \partial_{1}\varphi & -1 & 0 & 0 & 0 & 0 \\
    1 & 1 & 0 & 0 & -1 & -1 & 0 & 0\\
  \end{array}
\right),\\
&M_1:=                
 \left(\begin{array}{cccccccc}
    0 & -\frac{n}{\partial_2\Phi_r} & 0 & 0 & \frac{m}{\partial_2\Phi_{l}} &  \frac{n}{\partial_2\Phi_{l}} & 0 & 0 \\
    -\frac{m}{\partial_2\Phi_r} & 0 & 0 & 0 & 0 & 0 & 0 & 0 \\
    0 & 0 & -\frac{p_n\partial_1\varphi}{2n\partial_2\Phi_r} & \frac{p_n}{2n\partial_2\Phi_r} & 0 & 0 & \frac{p_m\partial_1\varphi}{2n\partial_2\Phi_l} & -\frac{p_m}{2n\partial_2\Phi_l} \\
  \end{array}
\right),\\
&N_1:=                
 \left(\begin{array}{cccccccc}
    0 & 0 & 0 & 0 & 0 & 0 & 0 & 0 \\
    0 & 0 & 0 & 0 & 0 & 0 & 0 & 0 \\
    0 & 0 & -\frac{p_n\partial_1\varphi}{2n\partial_2\Phi_r} & \frac{p_n}{2n\partial_2\Phi_r} & 0 & 0 & \frac{p_m\partial_1\varphi}{2n\partial_2\Phi_l} & -\frac{p_m}{2n\partial_2\Phi_l} \\
  \end{array}
\right),\\
\end{split}
\end{equation}
where the derivatives of the functions $\partial_2\Phi_r,\partial_2\Phi_l$ are all evaluated on the boundary $\{x_2=0\}.$
These matrices satisfy the relation:
\begin{equation}\label{matrix}
\frac{1}{\partial_2\Phi_{r,l}}\left[A_2(U_{r,l})-\partial_t\Phi_{r,l}I_{4\times4}-\partial_1\Phi_{r,l}A_1(U_{r,l})\right]|_{x_2=0}=M^T_1M+N^T_1N.
\end{equation}
Moreover, using \eqref{perturbation2}, we see that  $M_1,N_1,N\in W^{2,\infty}(\R^2).$
We can define a dual problem for \eqref{effective2}, \eqref{boundary7} in the following way:
\begin{eqnarray}\label{dual}
\left\{ \begin{split}
\displaystyle &\mathbb{L}'_e(U_{r,l},\Phi_{r,l})^{\ast}U^{\pm}=\tilde{f}_{\pm}, \quad x_2>0,\\
\displaystyle &N_1U|_{x_2=0}=0,\\
\displaystyle &\text{div}(b^TM_1U|_{x_2=0})-b^T_{\sharp}M_1U|_{x_2=0}=0,\\
\end{split}
\right.
\end{eqnarray}
where $N_1,M_1$ are defined in $\eqref{NM}$, $b$ is defined in $\eqref{bM}$, $b_{\sharp}$ is defined in \eqref{boundary7}, div denotes the divergence operator in $\R^2$ with respect to the variables $(t,x_1)$ and the dual operator $\mathbb{L}'_e(U_{r,l},\Phi_{r,l})^{\ast}$ are formal adjoints of $\mathbb{L}'_e(U_{r,l},\Phi_{r,l})$.
As in  \cite{Coulombel2008,Ruan2016}, we have 
 the following well-posedness result.
\begin{theorem}\label{wellposed2}
Let $T>0$ be any fixed constant. Assume that the background solution \eqref{rectilinear2} satisfies \eqref{supersonic} and the perturbation $\dot{U}_{r,l},\dot{\Phi}_{r,l}$ satisfy \eqref{compact}-\eqref{bound} and \eqref{perturbation2}.
Then, there exist positive constants $K_0>0$ and $\lambda_0>1,$ independent of $T$, such that, if $K\leq K_0,$  for the source terms $f_{\pm}\in L^2(\R_+;H^1_{\lambda}(\omega_T))$ and $g\in H^1_{\lambda}(\omega_T)$ that vanish for $t<0,$ the problem:
\begin{eqnarray}\label{dual2}
\left\{ \begin{split}
\displaystyle &\mathbb{L}'_e(U_{r,l},\Phi_{r,l})\dot{V}_{\pm}=f_{\pm},t<T, x_2>0,\\
\displaystyle &\mathbb{B}'_e(U_{r,l},\Phi_{r,l})(\dot{V}|_{x_2=0},\psi)=g,t<T,x_2=0,\\
\end{split}
\right.
\end{eqnarray}
has a unique solution $(\dot{V}_{\pm},\psi)\in L^2_{\lambda}(\Omega_T)\times H^1_{\lambda}(\omega_T)$ that vanishes for $t<0$ and satisfies $\mathbb{P}(\varphi)\dot{V}_{\pm}|_{x_2=0}\in L^2_{\lambda}(\omega_T).$ Moreover, the following estimate holds for all $\lambda\geq\lambda_0$ and for all $t\in[0,T],$
\begin{equation}\label{estimate2}
\lambda||\dot{V}_{\pm}||^2_{L^2_{\lambda}(\Omega_t)}+||\mathbb{P}(\varphi)\dot{V}_{\pm}|_{x_2=0}||^2_{L^2_{\lambda}(\omega_t)}+||\psi||^2_{H^1_{\lambda}(\omega_t)}\lesssim\lambda^{-3}||f_{\pm}||^2_{L^2(H^1_{\lambda}(\omega_t))}+\lambda^{-2}||g||^2_{H^1_{\lambda}(\omega_t)}.
\end{equation}
\end{theorem}
\subsection{An equivalent formulation of the linearized equations}\label{equivalent}
For simplicity, we transform the interior equation \eqref{effective2} in order to deal  with a hyperbolic operator that has a constant diagonal boundary matrix. We refer to \cite{Ruan2016} for details.
There exists a non-orthogonal matrix $T(U,\nabla\Phi)$  defined by
\begin{equation}\label{T}      
T(U,\nabla\Phi)=\left(                 
 \begin{array}{cccc}
   1 & 0 & \frac{m}{n}\langle\partial_1\Phi\rangle &  \frac{m}{n}\langle\partial_1\Phi\rangle\\
    -1 &  0 &\langle\partial_1\Phi\rangle  &\langle\partial_1\Phi\rangle \\
   0 & 1 & -\frac{c(m,n)}{n}\partial_1\Phi  & \frac{c(m,n)}{n}\partial_1\Phi\\
    0 & \partial_1\Phi & \frac{c(m,n)}{n} & -\frac{c(m,n)}{n}\\
  \end{array}
\right),
\end{equation}
where $\langle\partial_1\Phi\rangle=\sqrt{1+(\partial_1\Phi)^2},$ $c(m,n)=\sqrt{(1+\frac{m}{n})p_n}$, and its
  inverse    is
\begin{equation}\label{T-1}      
T^{-1}(U,\nabla\Phi)=\left(                 
 \begin{array}{cccc}
   \frac{n}{m+n} & -\frac{n}{m+n} & 0 & 0\\
    0 &  0 & \frac{1}{{\langle\partial_1\Phi\rangle}^2 } & \frac{\partial_1\Phi }{{\langle\partial_1\Phi\rangle}^2 }\\
   \frac{n}{2(m+n)\langle\partial_1\Phi\rangle} & \frac{n}{2(m+n)\langle\partial_1\Phi\rangle} & -\frac{n}{2c(m,n)} \frac{\partial_1\Phi }{{\langle\partial_1\Phi\rangle}^2 }& \frac{n}{2c(m,n)} \frac{1 }{{\langle\partial_1\Phi\rangle}^2 }\\
    \frac{n}{2(m+n)\langle\partial_1\Phi\rangle} &\frac{n}{2(m+n)\langle\partial_1\Phi\rangle} &\frac{n}{2c(m,n)} \frac{\partial_1\Phi }{{\langle\partial_1\Phi\rangle}^2 } & - \frac{n}{2c(m,n)} \frac{1 }{{\langle\partial_1\Phi\rangle}^2 }\\
  \end{array}
\right).
\end{equation}
We can define new variables:
\begin{equation}\label{W}
W_+:=T^{-1}(U_r,\nabla \Phi_r)\dot{V}_+, \; W_-:=T^{-1}(U_l,\nabla \Phi_l)\dot{V}_-,
\end{equation}
then, they solve the system:
\begin{equation}\label{new}
\begin{split}
& A^r_0\partial_tW_++A^r_1\partial_1W_++I_2\partial_2W_++A^r_0C^rW_+=F_+,\\
& A^l_0\partial_tW_-+A^l_1\partial_1W_-+I_2\partial_2W_-+A^l_0C^lW_-=F_-,\\
\end{split}
\end{equation}
where the matrices $A^{r,l}_0,A^{r,l}_1$ belong to $W^{2,\infty}(\Omega)$ and the matrix $C^{r,l}$ belongs to $W^{1,\infty}(\Omega)$; moreover, $A^{r,l}_0,A^{r,l}_1$ are $C^{\infty}$ functions of $(\dot{U}_{r,l},\nabla \dot{\Phi}_{r,l}),$ while $C^{r,l}$ is a $C^{\infty}$ function of $(\dot{U}_{r,l},\nabla\dot{U}_{r,l},\nabla\dot{\Phi}_{r,l},\nabla^2 \dot{\Phi}_{r,l})$ that vanishes at the origin;  and $I_2$ is defined as:
\begin{equation}\label{I_2}      
I_2=\left(                 
 \begin{array}{cccc}
   0 & 0 & 0 & 0\\
    0 &  0 & 0 &0\\
  0 &0 &1& 0\\
 0 &0 &0 & 1\\
  \end{array}
\right).
\end{equation}
The source terms $F_{\pm}$ are defined by:
\begin{equation}\label{F}
F_{\pm}(t,x)=A^{r,l}_0T^{-1}(U_{r,l},\nabla \Phi_{r,l})f_{\pm}(t,x).
\end{equation}
The system \eqref{new} is equivalent to \eqref{effective2} because $A^{r,l}_0$ are invertible. It is noted that the source terms in \eqref{new} differ from the source terms \eqref{effective2} by an invertible matrix.
Using the vector $W=(W_+,W_-)^\top,$ the linearized boundary conditions \eqref{boundary7} become:
\begin{equation}\label{boundary8}
\begin{split}
&\Psi^+|_{x_2=0}=\Psi^-|_{x_2=0}=\psi,\\
&b\nabla\psi+b_{\sharp}\psi+\mathbf{M} W|_{x_2=0}=g,\\
\end{split}
\end{equation}
where we denote
\begin{equation}\label{M}      
\mathbf{M}:=M\left(                 
 \begin{array}{cc}
   T(U_r,\nabla\Phi_r) & 0 \\
    0 &    T(U_l,\nabla\Phi_l) \\
\end{array}
\right).
\end{equation}
The matrices $b$ and $\mathbf{M}$ belong to $W^{2,\infty}(\R^2),$ and the vector $b_{\sharp}$ belongs to $W^{1,\infty}(\R^2)$ Moreover, $b$ is a $C^{\infty}$ function of $\dot{U}_{r,l}|_{x_2=0},$ $\mathbf{M}$ is a $C^{\infty}$ function of $(\dot{U}_{r,l}|_{x_2=0},\partial_1\varphi)$ and $b_{\sharp}$ is a $C^{\infty}$ function of $(\partial_2\dot{U}_{r,l}|_{x_2=0},\partial_1\varphi,\partial_2\dot{\Phi}_{r,l}|_{x_2=0})$ that vanishes at the origin. The new boundary conditions \eqref{boundary8} only involve the ``noncharacteristic" part of the vector $W=(W_+,W_-)^\top,$ the sub-vector $W^{nc}:=(W_{+,3},W_{+,4},W_{-,3},W_{-,4})^\top.$ These are the components whose trace can be controlled on the boundary $\{x_2=0\}.$  We always assume that $\lambda\geq\lambda_0$ and $K\leq K_0,$ where $\lambda_0$ and $K_0$ are given in Theorem \ref{wellposed2}.
Then, we can rewrite \eqref{estimate2}  as
\begin{equation}\label{estimate3}
\sqrt{\lambda}||W||_{L^2_{\lambda}(\Omega_T)}+||W^{nc}|_{x_2=0}||_{L^2_{\lambda}(\omega_T)}+||\psi||_{H^1_{\lambda}(\omega_T)}\lesssim\lambda^{-\frac{3}{2}}||F_{\pm}||_{L^2(H^1_{\lambda}(\omega_T))}+\lambda^{-1}||g||_{H^1_{\lambda}(\omega_T)}.
\end{equation}
From the linear analysis \cite{Ruan2016}, we know that there is at least one order of loss of tangential derivative, therefore, we need to use the Nash-Moser technique in the nonlinear analysis.
To this end,  we strengthen the smallness assumption in \eqref{perturbation2} on the perturbed background states as
\begin{equation}\label{per}
[(\dot{U},\dot{\Phi})]_{10,\lambda,T}\leq K
\end{equation}
in the weighted anisotropic Sobolev space $H^{s,\lambda}_{\ast}(\Omega_T)$ with $s=10$,
where we write $\dot{U},\dot{\Phi}$, instead of $\dot{U}_{r,l},\dot{\Phi}_{r,l}$ for simplicity.
We note that  \eqref{perturbation2} implies \eqref{per} up to zero extension in time.

\subsection{A \textit{priori} estimate of tangential derivatives}\label{tangential}
Let $s\in \mathbb{N},$ and $l\in[1,s]$ be a fixed integer.
Applying the tangential derivative $D^{\beta}=\partial^{\alpha_0}_t\partial^{\alpha_1}_1$ with $l=\alpha_0+\alpha_1$ to
system \eqref{new} yields the equation for $D^{\beta}W_{\pm}$ that
involves the linear terms of the derivatives: $D^{\beta-\beta'}\partial_t W_{\pm}$
 and $D^{\beta-\beta'}\partial_1 W_{\pm}$, with $|\beta'|=1.$ These terms cannot be treated simply as source terms, owning to the loss of derivatives in the energy estimate \eqref{estimate3}. To overcome such difficulty, we adopt the idea of \cite{Coulombel2008} and deal with a boundary value problem for all the tangential derivatives of order   $l,$ $i.e.,$ for $W^{(l)}:=\{\partial^{\alpha_0}_t\partial^{\alpha_1}_1, \, \alpha_0+\alpha_1=l\}.$ Such a problem satisfies the same regularity and stability properties as the original problem \eqref{new} and \eqref{boundary8}. We can find that $W^{(l)}$ obeys an energy estimate similar to \eqref{estimate3} with new source terms $\mathcal{F}^{(l)}$ and  $\mathcal{G}^{(l)}$. Then, we can use the Gagliardo-Nirenberg and Moser-type inequalities in \cite{Coulombel2008} to derive the following estimate for the tangential derivatives; see Appendix \ref{Appendix2} for the proof.
\begin{lemma}\label{tangentialestimate}
Let $T>0,s\in \mathbb{N}, s\geq1.$
Then, there exist two constants $C_s>0$ and $\lambda_s\geq1$ that do not depend on $T$, such that for all $\lambda\geq\lambda_s$,  if   $(W,\psi)\in H^{s+2}_{\lambda}(\Omega_T)\times H^{s+2}_{\lambda}(\omega_T)$ is a solution to \eqref{new} and \eqref{boundary8}, the following estimate holds:
\begin{equation}\label{estimate4}
\begin{split}
&\sqrt{\lambda}||W||_{L^2(H^{s}_{\lambda}(\omega_T))}+||W^{nc}|_{x_2=0}||_{{H^s_{\lambda}(\omega_T)}}+||\psi||_{{H^{s+1}_{\lambda}(\omega_T)}}\\
&\leq C_s\Big\{\lambda^{-1}||g||_{{H^{s+1}_{\lambda}(\omega_T)}}+\lambda^{-\frac{3}{2}}||F||_{L^2(H^{s+1}_{\lambda}(\omega_T))}\\
& \quad +\lambda^{-\frac{3}{2}}||W||_{W^{1,tan}(\Omega_T)}||(\dot{U},\nabla\dot{\Phi})||_{H^{s+2}_{\lambda}(\Omega_T)}\\
& \quad +\lambda^{-1}(||W^{nc}|_{x_2=0}||_{L^{\infty}(\omega_T)}+||\psi||_{W^{1,\infty}(\omega_T)})||(\dot{U},\partial_2\dot{U},\nabla\dot{\Phi})|_{x_2=0}||_{H^{s+1}_{\lambda}(\omega_T)}\Big\}.
\end{split}
\end{equation}
\end{lemma}

\subsection{Weighted normal derivatives}\label{weighted}
In order to obtain the estimate on the normal derivatives in the anisotropic Sobolev spaces,
we first need to estimate $[\partial_2W^{nc}]_{s-1,\lambda,T}.$  We rewrite \eqref{new} as
$$I_2\partial_2W_{\pm}=F_{\pm}-A^{r,l}_0\partial_tW_{\pm}-A^{r,l}_1\partial_1W_{\pm}-A^{r,l}_0C^{r,l}W_{\pm}.$$
Dropping the superscripts $r,l$ and subscripts $\pm$ for simplicity of notations and noting that $I_2=
diag\{0,0,1,1\},$ we have
$$[\partial_2W^{nc}]_{s-1,\lambda,T}\leq C\{[F]_{s-1,\lambda,T}+[A_0\partial_tW]_{s-1,\lambda,T}+[A_1\partial_1W]_{s-1,\lambda,T}+[A_0CW]_{s-1,\lambda,T}\}.$$
Then, by  Theorems \ref{product2} and   \ref{product3},
$$[A_0\partial_tW]_{s-1,\lambda,T}\leq[A_0W]_{s,\lambda,T}\leq C\{||A_0||_{W^{1,tan}}[W]_{s,\lambda,T}+[A_0]_{s,\lambda,T}||W||_{W^{1,tan}}\},$$
$$[A_1\partial_1W]_{s-1,\lambda,T}\leq[A_1W]_{s,\lambda,T}\leq C\{||A_1||_{W^{1,tan}}[W]_{s,\lambda,T}+[A_1]_{s,\lambda,T}||W||_{W^{1,tan}}\},$$
$$[A_0CW]_{s-1,\lambda,T}\leq C\{||A_0C||_{W^{1,tan}}[W]_{s-1,\lambda,T}+[A_0C]_{s-1,\lambda,T}||W||_{W^{1,tan}}\}.$$
We know that $A^{r,l}_0$ and $A^{r,l}_0$ are $C^{\infty}$ functions of $(\dot{U}^{r,l},\nabla\dot{\Phi}^{r,l})$ and $C^{r,l}$ are $C^{\infty}$ functions of $(\dot{U}^{r,l},\nabla\dot{U}^{r,l},\nabla\dot{\Phi}^{r,l},\nabla^{tan}\nabla\dot{\Phi}^{r,l} )$ which vanish at the origin. By the assumption \eqref{per}, we have
$$[A_0\partial_tW]_{s-1,\lambda,T}\leq C\{[W]_{s,\lambda,T}+[(\dot{U},\nabla\dot{\Phi})]_{s,\lambda,T}||W||_{W^{1,tan}}\},$$
$$[A_1\partial_1W]_{s-1,\lambda,T}\leq C\{[W]_{s,\lambda,T}+[(\dot{U},\nabla\dot{\Phi})]_{s,\lambda,T}||W||_{W^{1,tan}}\},$$
$$[A_0CW]_{s-1,\lambda,T}\leq C\{[W]_{s-1,\lambda,T}+([\dot{U}]_{s+1,\lambda,T}+[\nabla\dot{\Phi}]_{s,\lambda,T})||W||_{W^{1,tan}}\}.$$
Thus  we obtain
\begin{equation}\label{Wnc}
[\partial_2W^{nc}]_{s-1,\lambda,T}\leq C(K)\{[F]_{s-1,\lambda,T}+[W]_{s,\lambda,T}+[(\dot{U},\dot{\Phi})]_{s+2,\lambda,T}||W||_{W^{1,tan}}\}.
\end{equation}
 Then, we want to estimate the $s$th order derivatives in the anisotropic Sobolev spaces with the weight on $x_2$ derivative. Multiplying $A^{-1}_0$ to \eqref{new} and rewriting this equation in terms of the new variable $\tilde{W}:=e^{-\lambda t}W,$ we have
\begin{equation}\label{new2}
\lambda\tilde{W}+\partial_t\tilde{W}+A^{-1}_0A_1\partial_1\tilde{W}+A^{-1}_0I_2\partial_2\tilde{W}+C\tilde{W}=A^{-1}_0\tilde{F},
\end{equation}
with $\tilde{F}=e^{-\lambda t}F.$ Now, we consider $D^{\beta}$ with $\beta=(\alpha_0,\alpha_1,\alpha_2,k),$ where we require $\alpha_2\geq1$ and $1\leq\langle \beta \rangle\leq s.$ We differentiate \eqref{new2} by $D^{\beta}.$
Then, multiplying $D^{\beta}\tilde{W}$ and integrating the equation, we obtain
 \begin{equation}\label{ddddd}
\begin{split}
&\lambda\langle D^{\beta}\tilde{W},D^{\beta}\tilde{W}\rangle+\langle D^{\beta}\tilde{W},\partial_tD^{\beta}\tilde{W}\rangle+\langle D^{\beta}\tilde{W},A^{-1}_0A_1\partial_1D^{\beta}\tilde{W}\rangle+\langle D^{\beta}\tilde{W},A^{-1}_0I_2\partial_2D^{\beta}\tilde{W}\rangle\\
&+\langle D^{\beta}\tilde{W},D^{\beta}(C\tilde{W})\rangle+\langle D^{\beta}\tilde{W},[D^{\beta},A^{-1}_0A_1]\partial_1\tilde{W}\rangle+\langle D^{\beta}\tilde{W},[D^{\beta},A^{-1}_0I_2]\partial_2\tilde{W}\rangle\\
&-\langle D^{\beta}\tilde{W},\alpha_2\sigma'A^{-1}_0I_2\partial^{\alpha_0}_t\partial^{\alpha_1}_1(\sigma\partial_2)^{\alpha_2-1}\partial^{k+1}_2\tilde{W}\rangle=\langle D^{\beta}\tilde{W},D^{\beta}\tilde{F}\rangle,
\end{split}
\end{equation}
where we denote the commutator $[a,b]c:=a(bc)-b(ac)$. For the terms in \eqref{ddddd}, we have
$$\lambda\langle D^{\beta}\tilde{W},D^{\beta}\tilde{W}\rangle=\lambda||D^{\beta}\tilde{W}||^2_{L^2(\Omega_T)},$$
$$\langle D^{\beta}\tilde{W},\partial_tD^{\beta}\tilde{W}\rangle=\frac{1}{2}\int^T_{-\infty}\int_{\R^2_+}\partial_t((D^{\beta}\tilde{W})^TD^{\beta}\tilde{W})dxdt=\frac{1}{2}||D^{\beta}\tilde{W}(T)||^2_{\R^2_+},$$
 \begin{equation}\nonumber
\begin{split}
&\langle D^{\beta}\tilde{W},A^{-1}_0A_1\partial_1D^{\beta}\tilde{W}\rangle=\frac{1}{2}\int^T_{-\infty}\int_{\R^2_+}\partial_1\left((D^{\beta}\tilde{W})^TA^{-1}_0A_1D^{\beta}\tilde{W}\right)dxdt\\
&\quad\quad\qquad\qquad\qquad\qquad\quad -\frac{1}{2}\langle D^{\beta}\tilde{W},\partial_1(A^{-1}_0A_1)D^{\beta}\tilde{W}\rangle\leq C(K)||D^{\beta}\tilde{W}||^2_{L^2(\Omega_T)},\\
\end{split}
\end{equation}
 \begin{equation}\nonumber
\begin{split}
&\langle D^{\beta}\tilde{W},D^{\beta}(C\tilde{W})\rangle\leq||D^{\beta}\tilde{W}||^2_{L^2(\Omega_T)}+[CW]^2_{s,\lambda,T}\\
&\leq C(K)\{[W]^2_{s,\lambda,T}+||W||^2_{W^{1,tan}}(1+[(\dot{U},\nabla\dot{\Phi})]^2_{s+2,\lambda,T})\},
\end{split}
\end{equation}
$$\langle D^{\beta}\tilde{W},D^{\beta}\tilde{F}\rangle\leq||D^{\beta}\tilde{W}||^2_{L^2(\Omega_T)}+||D^{\beta}\tilde{F}||^2_{L^2(\Omega_T)}.$$
Since $\alpha_2\geq1,$ we have $D^{\beta}\tilde{W}|_{x_2=0}=0.$  Thus we obtain
\begin{equation}\nonumber
\begin{split}
 \langle D^{\beta}\tilde{W},A^{-1}_0I_2\partial_2D^{\beta}\tilde{W}\rangle=&\frac{1}{2}\int^T_{-\infty}\int_{\R^2_+}\partial_2\left((D^{\beta}\tilde{W})^TA^{-1}_0I_2D^{\beta}\tilde{W}\right)dxdt\\
&-\frac{1}{2}\langle D^{\beta}\tilde{W},\partial_2(A^{-1}_0I_2)D^{\beta}\tilde{W}\rangle\leq C(K)||D^{\beta}\tilde{W}||^2_{L^2(\Omega_T)}.
\end{split}
\end{equation}
For the last term on the left hand side of \eqref{ddddd}, we have
\begin{equation}\nonumber
\begin{split}
&\langle D^{\beta}\tilde{W},\alpha_2\sigma'A^{-1}_0I_2\partial^{\alpha_0}_t\partial^{\alpha_1}_1(\sigma\partial_2)^{\alpha_2-1}\partial^{k+1}_2\tilde{W}\rangle\\
&\leq||D^{\beta}\tilde{W}||^2_{L^2(\Omega_T)}+C(K)||\partial^{\alpha_0}_t\partial^{\alpha_1}_1(\sigma\partial_2)^{\alpha_2-1}\partial^{k+1}_2\tilde{W}^{nc}||^2_{L^2(\Omega_T)}\\
&\leq ||D^{\beta}\tilde{W}||^2_{L^2(\Omega_T)}+C(K)[\partial_2W^{nc}]^2_{s-1,\lambda,T},
\end{split}
\end{equation}
where $[\partial_2W^{nc}]_{s-1,\lambda,T}$ can be estimated by \eqref{Wnc}. Next, we turn to estimate the terms involving commutators as the following:
\begin{equation}\nonumber
\begin{split}
&\langle D^{\beta}\tilde{W},[D^{\beta},A^{-1}_0A_1]\partial_1\tilde{W}\rangle\\
&\lesssim||D^{\beta}\tilde{W}||^2_{L^2(\Omega_T)}+\sum_{|\alpha'|=1}||\partial^{\alpha'}_{\ast}(A^{-1}_0A_1)D^{\beta-\alpha'}\partial_1 \tilde{W}||^2_{L^2(\Omega_T)}+\sum_{\langle\beta'\rangle=2}[D^{\beta'}(A^{-1}_0A_1)\partial_1 \tilde{W}]^2_{s-2,T}\\
&\leq C(K)\{[W]^2_{s,\lambda,T}+||W||^2_{W^{1,tan}}(1+[(\dot{U},\nabla^{tan}\dot{\Phi})]^2_{s+1,\lambda,T}\},
\end{split}
\end{equation}
and
\begin{equation}\nonumber
\begin{split}
&\langle D^{\beta}\tilde{W},[D^{\beta},A^{-1}_0I_2]\partial_2\tilde{W}\rangle\\
&\lesssim||D^{\beta}\tilde{W}||^2_{L^2(\Omega_T)}+\sum_{|\alpha'|=1}||\partial^{\alpha'}_{\ast}(A^{-1}_0I_2)D^{\beta-\alpha'}\partial_2 \tilde{W}||^2_{L^2(\Omega_T)}+\sum_{\langle\beta'\rangle=2}[D^{\beta'}(A^{-1}_0I_2)\partial_2 \tilde{W}]^2_{s-2,T}\\
&\leq C(K)\{[W]^2_{s,\lambda,T}+[\partial_2W^{nc}]^2_{s-1,\lambda,T}+||W||^2_{W^{1,tan}}(1+[(\dot{U},\nabla\dot{\Phi})]^2_{s+1,\lambda,T})\}.
\end{split}
\end{equation}
Hence, we obtain
\begin{equation}\label{weightes}
\begin{split}
&\lambda||D^{\beta}\tilde{W}||^2_{L^2(\Omega_T)}+\frac{1}{2}||D^{\beta}\tilde{W}(T)||^2_{L^2(\R^2_+)}\\
&\leq C(K)\{[W]^2_{s,\lambda,T}+||W||^2_{W^{1,tan}}(1+[(\dot{U},\nabla\dot{\Phi})]^2_{s+2,\lambda,T})+[F]^2_{s,\lambda,T}\}.
\end{split}
\end{equation}
\subsection{Unweighted normal derivatives}\label{unweighted}
In this subsection, we estimate the $s$th order unweighted derivatives in the anisotropic Sobolev spaces on $x_2$ derivative. We consider the case when $\alpha_2=0,k\geq1$ in $\beta.$ Then, we apply $D^{\beta}=\partial^{\alpha}_{tan}\partial^{k}_2=\partial^{\alpha_0}_t\partial_1^{\alpha_1}\partial^k_2$ to \eqref{new2} and obtain
\begin{equation}
\begin{split}
&\lambda\langle D^{\beta}\tilde{W},D^{\beta}\tilde{W}\rangle+\langle D^{\beta}\tilde{W},\partial_tD^{\beta}\tilde{W}\rangle+\langle D^{\beta}\tilde{W},A^{-1}_0A_1\partial_1D^{\beta}\tilde{W}\rangle+\langle D^{\beta}\tilde{W},A^{-1}_0I_2\partial_2D^{\beta}\tilde{W}\rangle\\
&\qquad+\langle D^{\beta}\tilde{W},D^{\beta}(C\tilde{W})\rangle+\langle D^{\beta}\tilde{W},[D^{\beta},A^{-1}_0A_1]\partial_1\tilde{W}\rangle+\langle D^{\beta}\tilde{W},[D^{\beta},A^{-1}_0I_2]\partial_2\tilde{W}\rangle\\
&=\langle D^{\beta}\tilde{W},D^{\beta}\tilde{F}\rangle.
\end{split}
\end{equation}
Now, we only focus on the fourth term：
 \begin{equation}
\begin{split}
&\langle D^{\beta}\tilde{W},A^{-1}_0I_2\partial_2D^{\beta}\tilde{W}\rangle\\
&=\frac{1}{2}\int^T_{-\infty}\int_{\R^2_+}\partial_2\left((D^{\beta}\tilde{W})^TA^{-1}_0I_2D^{\beta}\tilde{W}\right)dxdt-\frac{1}{2}\langle D^{\beta}\tilde{W},\partial_2(A^{-1}_0I_2)D^{\beta}\tilde{W}\rangle\\
&=-\frac{1}{2}\int_{\omega_T}(D^{\beta}\tilde{W})^TA^{-1}_0I_2D^{\beta}\tilde{W}|_{x_2=0}dx_1dt-\frac{1}{2}\langle D^{\beta}\tilde{W},\partial_2(A^{-1}_0I_2)D^{\beta}\tilde{W}\rangle.\\
\end{split}
\end{equation}
We have already obtained that
$$\frac{1}{2}\langle D^{\beta}\tilde{W},\partial_2(A^{-1}_0I_2)D^{\beta}\tilde{W}\rangle\leq C(K)||D^{\beta}\tilde{W}||^2_{L^2(\Omega_T)}.$$
Then, for the integral on $\omega_T,$ one has
 \begin{equation}
\begin{split}
&\frac{1}{2}\int_{\omega_T} (D^{\beta}\tilde{W})^TA^{-1}_0I_2D^{\beta}\tilde{W}|_{x_2=0} dx_1dt\leq C(K)||D^{\beta}\tilde{W}^{nc}|_{x_2=0}||^2_{L^2(\omega_T)}\\
&\leq C(K)||\partial^{\alpha}_{tan}\partial^{k-1}_2(\tilde{F}-\lambda A_0 \tilde{W}-A_0\partial_t\tilde{W}-A_1\partial_1\tilde{W}-A_0C\tilde{W})|_{x_2=0}||^2_{L^2(\omega_T)}\\
&\leq C(K)\int_0^{\infty}\int_{\omega_ T}\partial_2|\partial^{\alpha}_{tan}\partial^{k-1}_2(\tilde{F}-\lambda A_0\tilde{W}-A_0\partial_t\tilde{W}-A_1\partial_1\tilde{W}-A_0C\tilde{W})|^2dx_1dtdx_2.
\end{split}
\end{equation}
It is noted that there are at most $s+1$th order derivatives of anisotropic Sobolev spaces in the above integral. Therefore, using previous argument in Section \ref{weighted}, we can obtain
\begin{equation}\label{unweightes}
\begin{split}
&\lambda||D^{\beta}\tilde{W}||^2_{L^2(\Omega_T)}+\frac{1}{2}||D^{\beta}\tilde{W}(T)||^2_{L^2(\R^2_+)}\\
&\leq C(K)\{[W]^2_{s,\lambda,T}+||W||^2_{W^{1,tan}}(1+[(\dot{U},\nabla\dot{\Phi})]^2_{s+2,\lambda,T})+[F]^2_{s,\lambda,T}\}.
\end{split}
\end{equation}
\subsection{The a \textit{priori} tame estimate}\label{tame2}
From the above estimates \eqref{estimate4}, \eqref{weightes} and \eqref{unweightes}, we can prove the following tame estimate:
\begin{theorem}\label{tame3}
Let $T>0,s\in \mathbb{N}.$   Assume that the background solution \eqref{rectilinear2} satisfies \eqref{supersonic} and the perturbation $(\dot{U}, \dot{\Phi})$ satisfies \eqref{compact}, \eqref{perturbation2}.
Then there exists a constant $ K_0>0,$ that does not depend on $s$ and $T,$ and there exist two constants $C_s>0$ and $\lambda_s\geq1$ that depend on $s$ but not on $T$, such that if $K\leq K_0,\lambda\geq \lambda_s$,  and   $(\dot{V}_{\pm},\psi)\in H^{s}_{\lambda}(\Omega_T)\times H^{s+1}_{\lambda}(\omega_T)$ is a solution to  \eqref{effective2} and \eqref{boundary7},   the following estimate holds:
\begin{equation}\label{tameestimate}
\begin{split}
&\sqrt{\lambda}[\dot{V}]_{s,\lambda,T}+||\dot{V}^{nc}|_{x_2=0}||_{H^{s}_{\lambda}(\omega_T)}+||\psi||_{H^{s+1}_{\lambda}(\omega_T)}\\
&\leq C_s\{[f]_{s+1,\lambda,T}+||g||_{H^{s+1}_{\lambda}(\omega_T)}+([f]_{6,\lambda,T}+||g||_{H^6_{\lambda}(\omega_T)})[(\dot{U},\dot{\Phi})]_{s+4,\lambda,T}\}.\\
\end{split}
\end{equation}
\end{theorem}
\begin{proof}
Combining \eqref{estimate4}, \eqref{weightes}, and \eqref{unweightes} yields
\begin{equation}\nonumber
\begin{split}
&\sqrt{\lambda}[W]_{s,\lambda,T}+||W^{nc}|_{x_2=0}||_{H^s_{\lambda}(\omega_T)}+||\psi||_{H^{s+1}_{\lambda}(\omega_T)}\\
&\leq C(K)\Big\{[F]_{s+1,\lambda,T}+||g||_{H^{s+1}_{\lambda}(\omega_T)}+||W||_{W^{1,tan}}[(\dot{U},\nabla\dot{\Phi})]_{s+2,\lambda,T}\\
&\quad +\frac{1}{\lambda}(||W^{nc}|_{x_2=0}||_{L^{\infty}(\omega_T)}+||\psi||_{W^{1,\infty}(\omega_T)})||(\dot{U},\partial_2\dot{U},\nabla\dot{\Phi})|_{x_2=0}||_{H^{s+1}_{\lambda}(\omega_T)}\Big\}.
\end{split}
\end{equation}
To absorb $||W||_{W^{1,tan}},||W^{nc}|_{x_2=0}||_{L^{\infty}(\omega_T)}$ and $||\psi||_{W^{1,\infty}(\omega_T)}$ on the right side, we take $s=5$ and obtain
 \begin{equation}\nonumber
\begin{split}
&\sqrt{\lambda}[W]_{5,\lambda,T}+||W^{nc}|_{x_2=0}||_{H^5_{\lambda}(\omega_T)}+||\psi||_{H^{6}_{\lambda}(\omega_T)}\\
&\leq C(K)\{[F]_{6,\lambda,T}+||g||_{H^{6}_{\lambda}(\omega_T)}+||W||_{W^{1,tan}}[(\dot{U},\nabla\dot{\Phi})]_{7,\lambda,T}\\
&\quad +\frac{1}{\lambda}(||W^{nc}|_{x_2=0}||_{L^{\infty}(\omega_T)}+||\psi||_{W^{1,\infty}(\omega_T)})||(\dot{U},\partial_2\dot{U},\nabla\dot{\Phi})|_{x_2=0}||_{H^{6}_{\lambda}(\omega_T)}\}.
\end{split}
\end{equation}
Since $||W||_{W^{1,tan}}\leq [W]_{5,\lambda,T}$ and we assume 
$$||(\dot{U},\partial_2\dot{U},\nabla\dot{\Phi})|_{x_2=0}||_{H^6_{\lambda}(\omega_T)}+[(\dot{U},\nabla\dot{\Phi})]_{7,\lambda,T}\leq[(\dot{U},\nabla\dot{\Phi})]_{10,\lambda,T}\leq K_0,$$  we have
\begin{equation}\nonumber
\begin{split}
&\sqrt{\lambda}[W]_{5,\lambda,T}+||W^{nc}|_{x_2=0}||_{H^5_{\lambda}(\omega_T)}+||\psi||_{H^{6}_{\lambda}(\omega_T)}\leq C(K)\{[F]_{6,\lambda,T}+||g||_{H^{6}_{\lambda}(\omega_T)}\}.\\
\end{split}
\end{equation}
Hence, one has
\begin{equation}\nonumber
\begin{split}
&\sqrt{\lambda}[W]_{s,\lambda,T}+||W^{nc}|_{x_2=0}||_{H^s_{\lambda}(\omega_T)}+||\psi||_{H^{s+1}_{\lambda}(\omega_T)}\leq C(K)\{[F]_{s+1,\lambda,T}+||g||_{H^{s+1}_{\lambda}(\omega_T)}\\
&+([F]_{6,\lambda,T}+||g||_{H^{6}_{\lambda}(\omega_T)})([(\dot{U},\nabla\dot{\Phi})]_{s+2,\lambda,T}+||(\dot{U},\partial_2\dot{U},\nabla\dot{\Phi})|_{x_2=0}||_{H^{s+1}_{\lambda}(\omega_T)})\}\\
&\leq C(K)\{[F]_{s+1,\lambda,T}+||g||_{H^{s+1}_{\lambda}(\omega_T)}+([F]_{6,\lambda,T}+||g||_{H^{6}_{\lambda}(\omega_T)})[(\dot{U},\nabla\dot{\Phi})]_{s+4,\lambda,T})\},\\
\end{split}
\end{equation}
for $s\geq5.$
From \eqref{new}, we know $W_{\pm}:=T^{-1}\dot{V}_{\pm}$ and $F_{\pm}=A^{r,l}_0T^{-1}f_{\pm}.$
Then, 
\begin{equation}\nonumber
\begin{split}
[\dot{V}]_{s,\lambda,T}&\leq C(K)\{||T||_{W^{1,tan}}[W]_{s,\lambda,T}+[T]_{s,\lambda,T}||W||_{W^{1,tan}}\}\\
&\leq C(K)\{[W]_{s,\lambda,T}+||W||_{W^{1,tan}}[(\dot{U},\nabla^{tan}\dot{\Phi})]_{s,\lambda,T}\},\\
\end{split}
\end{equation}
\begin{equation*}
\begin{split}
&||\dot{V}^{nc}|_{x_2=0}||_{H^s_{\lambda}(\omega_T)}\\
&\leq C(K)\{||W^{nc}|_{x_2=0}||_{H^{s}_{\lambda}(\omega_T)}+||W^{nc}|_{x_2=0}||_{L^{\infty}(\omega_T)}||(\dot{U},\nabla^{tan}\dot{\Phi})|_{x_2=0}||_{H^s_{\lambda}(\omega_T)}\},
\end{split}
\end{equation*}
$$[F]_{s+1,\lambda,T}\leq C(K)\{[f]_{s+1,\lambda,T}+||f||_{W^{1,tan}}[(\dot{U},\nabla{\dot{\Phi}})]_{s+1,\lambda,T}\}.$$
Combining all the above estimates, we complete the proof of \eqref{tameestimate}.
\end{proof}

\section{Compatibility Conditions for the Initial Data}\label{compatibility}

In this section, we shall construct the approximate solutions as in \cite{Coulombel2008}.

\subsection{The compatibility conditions}\label{compatibility2}

Let $\mu\in \mathbb{N}.$
Given the initial data  $$U^{\pm}_0=(m^{\pm}_0,n^{\pm}_0,v^{\pm}_0,u^{\pm}_0)^\top,$$ such that $U^{\pm}_0=\bar{U}^{\pm}+\dot{U}^{\pm}_0$ with $\dot{U}^{\pm}_0\in H^{2\mu+1}_{\ast}(\R^2_+),$ and $\varphi_0\in H^{2\mu+2}(\R),$ we need to prescribe the necessary compatibility conditions for the existence of a smooth solution $(U^{\pm},\Phi^{\pm})$ to \eqref{equation2}-\eqref{eikonal}.
We also assume without loss of generality that the initial data $\dot{U}^{\pm}_0$ and $\varphi_0$ have compact support:
\begin{equation}\label{compact2}
\supp \dot{U}^{\pm}_0\subseteq \{x_2\geq0,\sqrt{x^2_1+x^2_2}\leq1\}, \quad \supp \varphi_0\subseteq [-1,1].
\end{equation}
Let us extend $\varphi_0$ from $\R$ to $\R^2_+$ by constructing $\dot{\Phi}^+_0=\dot{\Phi}^-_0\in H^{2\mu+3}_{\ast}(\R^2_+),$ satisfying $\dot{\Phi}^{\pm}_0|_{x_2=0}=\varphi_0$ and the estimate:
\begin{equation}\label{phi2}
[\dot{\Phi}^{\pm}_0]_{2\mu+3,\ast}\leq C||\varphi_{0}||_{H^{2\mu+2}(\R)}.
\end{equation}
Up to multiplying $\dot{\Phi}^{\pm}_0$ by a $C^{\infty}$ function with compact support (whose choice only depends on support of $\varphi_0$), we may assume that $\dot{\Phi}^{\pm}_0$ satisfies
\begin{equation}\label{phi3}
\supp \dot{\Phi}^{\pm}_0 \subseteq  \Big\{x_2\geq0,\sqrt{x^2_1+x^2_2}\leq1+\frac{\lambda_{max}}{2}T\Big\}.
\end{equation}
We define $\Phi^{\pm}_0:=\pm x_2+\dot{\Phi}^{\pm}_0.$ Taking $\mu\geq2,$ for $\varphi_0$ small enough in $H^{2\mu+2}(\R)$, the Sobolev embedding theorem yields that:
\begin{equation}\label{bound2}
\pm \partial_2\Phi^{\pm}_0\geq\frac{7}{8}, \text{ for all } x\in \R^2_+.
\end{equation}
For the eikonal equation \eqref{eikonal}, we prescribe the initial data:
\begin{equation}\label{initialdata}
\Phi^{\pm}|_{t=0}=\Phi^{\pm}_0
\end{equation}
in the space domain $\R^2_+.$ We define the traces of the $l$th  order time derivatives on ${t=0}$ by
$$\dot{U}^{\pm}_l:=\partial^{l}_t\dot{U}^{\pm}|_{t=0},\quad \dot{\Phi}^{\pm}_l:=\partial^{l}_t\dot{\Phi}^{\pm}|_{t=0},\quad l\in \mathbb{N }.$$
To introduce the compatibility conditions, we need to determine the traces $\dot{U}^{\pm}_l$ and $\dot{\Phi}^{\pm}_l$ in terms of the initial data $\dot{U}^{\pm}_0$ and $\dot{\Phi}^{\pm}_0$ through \eqref{equation2} and \eqref{eikonal}.
 As in \cite{Coulombel2008}, we set $\mathcal{W}^{\pm}:=(\dot{U},\nabla_x\dot{U},\nabla_x\dot{\Phi})\in \R^{14}$ and rewrite \eqref{equation2} and \eqref{eikonal} as
\begin{equation}\label{taylor}
\partial_t \dot{U}^{\pm}=F_1(\mathcal{W}^{\pm}),\quad \partial_t \dot{\Phi}^{\pm}=F_2(\mathcal{W}^{\pm}),
\end{equation}
where $F_1$ and $F_2$ are suitable $C^\infty$ functions that vanish at the origin. After applying the operator $\partial^{l}_t$ to \eqref{taylor}, we take the traces at $t=0.$ Hence, we can apply the generalized Fa\`a  di Bruno's formula \cite{Mishkov} to derive
\begin{equation}\label{formula}
\begin{split}
&\dot{U}^{\pm}_{l+1}=\sum_{\alpha_i\in\mathbb{N}^{14},|\alpha_1|+\cdots l|\alpha_l|=l}D^{\alpha_1+\cdots\alpha_l}F_1(\mathcal{W}^{\pm}_0)\prod^{l}_{i=1}\frac{l!}{\alpha_i!}(\frac{\mathcal{W}^{\pm}_i}{i!})^{\alpha_i},\\
&\dot{\Phi}^{\pm}_{l+1}=\sum_{\alpha_i\in\mathbb{N}^{14},|\alpha_1|+\cdots l|\alpha_l|=l}D^{\alpha_1+\cdots\alpha_l}F_2(\mathcal{W}^{\pm}_0)\prod^{l}_{i=1}\frac{l!}{\alpha_i!}(\frac{\mathcal{W}^{\pm}_i}{i!})^{\alpha_i},\\
\end{split}
\end{equation}
where $\mathcal{W}^{\pm}_i$ denotes the traces $(\dot{U}_i,\nabla_x\dot{U}_i,\nabla_x\dot{\Phi}_i)$ at $t=0.$ From \eqref{formula}, we can determine $(\dot{U}^{\pm}_l,\dot{\Phi}^{\pm}_l)_{l\geq0}$ inductively as functions of the initial data $(\dot{U}^{\pm}_0$,$\dot{\Phi}^{\pm}_0).$
Furthermore, we have the following lemma  from \cite{Rauch1974}.
\begin{lemma}\label{compatibility3}
The equations \eqref{taylor} and \eqref{formula} determine $\dot{U}^{\pm}_l\in H^{2(\mu-l)+1}_{\ast}(\R^2_+)$ for $l=1,\cdots,\mu$ and $\dot{\Phi}^{\pm}_l\in H^{2(\mu-l)+3}_{\ast}(\R^2_+)$ for $l=1,\cdots,\mu+1.$ Moreover, these functions satisfy
$$\supp \dot{\Phi}^{\pm}_{l}\subseteq \{x_2\geq0, \sqrt{x^2_1+x^2_2}\leq 1+\lambda_{max}T\},\quad
\supp \dot{U}^{\pm}_{l}\subseteq \{x_2\geq0, \sqrt{x^2_1+x^2_2}\leq1\},$$
and there exists a constant $C>0$  depending  only   on $\mu$, such that,
\begin{equation}\label{compatibilityestimate}
\sum^{\mu}_{l=1}[\dot{U}^{\pm}_{l}]_{2(\mu-l)+1,\ast}+\sum^{\mu+1}_{l=1}[\dot{\Phi}^{\pm}_l]_{2(\mu-l)+3,\ast}\leq C([\dot{U}^{\pm}_0]_{2\mu+1,\ast}+||\varphi_0||_{H^{2\mu+2}(\R)}).
\end{equation}
\end{lemma}
To construct a smooth approximate solution, we impose certain assumptions on traces of $\dot{U}^{\pm}_l$ and $\dot{\Phi}^{\pm}_l.$ Now we   state the compatibility conditions of initial data.
\begin{Definition}\label{c}
For $\mu\in \mathbb{N},$ $\mu\geq2,$ let $U^{\pm}_0=(m^{\pm}_0,n^{\pm}_0,v^{\pm}_0,u^{\pm}_0)^T$ have the form $U^{\pm}_0=\bar{U}^{\pm}+\dot{U}^{\pm}_0$ with $\dot{U}^{\pm}_0\in H^{2\mu+1}_{\ast}(\R^2_+),$ $\varphi_0\in H^{2\mu+2}(\R)$  satisfying \eqref{compact2}. Consider the functions $\Phi^{\pm}_0=\pm x_2+\dot{\Phi}^{\pm}_0$ that satisfy \eqref{phi2}-\eqref{bound2}, where $\varphi_0$ is sufficiently small.
The initial data $(\dot{U}^{\pm}_0,\varphi_0)$ is said to be compatible up to order $\mu$ if the traces of the functions $\dot{U}^{\pm}_1,\cdots\dot{U}^{\pm}_\mu,\dot{\Phi}^{\pm}_1,\cdots \dot{\Phi}^{\pm}_{\mu+1}$ satisfy the following:
\begin{equation}\label{b}
\begin{split}
& \partial^j_2(\dot{\Phi}^+_l-\dot{\Phi}^-_l)|_{x_2=0}=0, ~~~~~~~~\text{ for } l=0,\cdots, \mu \text{ and } j=0,\cdots,\mu-l,\\
& \partial^j_2(\dot{m}^+_l-\dot{m}^-_l)|_{x_2=0}=0, ~~~~~~~~~~~~~~\text{ for } l=0,\cdots, \mu-1 \text{ and } j=0,\cdots,\mu-1-l,\\
& \partial^j_2(\dot{n}^+_l-\dot{n}^-_l)|_{x_2=0}=0, ~~~~~~~~~~~~~~\text{ for } l=0,\cdots, \mu-1 \text{ and } j=0,\cdots,\mu-1-l,\\
\end{split}
\end{equation}
and
\begin{equation}\label{trace2}
\begin{split}
&\int_{\R^2_+}|\partial^{\mu+1-j}_2(\dot{\Phi}^+_j-\dot{\Phi}^-_j)|^2dx_1\frac{dx_2}{x_2}<+\infty, \text{\quad for } j=0,\cdots,\mu+1,\\
&\int_{\R^2_+}|\partial^{\mu-j}_2(\dot{m}^+_j-\dot{m}^-_j)|^2dx_1\frac{dx_2}{x_2}<+\infty, \text{\quad for } j=0,\cdots,\mu,\\
&\int_{\R^2_+}|\partial^{\mu-j}_2(\dot{n}^+_j-\dot{n}^-_j)|^2dx_1\frac{dx_2}{x_2}<+\infty, \text{\quad for } j=0,\cdots,\mu.\\
\end{split}
\end{equation}
\end{Definition}
\subsection{Construction of approximate solutions}\label{approximate}
 We shall take $\mu=\alpha+7$ in the previous paragraph and consider the compatibility of initial data $(\dot{U}^{\pm}_0,\varphi_0)$ in the sense of Definition \ref{c} in Theorem \ref{stability}. In particular, the initial data are compatible up to order $\alpha+7.$  We have the following result similar to \cite{Coulombel2008}. 
\begin{lemma}\label{approximate2}
If $\dot{U}^{\pm}_0$ and $\varphi_0$ are sufficiently small,  there exist some functions $U^{a\pm}$, $\Phi^{a\pm}$, $\varphi^{a},$ such that $U^{a\pm}-\bar{U}^{a\pm}=\dot{U}^{a\pm}\in H^{\mu+1}_{\ast}(\Omega_T), \Phi^{a\pm}\mp x_2=\dot{\Phi}^{a\pm}\in H^{\mu+2}_{\ast}(\Omega_T), \varphi^a\in H^{\mu+1}(\omega_T)$, and
\begin{equation}\label{dequation}
\partial_t\Phi^{a\pm}+v^{a\pm}\partial_1\Phi^{a\pm}-u^{a\pm}=0, \text {in $\Omega_T$},\\
\end{equation}
\begin{equation}\label{dequation2}
\qquad\qquad\qquad\partial^{j}_t \mathbb{L}(U^{a\pm},\nabla\Phi^{a\pm})|_{t=0}=0, \text{\quad\quad for } j=0,\cdots,\mu-1,\\
\end{equation}
\begin{equation}\label{dequation3}
\Phi^{a+}|_{x_2=0}=\Phi^{a-}|_{x_2=0}=\varphi^a, \text{\quad on } \omega_T,
\end{equation}
\begin{equation}\label{dequation4}
\mathbb{B}(U^{a\pm}|_{x_2=0},\varphi^a)=0, \text{\quad\quad\quad\quad on }\omega_T.
\end{equation}
\end{lemma}
Furthermore, we have
\begin{equation}\label{bound3}
\pm\partial_2\Phi^{a\pm}\geq\frac{3}{4},\; \forall (t,x)\in\Omega_T,
\end{equation}
\begin{equation}\label{estimatea}
[\dot{U}^{a\pm}]_{\mu+1,\ast,T}+[\dot{\Phi}^{a\pm}]_{\mu+2,\ast,T}+||\varphi^a||_{H^{\mu+1}(\omega_T)}\leq \varepsilon_0([\dot{U}^{\pm}_0]_{2\mu+1,\ast}+||\varphi_0||_{H^{2\mu+2}(\R)}),
\end{equation}
and the following compact support:
\begin{equation}\label{compact3}
\supp(\dot{U}^{a\pm},\dot{\Phi}^{a\pm})\subseteq \{t\in[-T,T],x_2\geq0, \sqrt{x^2_1+x^2_2}\leq 1+\lambda_{max}T\},
\end{equation}
\begin{equation}\label{compact4}
\supp \varphi^a\subseteq \{t\in[-T,T],|x_1|\leq 1+\lambda_{max}T\},
\end{equation}
where  we denote by $\varepsilon_0(\cdot)$  a function that tends to $0$, when $\cdot$ tends to $0$.

We write $U^a:=(U^{a+},U^{a-})^T, \Phi^a:=(\Phi^{a+},\Phi^{a-})^T$ and turn to reformulate the original problem into one with zero initial data by using the approximate solution $(U^a,\Phi^a)$. Set
\begin{equation}\label{f}
f^a:=\left\{ \begin{split}
\displaystyle &-\mathbb{L}(U^a,\nabla\Phi^a),\; t>0,\\
\displaystyle &~~~~~~~~0,\; t<0.\\
\end{split}
\right.
\end{equation}
From $\dot{U}^a,\nabla\dot{\Phi}^a\in H^{\mu+1}_{\ast}(\Omega)$ and \eqref{dequation2}, we have
$f^a\in H^{\mu-1}_{\ast}(\Omega).$ Using \eqref{compact3} and \eqref{compact4}, we get
\begin{equation}\label{compact5}
\supp f^a\subseteq \{t\in[0,T],x_2\geq0,\sqrt{x^2_1+x^2_2}\leq 1+\lambda_{max}T\}.
\end{equation}
From \eqref{estimatea}, we can obtain that
\begin{equation}\label{fa}
[f^a]_{\mu-1,\ast}\leq \varepsilon_0([\dot{U}^{\pm}_0]_{2\mu+1,\ast}+||\varphi_0||_{H^{2\mu+2}(\R)}).
\end{equation}
Given the approximate solution $(U^a,\Phi^a)$ of Lemma \ref{approximate2}  and $f^a$   defined in \eqref{f}, we see that $(U,\Phi)=(U^a,\Phi^a)+(V,\Psi)$ is a solution of the original problem on $\Omega_T$ of \eqref{equation2}-\eqref{eikonal}, if $V=(V^+,V^-)^T,\Psi=(\Psi^+,\Psi^-)^T$ satisfy the following problem:
\begin{equation}\label{system}
\left\{ \begin{split}
\displaystyle & \mathcal{L}(V,\Psi):=\mathbb{L}(U^a+V,\nabla(\Phi^a+\Psi))-\mathbb{L}(U^a,\nabla\Phi^a)=f^a, \text{\quad in } \Omega_T,\\
\displaystyle & \mathcal{E}(V,\Psi):=\partial_t\Psi+(v^a+v)\partial_1\Psi-u+v\partial_1\Phi^a=0, \text{\quad\quad\qquad in } \Omega_T,\\
\displaystyle & \mathcal{B}(V|_{x_2=0},\psi):=\mathbb{B}(U^a|_{x_2=0}+V|_{x_2=0},\varphi^a+\psi)=0,\text{\qquad\quad on } \omega_T,\\
\displaystyle & \Psi^+|_{x_2=0}=\Psi^-|_{x_2=0}=\psi, \text{\qquad\qquad\qquad\qquad\qquad\qquad\quad\quad on } \omega_T,\\
\displaystyle & (V,\Psi)=0, \text{\qquad\qquad\qquad\qquad\qquad\qquad\qquad\qquad\qquad\qquad for } t<0.
\end{split}
\right.
\end{equation}
The original nonlinear problem on $[0,T]\times\R^2_+$ can be reformulated as a  problem on $\Omega_T$ whose solutions vanish in the past.

\section{Nash-Moser Theorem}\label{nash}
In this section, we shall prove the local existence of solutions to \eqref{system} by a suitable iteration scheme of Nash-Moser type, following  \cite{Coulombel2008}. First, we introduce the smoothing operators $S_{\theta}$ and describe the iterative scheme for problem \eqref{system}. For more details, please refer to \cite{ChenG2008,Coulombel2008,Trakhinin2009}.
\begin{lemma}\label{smooth}
We can define a family of smoothing operators $\{S_{\theta}\}_{{\theta\geq1}}$ on the anisotropic Sobolev space $H^{s,\lambda}_{\ast}(\Omega_T)$,   vanishing in the past, such that
\begin{equation}\label{as1}
[S_{\theta}u]_{\beta,\lambda,T}\leq C\theta^{(\beta-\alpha)_+}[u]_{\alpha,\lambda,T}, \text{ for all } \alpha,\beta\geq0,
\end{equation}
\begin{equation}\label{as2}
[S_{\theta}u-u]_{\beta,\lambda,T}\leq C\theta^{\beta-\alpha}[u]_{\alpha,\lambda,T}, \text{ for all } 0\leq\beta\leq\alpha,
\end{equation}
\begin{equation}\label{as3}
[\frac{d}{d\theta}S_{\theta}u]_{\beta,\lambda,T}\leq C\theta^{\beta-\alpha-1}[u]_{\alpha,\lambda,T}, \text{ for all } \alpha,\beta\geq0,
\end{equation}
and
\begin{equation}\label{s4}
||(S_\theta u-S_{\theta}v)|_{x_2=0}||_{H^{\beta}_{\lambda}(\omega_T)}\leq C\theta^{(\beta+1-\alpha)_+}||(u-v)|_{x_2=0}||_{H^{\alpha}_{\lambda}(\omega_T)}, \text{ for all } \alpha,\beta\in[1,\mu],
\end{equation}
where $\alpha,\beta\in \mathbb{N},(\beta-\alpha)_+:=\max\{0,\beta-\alpha\}$ and $C>0$ is a constant depending only on $\mu.$ In particular, if $u=v$ on $\omega_T,$ then $S_{\theta}u=S_{\theta}v$ on $\omega_T.$
\end{lemma}
Now, we begin to formulate the Nash-Moser iteration scheme.
\subsection{Iteration Scheme}\label{iteration}
The scheme starts from $(V_0,\Psi_0,\psi_0)=(0,0,0),$  and $(V_i,\Psi_i,\psi_i)$ is given such that
\begin{equation}\label{i1}
(V_i,\Psi_i,\psi_i)|_{t<0}=0,\quad \Psi^+_{i}|_{x_2=0}=\Psi^-_{i}|_{x_2=0}=\psi_i.
\end{equation}
Let us consider
\begin{equation}\label{i2}
V_{i+1}=V_i+\delta V_i, \; \Psi_{i+1}=\Psi_i+\delta\Psi_i,\; \psi_{i+1}=\psi_i+\delta\psi_i,
\end{equation}
where the  differences shall be determined below.
First, we can obtain $(\delta \dot{V}_i,\delta\psi_i)$ by solving the effective linear problem:
\begin{eqnarray}\label{effective3}
\left\{ \begin{split}
\displaystyle &\mathbb{L}_e'(U^a+V_{i+\frac{1}{2}},\Phi^a+\Psi_{i+\frac{1}{2}})\delta \dot{V}_i=f_i, \text{\quad\quad\quad in } \Omega_T,\\
\displaystyle &\mathbb{B}_e'(U^a+V_{i+\frac{1}{2}},\Phi^a+\Psi_{i+\frac{1}{2}})(\delta\dot{ V}_i,\delta\psi_i)=g_i, \text{ on }\omega_T    ,\\
\displaystyle &(\delta \dot{V}_i,\delta\psi_i)=0, \text{\qquad\qquad\qquad\quad\qquad\qquad\quad for } t<0,\\
\end{split}
\right.
\end{eqnarray}
where operators $\mathbb{L}_e', \mathbb{B}_e'$ are defined in \eqref{effective2} and \eqref{boundary7},
\begin{equation}\label{goodunkown}
\delta \dot{V}_i:=\delta V_i-\frac{\partial_2(U^a+V_{i+\frac{1}{2}})}{\partial_2(\Phi^a+\Psi_{i+\frac{1}{2}})}\delta\Psi_i
\end{equation}
is the ``good unknown" and $(V_{i+\frac{1}{2}},\Psi_{i+\frac{1}{2}})$ is a smooth modified state such that $(U^a+V_{i+\frac{1}{2}},\Phi^a+\Psi_{i+\frac{1}{2}})$ satisfies \eqref{perturbation2}, \eqref{compact},  \eqref{boundary6}, \eqref{eikonal2}, and \eqref{bound}. The source terms $(f_i,g_i)$ will be defined through the accumulated errors at step $i$.
Let the error $\varepsilon^{i}_{1,2,3}$ be defined by
\begin{equation}\label{error1}
\varepsilon^i_1:=(S_{\theta_i}m^+_i)|_{x_2=0}-(S_{\theta_i}m^-_i)|_{x_2=0},
\end{equation}
\begin{equation}\label{error2}
\varepsilon^i_2:=(S_{\theta_i}n^+_i)|_{x_2=0}-(S_{\theta_i}n^-_i)|_{x_2=0},
\end{equation}
\begin{equation}\label{error2}
\varepsilon^i_3:=\mathcal{E}(V_i,\Phi_i).
\end{equation}
We define the modified state as
\begin{eqnarray}\label{modified}
\left\{ \begin{split}
\displaystyle &\Psi^{\pm}_{i+\frac{1}{2}}:=S_{\theta_i}\Psi^{\pm}_i,\\
\displaystyle &m^{\pm}_{i+\frac{1}{2}}:=S_{\theta_i}m^{\pm}_i\mp\frac{1}{2}\mathcal{R}_T\varepsilon^i_1,\\
\displaystyle &n^{\pm}_{i+\frac{1}{2}}:=S_{\theta_i}n^{\pm}_i\mp\frac{1}{2}\mathcal{R}_T\varepsilon^i_2,\\
\displaystyle &v^{\pm}_{i+\frac{1}{2}}:=S_{\theta_i}v^{\pm}_i,\\
\displaystyle &u^{\pm}_{i+\frac{1}{2}}:=\partial_t\Psi^{\pm}_{i+\frac{1}{2}}+(v^{a\pm}+v^{\pm}_{i+\frac{1}{2}})\partial_1\Psi^{\pm}_{i+\frac{1}{2}}+v^{\pm}_{i+\frac{1}{2}}\partial_1\Phi^{a\pm},\\
\end{split}
\right.
\end{eqnarray}
where $S_{\theta_i}$ is the smoothing operator with   ${\theta_i}$ defined by
\begin{equation}\label{theta}
\theta_0\geq1, \; \theta_i=\sqrt{\theta^2_0+i},
\end{equation}
and $\mathcal{R}_T:H^{s-1}_{\lambda}(\omega_T)\rightarrow H^{s,\lambda}_{\ast}(\Omega_T)$  is the lifting operator from the boundary to the interior for $s>1.$ For more details on the trace theorem for the anisotropic spaces,  see \cite{Ohno}. Thanks to \eqref{i1},
we have
\begin{eqnarray}\label{modified22}
\left\{ \begin{split}
\displaystyle &\Psi^{+}_{i+\frac{1}{2}}|_{x_2=0}=\Psi^{+}_{i+\frac{1}{2}}|_{x_2=0}=\psi_{i+\frac{1}{2}},\\
\displaystyle &m^{+}_{i+\frac{1}{2}}|_{x_2=0}=m^{-}_{i+\frac{1}{2}}|_{x_2=0},\\
\displaystyle &n^{+}_{i+\frac{1}{2}}|_{x_2=0}=n^{-}_{i+\frac{1}{2}}|_{x_2=0},\\
\displaystyle &\mathcal{E}(V_{i+\frac{1}{2}},\Psi_{i+\frac{1}{2}})=0,\\
\displaystyle &(V_{i+\frac{1}{2}},\Psi_{i+\frac{1}{2}},\psi_{i+\frac{1}{2}})|_{t<0}=0.\\
\end{split}
\right.
\end{eqnarray}
It then follows from \eqref{dequation}-\eqref{dequation4} that $(U^a+V_{i+\frac{1}{2}},\Phi^a+\Psi_{i+\frac{1}{2}})$ satisfies the Rankine-Hugoniot conditions and the eikonal equations. We note that \eqref{bound3} will be derived when the initial data is chosen to be small enough.
The errors at step $i$ can be defined from the following decompositions:
\begin{equation}\label{l}
\begin{split}
&\mathcal{L}(V_{i+1},\Psi_{i+1})-\mathcal{L}(V_i,\Psi_i)\\
&=\mathbb{L}'(U^a+V_i,\Phi^a+\Psi_i)(\delta V_i,\delta\Psi_i)+e'_i\\
&=\mathbb{L}'(U^a+S_{\theta_i}V_{i},\Phi^a+S_{\theta_i}\Psi_i)(\delta V_i,\delta\Psi_i)+e'_i+e''_i\\
&=\mathbb{L}'(U^a+V_{i+\frac{1}{2}},\Phi^a+\Psi_{i+\frac{1}{2}})(\delta V_i,\delta\Psi_i)+e'_i+e''_i+e'''_i\\
&=\mathbb{L}'_e(U^a+V_{i+\frac{1}{2}},\Phi^a+\Psi_{i+\frac{1}{2}})\delta \dot{V}_i+e'_i+e''_i+e'''_i+D_{i+\frac{1}{2}}\delta\Psi_i\\
\end{split}
\end{equation}
and
\begin{equation}\label{B}
\begin{split}
&\mathcal{B}(V_{i+1}|_{x_2=0},\psi_{i+1})-\mathcal{B}(V_i|_{x_2=0},\psi_i)\\
&=\mathbb{B}'((U^a+V_i)|_{x_2=0},\varphi^a+\psi_i)(\delta V_i|_{x_2=0},\delta\psi_i)+\tilde{e}'_i\\
&=\mathbb{B}'((U^a+S_{\theta_i}V_{i})|_{x_2=0},\varphi^a+S_{\theta_i}\Psi_i|_{x_2=0})(\delta V_i|_{x_2=0},\delta\psi_i)+\tilde{e}'_i+\tilde{e}''_i\\
&=\mathbb{B}'_e((U^a+V_{i+\frac{1}{2}})|_{x_2=0},\varphi^a+\psi_{i+\frac{1}{2}})(\delta \dot{V}_i|_{x_2=0},\delta\psi_i)+\tilde{e}'_i+\tilde{e}''_i+\tilde{e}'''_i,\\
\end{split}
\end{equation}
where we write
\begin{equation}\label{D}
D_{i+\frac{1}{2}}:=\frac{1}{\partial_2(\Phi^a+\Psi_{i+\frac{1}{2}})}\partial_2\mathbb{L}(U^a+V_{i+\frac{1}{2}},\Phi^a+\Psi_{i+\frac{1}{2}}),
\end{equation}
and have used \eqref{linearize2} to get the last identity in \eqref{l}.
Denote
\begin{equation}\label{e}
e_i:=e'_i+e''_i+e'''_i+D_{i+\frac{1}{2}}\delta\Psi_i, \quad \tilde{e}_i:=\tilde{e}'_i+\tilde{e}''_i+\tilde{e}'''_i.
\end{equation}
We assume $f_0:=S_{\theta_0}f^a,(e_0,\tilde{e}_0,g_0):=0$ and $(f_k,g_k,e_k,\tilde{e}_k)$ are already given and vanish in the past for $k=0,\cdots,i-1.$
We can calculate the accumulated errors at step $i,i\geq1$, by
\begin{equation}\label{ae}
E_i:=\sum^{i-1}_{k=0}e_k,\tilde{E}_i:=\sum^{i-1}_{k=0}\tilde{e}_k.
\end{equation}
Then, we obtain $f_i$ and $g_i$ for $i\geq1$ from the equations:
\begin{equation}\label{fg}
\sum^{i}_{k=0}f_k+S_{\theta_i}E_i=S_{\theta_i}f^a,\quad \sum^{i}_{k=0}g_k+S_{\theta_i}\tilde{E}_i=0.
\end{equation}
Therefore, $(V_{i+\frac{1}{2}},\Psi_{i+\frac{1}{2}})$ and $(f_i,g_i)$ have been determined from \eqref{modified}, \eqref{modified22} and \eqref{fg} separately. Then, we can obtain $(\delta \dot{V}_i,\delta\psi_i)$ as the solutions of the linear problem \eqref{effective3}.
Now, we need to construct $\delta\Psi_i=(\delta\Psi^{+}_i,\delta\Psi^{-}_i)^T$ satisfying $\delta\Psi^{\pm}_{i}|_{x_2=0}=\delta\psi_i.$ We use the boundary conditions in \eqref{modified22}, \eqref{boundary7} to derive that $\delta\psi_i$ satisfies
\begin{equation}\label{delta1}
\begin{split}
&\partial_t\delta\psi_i+(v^{a+}+v^+_{i+\frac{1}{2}})|_{x_2=0}\partial_1\delta\psi_i\\
&+\left\{\partial_1(\varphi^a+\psi_{i+\frac{1}{2}})\frac{\partial_2(v^{a+}+v^{+}_{i+\frac{1}{2}})|_{x_2=0}}{\partial_2(\Phi^{a+}+\Psi^{+}_{i+\frac{1}{2}})|_{x_2=0}}-\frac{\partial_2(u^{a+}+u^{+}_{i+\frac{1}{2}})|_{x_2=0}}{\partial_2(\Phi^{a+}+\Psi^{+}_{i+\frac{1}{2}})|_{x_2=0}}\right\}\delta\psi_i\\
&+\partial_1(\varphi^a+\psi_{i+\frac{1}{2}})(\delta \dot{v}^+_i)|_{x_2=0}-(\delta \dot{u}^+_i)|_{x_2=0}=g_{i,2},
\end{split}
\end{equation}
and the equation
\begin{equation}\label{delta2}
\begin{split}
&\partial_t\delta\psi_i+(v^{a-}+v^-_{i+\frac{1}{2}})|_{x_2=0}\partial_1\delta\psi_i\\
&+\left\{\partial_1(\varphi^a+\psi_{i+\frac{1}{2}})\frac{\partial_2(v^{a-}+v^{-}_{i+\frac{1}{2}})|_{x_2=0}}{\partial_2(\Phi^{a-}+\Psi^{-}_{i+\frac{1}{2}})|_{x_2=0}}-\frac{\partial_2(u^{a-}+u^{-}_{i+\frac{1}{2}})|_{x_2=0}}{\partial_2(\Phi^{a-}+\Psi^{-}_{i+\frac{1}{2}})|_{x_2=0}}\right\}\delta\psi_i\\
&+\partial_1(\varphi^a+\psi_{i+\frac{1}{2}})(\delta \dot{v}^-_i)|_{x_2=0}-(\delta \dot{u}^-_i)|_{x_2=0}=g_{i,2}-g_{i,1}.
\end{split}
\end{equation}
Here $g_{i,1},g_{i,2}$ represent the first and second components of $g_i$ respectively.
We shall define $\delta\Psi^+_i,\delta\Psi^-_i$ as solutions to the following equations:
\begin{equation}\label{delta3}
\begin{split}
&\partial_t\delta\Psi^+_i+(v^{a+}+v^+_{i+\frac{1}{2}})\partial_1\delta\Psi^+_i\\
&+\left\{\partial_1(\Phi^{a+}+\Psi^+_{i+\frac{1}{2}})\frac{\partial_2(v^{a+}+v^{+}_{i+\frac{1}{2}})}{\partial_2(\Phi^{a+}+\Psi^{+}_{i+\frac{1}{2}})}-\frac{\partial_2(u^{a+}+u^{+}_{i+\frac{1}{2}})}{\partial_2(\Phi^{a+}+\Psi^{+}_{i+\frac{1}{2}})}\right\}\delta\Psi^+_i\\
&+\partial_1(\Phi^{a+}+\Psi^{+}_{i+\frac{1}{2}})\delta \dot{v}^+_i-\delta \dot{u}^+_i=\mathcal{R}_Tg_{i,2}+h^+_i,
\end{split}
\end{equation}
\begin{equation}\label{delta4}
\begin{split}
&\partial_t\delta\Psi^-_i+(v^{a-}+v^-_{i+\frac{1}{2}})\partial_1\delta\Psi^-_i\\
&+\left\{\partial_1(\Phi^{a-}+\Psi^-_{i+\frac{1}{2}})\frac{\partial_2(v^{a-}+v^{-}_{i+\frac{1}{2}})}{\partial_2(\Phi^{a-}+\Psi^{-}_{i+\frac{1}{2}})}-\frac{\partial_2(u^{a-}+u^{-}_{i+\frac{1}{2}})}{\partial_2(\Phi^{a-}+\Psi^{-}_{i+\frac{1}{2}})}\right\}\delta\Psi^-_i\\
&+\partial_1(\Phi^{a-}+\Psi^{-}_{i+\frac{1}{2}})\delta \dot{v}^-_i-\delta \dot{u}^-_i=\mathcal{R}_T(g_{i,2}-g_{i,1})+h^-_i,
\end{split}
\end{equation}
where $h^{\pm}_{i}$ will be chosen by correcting the eikonal equations.
We define the error terms $\hat{e}'_i,\hat{e}''_i$ and $\hat{e}'''_i$ as
\begin{equation}\label{error4}
\begin{split}
\mathcal{E}(V_{i+1},\Psi_{i+1})-\mathcal{E}(V_{i},\Psi_{i})&=\mathcal{E}'(V_i,\Psi_i)(\delta V_i,\delta\Psi_i)+\hat{e}'_i\\
&=\mathcal{E}'(S_{\theta_i}V_i,S_{\theta_i}\Psi_i)(\delta V_i,\delta\Psi_i)+\hat{e}'_i+\hat{e}''_i\\
&=\mathcal{E}'(V_{i+\frac{1}{2}},\Psi_{i+\frac{1}{2}})(\delta V_i,\delta\Psi_i)+\hat{e}'_i+\hat{e}''_i+\hat{e}'''_i.\\
\end{split}
\end{equation}
Here $\hat{e}'_i$ is the quadratic error, $ \hat{e}''_i$ is the first ``substitution'' error and $\hat{e}'''_i$ is the second ``substitution'' error.
Denote
\begin{equation}\label{he}
\hat{e}_i:=\hat{e}'_i+\hat{e}''_i+\hat{e}'''_i,\quad \hat{ E}_i:=\sum^{i-1}_{k=0}\hat{e}_k.
\end{equation}
Note that $$\mathcal{E}(V,\Psi)=\partial_t\Psi+(v^a+v)\partial_1\Psi-u+v\partial_1\Phi^a.$$
Then, using the good unknown \eqref{goodunkown}, and omitting   $\pm$ superscripts, we compute
\begin{equation}
\begin{split}
&\mathcal{E}'(V_{i+\frac{1}{2}},\Psi_{i+\frac{1}{2}})(\delta V_i,\delta\Psi_i)=\partial_t\delta\Psi_i+(v^a+v_{i+\frac{1}{2}})\partial_1\delta\Psi_i\\
&\qquad +\left\{\partial_1(\Phi^a+\Psi_{i+\frac{1}{2}})\frac{\partial_2(v^a+v_{i+\frac{1}{2}})}{\partial_2(\Phi^a+\Psi_{i+\frac{1}{2}})}-\frac{\partial_2(u^a+u_{i+\frac{1}{2}})}{\partial_2(\Phi^a+\Psi_{i+\frac{1}{2}})}\right\}\delta\Psi_i\\
&\qquad+\partial_1(\Phi^a+\Psi_{i+\frac{1}{2}})\delta \dot{v}_i-\delta\dot{u}_i.
\end{split}
\end{equation}
Hence, using \eqref{delta3}-\eqref{error4}, we obtain
\begin{equation}
\mathcal{E}(V_{i+1},\Psi_{i+1})-\mathcal{E}(V_{i},\Psi_{i})=\left(
\begin{array}{c}
    \mathcal{R}_Tg_{i,2}+h^+_i+\hat{e}^+_i  \\
    \mathcal{R}_T(g_{i,2}-g_{i,1})+h^-_i+\hat{e}^-_i    \\
\end{array}
\right).
\end{equation}
Combining  these equations and using $\mathcal{E}(V_0,\Psi_0)=0$, we obtain that
$$\mathcal{E}(V^+_{i+1},\Phi^+_{i+1})=\mathcal{R}_T(\sum^i_{k=0}g_{k,2})+\sum^i_{k=0}h^+_k+\hat{E}^+_{i+1},$$
$$\mathcal{E}(V^-_{i+1},\Phi^-_{i+1})=\mathcal{R}_T(\sum^i_{k=0}(g_{k,2}-g_{k,1}))+\sum^i_{k=0}h^-_k+\hat{E}^-_{i+1}.$$
Then, we have
\begin{equation}\label{EE}
\begin{split}
&\mathcal{E}(V^+_{i+1},\Phi^+_{i+1})=\mathcal{R}_T(\mathcal{E}(V^+_{i+1}|_{x_2=0},\psi_{i+1})-\tilde{E} _{i+1,2})+\sum^i_{k=0}h^+_k+\hat{E}^+_{i+1},\\
&\mathcal{E}(V^-_{i+1},\Phi^-_{i+1})=\mathcal{R}_T(\mathcal{E}(V^-_{i+1}|_{x_2=0},\psi_{i+1})-\tilde{E}_{i+1,2}+\tilde{E}_{i+1,1})+\sum^i_{k=0}h^-_k+\hat{E}^-_{i+1}.\\
\end{split}
\end{equation}
Here, we have used \eqref{effective3},\eqref{B} to get
$$g_i=\mathcal{B}(V_{i+1}|_{x_2=0},\psi_{i+1})-\mathcal{B}(V_{i}|_{x_2=0},\psi_{i})-\tilde{e}_i.$$
Note that from \eqref{boundary5} and \eqref{system}, we have
\begin{equation}\label{r}
\begin{split}
(\mathcal{B}(V_{i+1}|_{x_2=0},\psi_{i+1}))_2&=\mathcal{E}(V^+_{i+1}|_{x_2=0},\psi_{i+1})\\
&=\mathcal{E}(V^-_{i+1}|_{x_2=0},\psi_{i+1})+(\mathcal{B}(V_{i+1}|_{x_2=0},\psi_{i+1}))_1.
\end{split}
\end{equation}
Hence, \eqref{EE} is obtained by $\mathcal{B}(V_0|_{x_2=0},\psi_0)=0.$
We further assume $(h^+_0,h^-_0,\hat{e}_0)=0$ and $(h^+_k,h^-_k,\hat{e}_k)$ are already given and vanish in the past for $k=0,\cdots,i-1.$
Using all the assumptions before and taking into account \eqref{EE} and the property of $\mathcal{R}_T,$ we can compute the source terms $h^{\pm}_i$:
\begin{equation}\label{+}
S_{\theta_i}(\hat{E}^+_i-\mathcal{R}_T \tilde{E}_{i,2})+\sum^i_{k=0}h^+_k=0,
\end{equation}
\begin{equation}\label{-}
S_{\theta_i}(\hat{E}^-_i-\mathcal{R}_T \tilde{E}_{i,2}+\mathcal{R}_T\tilde{E}_{i,1})+\sum^i_{k=0}h^-_k=0.
\end{equation}
 It is easy to check that $h^{\pm}_i$ and the trace of  $h^{\pm}_i$ on $\omega_T$ vanish in the past. Hence, we see $\delta\Psi^{\pm}_i$ vanishing in the past and satisfying $\delta\Psi^{\pm}_i|_{x_2=0}=\delta\psi_i$ as the unique smooth solutions to the transport equation \eqref{delta3}-\eqref{delta4}. Once $\delta\Psi_i$ is determined, we can obtain $\delta V_i$ from \eqref{goodunkown} and $(V_{i+1},\Psi_{i+1},\psi_{i+1})$ from \eqref{i2}. These error terms: $e'_i, e''_i, e'''_i, \tilde{e}'_i, \tilde{e}''_i, \tilde{e}'''_i, \hat{e}'_i, \hat{e}''_i, \hat{e}'''_i$ are calculated from \eqref{l}, \eqref{B}, \eqref{error4}. Then, $e_i,$ $\tilde{e}_i, $ $\hat{e}_i$ are obtained from \eqref{e} and \eqref{he}.
Using \eqref{effective3} and \eqref{fg}, we sum from $i=0$ to $N$ to obtain
\begin{equation}\label{lb}
\mathcal{L}(V_{N+1},\Psi_{N+1})=\sum^N_{i=0}f_i+E_{N+1}=S_{\theta_N}f^a+(I-S_{\theta_N})E_N+e_N,
\end{equation}
\begin{equation}\label{lb2}
\mathcal{B}(V_{N+1}|_{x_2=0},\psi_{N+1})=\sum^N_{i=0}g_i+\tilde{E}_{N+1}=(I-S_{\theta_N})\tilde{E}_N+\tilde{e}_N.
\end{equation}
Substituting \eqref{+}, \eqref{-} into \eqref{EE} and using \eqref{r}, we have
\begin{eqnarray}\label{+-}
\left\{ \begin{split}
\displaystyle \mathcal{E}(V^+_{N+1},\Psi^+_{N+1})=&\mathcal{R}_T((\mathcal{B}(V_{N+1}|_{x_2=0},\psi_{N+1}))_2)\\
&+(I-S_{\theta_N})(\hat{E}^-_N-\mathcal{R}_T\tilde{E}_{N,2})+\hat{e}^+_N-\mathcal{R}_T\tilde{e}_{N,2},\\
\displaystyle \mathcal{E}(V^-_{N+1},\Psi^-_{N+1})=&\mathcal{R}_T((\mathcal{B}(V_{N+1}|_{x_2=0},\psi_{N+1}))_2-(\mathcal{B}(V_{N+1}|_{x_2=0},\psi_{N+1}))_1)\\
&+(I-S_{\theta_N})(\hat{E}^-_N-\mathcal{R}_T(\tilde{E}_{N,2}-\tilde{E}_{N,1}))+\hat{e}^-_N-\mathcal{R}_T(\tilde{e}_{N,2}-\tilde{e}_{N,1}).\\
\end{split}
\right.
\end{eqnarray}
Since $S_{\theta_N}\rightarrow I $ as $N\rightarrow\infty,$ we can formally obtain the solution to problem \eqref{system} from $\mathcal{L}(V_{N+1},\Psi_{N+1})\rightarrow f^a,\mathcal{B}(V_{N+1}|_{x_2=0},\psi_{N+1})\rightarrow0,$ and $\mathcal{E}(V_{N+1},\Psi_{N+1})\rightarrow 0,$ as error terms $(e_N,\tilde{e}_N,\hat{e}_N)\rightarrow0.$


\section{More Tame Estimates}\label{tameestimate4}

Now, as in  \cite{Coulombel2008}  we present the following lemma for the second derivatives of the system and the tame estimates for the effective linearized problem \eqref{effective3}.

\begin{lemma}\label{secondderi}
Let $T>0, s\in \mathbb{N}$, and $\lambda\geq1$. Assume that the perturbations $\dot{U},\dot{\Phi}$ satisfy
\begin{equation}\label{perturbation3}
[(\dot{U},\dot{\Phi})]_{7,\lambda,T}\leq K,
\end{equation}
where $K$ is a fixed constant that does not depend on $T$ and $\lambda$,     and $(V',\Psi'),(V'',\Psi'')\in H^{s+2}_{\lambda}(\Omega_T),$ then we have
\begin{equation}\label{oned}
\begin{split}
&[\mathbb{L}'(U_{r,l},\Phi_{r,l})(V',\Psi')]_{s,\lambda,T}\leq C(K)\{[(V',\Psi')]_{s+2,\lambda,T}+[(\dot{U},\dot{\Phi})]_{s+2,\lambda,T}[(V',\Phi')]_{7,\lambda,T}\},
\end{split}
\end{equation}
\begin{equation}\label{Lsecond}
\begin{split}
&[\mathbb{L}''(U_{r,l},\Phi_{r,l})((V',\Psi'),(V'',\Psi''))]_{s,\lambda,T}\\
 &\leq C(K)\{[(\dot{U},\dot{\Phi})]_{s+2,\lambda,T}[(V',\Psi')]_{7,\lambda,T}[(V'',\Psi'')]_{7,\lambda,T}\\
&\quad +[(V',\Psi')]_{s+2,\lambda,T}[(V'',\Psi'')]_{7,\lambda,T}\\
&\quad +[(V'',\Psi'')]_{s+2,\lambda,T}[(V',\Psi')]_{7,\lambda,T}\},
\end{split}
\end{equation}
and
\begin{equation}\label{Esecond}
\begin{split}
 [\mathcal{E}''((V',\Psi'),(V'',\Psi''))]_{s,\lambda,T}\leq &C(K)\{[V']_{s,\lambda,T}[\Psi'']_{7,\lambda,T}+[V']_{5,\lambda,T}[\Psi'']_{s+1,\lambda,T}\\
&+[V'']_{5,\lambda,T}[\Psi']_{s+1,\lambda,T}+[V'']_{s,\lambda,T}[\Psi']_{7,\lambda,T}\}.\\
\end{split}
\end{equation}
Moreover, if $(W',\psi'),(W'',\psi'')\in H^{s}_{\lambda}(\omega_T)\times H^{s+1}_{\lambda}(\omega_T),$ then
\begin{equation}\label{Bsecond}
\begin{split}
 &||\mathcal{B}''((W',\psi'),(W'',\psi''))||_{H^{s}_{\lambda}(\omega_T)}\\
 &\leq C(K)\big\{||W'||_{H^{s}_{\lambda}(\omega_T)}||\psi''||_{H^3_{\lambda}(\omega_T)}
 +||W'||_{H^2_{\lambda}(\omega_T)}||\psi''||_{H^{s+1}_{\lambda}(\omega_T)}\\
 &\quad +||W''||_{H^{s}_{\lambda}(\omega_T)}||\psi'||_{H^3_{\lambda}(\omega_T)}
+||W''||_{H^2_{\lambda}(\omega_T)}||\psi'||_{H^{s+1}_{\lambda}(\omega_T)}\big\}.\\
\end{split}
\end{equation}
\end{lemma}

The proof of Lemma \ref{secondderi} is a direct application of Theorems \ref{product}, \ref{composed} and \ref{product3} as well as the Sobolev embedding theorem, see \cite{Trakhinin2005}.

Now, we turn to derive a \textit{priori} estimates for $\delta\Psi_i$ constructed in \eqref{delta3} and \eqref{delta4}. We take the weighted energy estimate  on  $\eqref{delta3}$ and write it in terms of $\delta\tilde{\Psi}_i:=e^{-\lambda t}\delta \Psi_i$ as
\begin{equation}\label{d222} \lambda\delta\tilde{\Psi}_i+\partial_t\delta\tilde{\Psi}_i+a_1\partial_1\delta\tilde{\Psi}_i+a_2\partial_2\delta\tilde{\Psi}_i+a_3e^{-\lambda t}\delta\dot{V}_i=e^{-\lambda t}\mathcal{R}_Tg_{i,2}+e^{-\lambda t}h^{+}_i,
\end{equation}
where $a_1:=v^a+v_{i+\frac{1}{2}},$ $a_2$ and $a_3$ are smooth functions of $\nabla(\Phi^a+{\Psi}_{i+\frac{1}{2}})$ and $\nabla(U^a+V_{i+\frac{1}{2}}).$ For multi-index $\beta,$  we differentiate \eqref{d222} by $D^{\beta}$, then multiply the resulting equality by $\lambda D^{\beta}\delta\tilde{\Psi}_i$ and integrate over $\Omega_T$ to obtain
\begin{equation}\label{energy}
\begin{split}
&\lambda^2\langle D^{\beta}\delta\tilde{\Psi}_i,D^{\beta}\delta\tilde{\Psi}_i\rangle+\lambda\langle D^{\beta}\delta\tilde{\Psi}_i,\partial_tD^{\beta}\delta\tilde{\Psi}_i\rangle+\lambda\langle D^{\beta}\delta\tilde{\Psi}_i,a_1\partial_1D^{\beta}\delta\tilde{\Psi}_i\rangle\\
&\quad +\lambda\langle D^{\beta}\delta\tilde{\Psi}_i,a_2D^{\beta}\delta\tilde{\Psi}_i\rangle+\lambda\langle D^{\beta}\delta\tilde{\Psi}_i,a_3D^{\beta}(e^{-\lambda t}\delta\dot{V}_i)\rangle+\lambda\langle D^{\beta}\delta\tilde{\Psi}_i,[D^{\beta},a_1]\partial_1\delta\tilde{\Psi}_i\rangle\\
&\quad +\lambda\langle D^{\beta}\delta\tilde{\Psi}_i,[D^{\beta},a_2]\delta\tilde{\Psi}_i\rangle+\lambda\langle D^{\beta}\delta\tilde{\Psi}_i,[D^{\beta},a_3]e^{-\lambda t}\delta\dot{V}_i\rangle\\
&=\lambda \langle D^{\beta}\delta\tilde{\Psi}_i,D^{\beta}(e^{-\lambda t}\mathcal{R}_Tg_{i,2})\rangle+\lambda\langle D^{\beta}\delta\tilde{\Psi}_i,D^{\beta}(e^{-\lambda t}h^+_i)\rangle.
\end{split}
\end{equation}
Using the similar argument for the a \textit{priori} estimates in the linear system, we can obtain for $s\geq5$ and $\lambda>1$ large enough,
\begin{equation}\label{energye}
\begin{split}
&\lambda^2[\delta \Psi_i]^2_{s,\lambda,T}\leq C(K)\big\{[\delta \dot{V}_i]^2_{s,\lambda,T}+[\delta \Psi_i]^2_{5,\lambda,T}[(\dot{U}^a+V_{i+\frac{1}{2}},\dot{\Phi}^a+\Psi_{i+\frac{1}{2}})]^2_{s+2,\lambda,T}\\
&+[\delta \dot{V}_i]^2_{5,\lambda,T}[(\dot{U}^a+V_{i+\frac{1}{2}},\dot{\Phi}^a+\Psi_{i+\frac{1}{2}})]^2_{s+1,\lambda,T}+||g_i||^2_{H^{s+1}_{\lambda}(\omega_T)}+[h_i^{\pm}]^2_{s,\lambda,T}\big\}.
\end{split}
\end{equation}
Taking $s=5,$ using the Sobolev embedding theorem and the estimates of $\delta \dot{V}_i$ in \eqref{tameestimate}, we have
\begin{equation}\label{energye2}
\begin{split}
&\lambda[\delta \Psi_i]_{5,\lambda,T}\leq C(K)\{[f_i]_{6,\lambda,T}+||g_i||_{H^6_{\lambda}(\omega_T)}+[h_i^{\pm}]_{5,\lambda,T}\}.
\end{split}
\end{equation}
Combining the above a \textit{priori} tame estimates, we get
\begin{equation}\label{energye3}
\begin{split}
&\sqrt{\lambda}[\delta \dot{V}_i]_{s,\lambda,T}+\lambda[\delta \Psi_i]_{5,\lambda,T}+||\delta\psi_i||_{H^{s+1}_{\lambda}(\omega_T)}\\&\leq C(K)\Big\{[f_i]_{s+1,\lambda,T}+||g_i||_{H^{s+1}_{\lambda}(\omega_T)}+[h_i^{\pm}]_{s,\lambda,T}\\
&\quad +([f_i]_{6,\lambda,T}+||g_i||_{H^6_{\lambda}(\omega_T)}+[h_i^{\pm}]_{5,\lambda,T})[(\dot{U}^a+V_{i+\frac{1}{2}},\dot{\Phi}^a+V_{i+\frac{1}{2}})]_{s+4,\lambda,T}\Big\},
\end{split}
\end{equation}
which is crucial in the proof of the convergence of the iteration scheme in the next section.


\section{Proof Of the Main Result}\label{proof}

In this section we shall follow \cite{Coulombel2008} to show  the convergence of the Nash-Moser scheme and thus prove the main Theorem \ref{stability}.
From the sequence $\{\theta_i\}$ defined in \eqref{theta},  we set $\Delta_i:=\theta_{i+1}-\theta_i.$ Then, the sequence $\{\Delta_i\}$ is decreasing and tends to $0$ as $i$ goes to infinity. Moreover, we have
$$\frac{1}{3\theta_i}\leq\Delta_i=\sqrt{\theta^2_i+1}-\theta_i\leq\frac{1}{2\theta_i}, \; \forall \, i\in \mathbb{N}.$$

\subsection{Inductive analysis}\label{inductive}
Given a small number $\delta>0$, we assume that the following estimate holds:
\begin{equation}\label{small}
[\dot{U}^a]_{\mu+1,\ast,T}+[\dot{\Phi}^a]_{\mu+2,\ast,T}+||\varphi^a||_{H^{\mu+1}_{\lambda}(\Omega_T)}+[f^a]_{\mu-1,\ast,T}\leq\delta.
\end{equation}
Given the integer $\mu:=\tilde{\alpha}+3,$ our inductive assumptions read
\begin{eqnarray}\label{Hn-1}
(H_{i-1}) \left\{ \begin{split}
&\displaystyle (a)\quad[(\delta V_k,\delta\Psi_k)]_{s,\lambda,T}+||\delta\psi_k||_{H^{s+1}_{\lambda}(\omega_T)}\leq \delta\theta^{s-\alpha-1}_{k}\Delta_k, \\
&\qquad \forall k=0,\cdots, i-1, \forall s\in [7,\tilde{\alpha}]\cap \mathbb{N}.\\
&\displaystyle (b)\quad[\mathcal{L}(V_k,\Psi_k)-f^a]_{s,\lambda,T}\leq2\delta\theta^{s-\alpha-1}_k,\\
&\qquad \forall k=0,\cdots,i-1, \forall s\in [7,\tilde{\alpha}-2]\cap \mathbb{N}.\\
&\displaystyle (c)\quad||\mathcal{B}(V_k|_{x_2=0},\psi_k)||_{H^{s}_{\lambda}(\omega_T)}\leq\delta\theta^{s-\alpha-1}_k,\\
&\qquad \forall k=0,\cdots,i-1, \forall s\in [7,\tilde{\alpha}-2]\cap \mathbb{N}.\\
&\displaystyle (d)\quad||\mathcal{E}(V_k,\Psi_k)||_{H^{7}_{\lambda}(\Omega_T)}\leq\delta\theta^{6-\alpha}_k,\\
&\qquad \forall k=0,\cdots,i-1.\\
\end{split}
\right.
\end{eqnarray}
Our goal is to show that   $(H_0)$ holds   and  $(H_{i-1})$ implies  $(H_i)$,  for a suitable choice of parameter $\theta_0\geq1$ and $\delta>0$,  and for $f^a$ small enough.
  Then, we conclude that $(H_i)$ holds for all $i\in \mathbb{N}.$

Assume that, for $\alpha\geq15, \tilde{\alpha}=\alpha+4,$ $\mu= \tilde{\alpha}+3$,  \eqref{small} holds, $\delta>0$ and $[f^a]_{\alpha+1,\lambda,T}/\delta$ are sufficiently small, and $\theta_0\geq1$ is large enough, we first show that $(H_0)$ holds, then show that $(H_{i-1})$ implies  $(H_i)$.

We now prove $(H_0).$
\begin{lemma}\label{H0}
If $[f^a]_{\alpha+1,\lambda,T}/\delta$ is sufficiently small, then  $(H_0)$ holds.
\end{lemma}
\begin{proof}
We recall that $V_0=\Phi_0=\psi_0=0.$ Using the definition of the approximate solutions, Lemma \ref{approximate2} and the construction of the modified states, we have $V_{\frac{1}{2}}=\Phi_{\frac{1}{2}}=\psi_{\frac{1}{2}}=0.$ So the problem becomes
\begin{eqnarray}\label{p}
\begin{cases}
\mathbb{L}'_e(U^a,\Phi^a)\delta\dot{V}_0=S_{\theta_0}f^a, &\text{in } \Omega_T,\\
\mathbb{B}'_e(U^a|_{x_2=0},\varphi^a)(\delta \dot{V}_0|_{x_2=0},\delta\psi_0)=0, &\text{on } \omega_T,\\
\delta\dot{V}_0=0,\delta\psi_0=0, &\text{for }t<0.\\
\end{cases}
\end{eqnarray}
It is easy to see that the Rankine-Hugoniot conditions, eikonal equations  and the tame estimates are satisfied. Moreover,
$\delta\Psi^{\pm}_0$ can be solved from the equations \eqref{delta3} and \eqref{delta4},  i.e.,
$$\partial_t\delta\Psi^{\pm}_0+v^{a\pm}\partial_1\delta\Psi^{\pm}_0+\{\partial_1\Phi^{a\pm}\frac{\partial_2v^{a\pm}}{\partial_2\Phi^{a\pm}}-\frac{\partial_2u^{a\pm}}{\partial_2\Phi^{a\pm}}\}\delta\Psi^{\pm}_0+\partial_1\Phi^{a\pm}\delta\dot{v}^{\pm}_0-\delta\dot{u}^{\pm}_0=0.$$
By \eqref{energye2}, \eqref{energye3} and \eqref{per}, one has
\begin{equation}\nonumber
\begin{split}
&[(\delta V_0,\delta\Psi_0)]_{s,\lambda,T}+||\delta\psi_0||_{H^{s+1}_{\lambda}(\omega_T)}\leq[S_{\theta_0}f^a]_{s+1,\lambda,T}+[S_{\theta_0}f^a]_{6,\lambda,T}[(\dot{U}^a,\dot{\Phi}^a)]_{s+4,\lambda,T}\\
&\leq C[S_{\theta_0}f^a]_{s+1,\lambda,T}\leq C\theta^{(s-\alpha)_+}_0[f^a]_{\alpha+1,\lambda,T}.
\end{split}
\end{equation}
Taking $[f^a]_{\alpha+1,\lambda,T}/\delta$ sufficiently small, we have
$$[(\delta V_0,\delta\Psi_0)]_{s,\lambda,T}+||\delta\psi_0||_{H^{s+1}_{\lambda}(\omega_T)}\leq \delta \theta^{s-\alpha-1}_0\Delta_0,$$
for all $7\leq s\leq\tilde{\alpha}.$
The remaining three inequalities in $(H_0)$ can be proved by taking $[f^a]_{\alpha+1,\lambda,T}$ small enough.
\end{proof}

 Now we prove that $(H_{i-1})$ implies  $(H_i)$. The hypothesis $(H_{i-1})$ yields the following lemma.
\begin{lemma}\label{estimate5}
If $\theta_0$ is large enough, then, for each $k=0,\cdots,i$,  and each integer $s\in[7,\tilde{\alpha}],$
\begin{equation}\label{estimate6}
[(V_k,\Psi_k)]_{s,\lambda,T}+||\psi_k||_{H^{s+1}_{\lambda}(\omega_T)}\leq
 \begin{cases}
\delta \theta^{(s-\alpha)_+}_k, &\text{if }s\neq\alpha,\\
\delta \log\theta_k, &\text{if } s=\alpha,\\
\end{cases}
\end{equation}
\begin{equation}\label{estimate7}
[(I-S_{\theta_k})V_k,(I-S_{\theta_k})\Psi_k]_{s,\lambda,T}\leq C\delta\theta^{s-\alpha}_k.
\end{equation}
Furthermore, for each $k=0,\cdots i$,  and each integer $s\in[7,\tilde{\alpha}+5],$
\begin{equation}\label{estimate8}
[(S_{\theta_k}V_k,S_{\theta_k}\Psi_k)]_{s,\lambda,T}\leq
\begin{cases}
\delta \theta^{(s-\alpha)_+}_k, &\text {if } s\neq\alpha,\\
\delta \log\theta_k, &\text {if } s=\alpha.\\
\end{cases}
\end{equation}
\end{lemma}
The proof of this lemma is based on the classical comparison between series and integrals and Lemma \ref{smooth}; see \cite{Coulombel2004}.
\subsection{Estimate of the quadratic errors}\label{quadratic}
We denote the errors by
\begin{equation}\label{e1}
e'_k:=\mathcal{L}(V_{k+1},\Psi_{k+1})-\mathcal{L}(V_{k},\Psi_{k})-\mathcal{L}'(V_{k},\Psi_{k})(\delta V_k,\delta\Psi_k),
\end{equation}
\begin{equation}\label{e2}
\hat{e}'_k:=\mathcal{E}(V_{k+1},\Psi_{k+1})-\mathcal{E}(V_{k},\Psi_{k})-\mathcal{E}'(V_{k},\Psi_{k})(\delta V_k,\delta\psi_k),
\end{equation}
\begin{equation}\label{e3}
\tilde{e}'_k:=\mathcal{B}(V_{k+1}|_{x_2=0},\psi_{k+1})-\mathcal{B}(V_{k}|_{x_2=0},\psi_{k})-\mathcal{B}'(V_{k}|_{x_2=0},\psi_{k})(\delta V_k|_{x_2=0},\delta\psi_k).
\end{equation}
\begin{lemma}\label{eestimate}
Let $\alpha\geq8.$ There exist $\delta>0$ sufficiently small and $\theta_0\geq1$ sufficiently large such that, for all $k=0,\cdots,i-1,$ and all integers $s\in[7,\tilde{\alpha}-2],$ we have
\begin{equation}\label{ee1}
[e'_k]_{s,\lambda,T}\leq C\delta^2\theta^{L_1(s)-1}_k\Delta_k,
\end{equation}
\begin{equation}\label{ee2}
[\hat{e}'_k]_{s,\lambda,T}\leq C\delta^2\theta^{s+5-2\alpha}_k\Delta_k,
\end{equation}
\begin{equation}\label{ee3}
||\tilde{e}'_k||_{H^{s}_{\lambda}(\omega_T)}\leq C\delta^2\theta^{L_1(s)-1}_k\Delta_k,
\end{equation}
where $L_1(s):=\max\{(s+2-\alpha)_++12-2\alpha;s+7-2\alpha\}.$
\end{lemma}

\begin{proof}
First, we note that
$$e'_k=\int^1_0(1-\tau)\mathbb{L}''(U^a+V_k+\tau\delta V_k,\Phi^a+\Psi_k+\tau \delta\Psi_k)((\delta V_k,\delta \Psi_k)(\delta V_k,\delta \Psi_k))d\tau.$$
From \eqref{small}, Lemma \ref{smooth} and $(H_{i-1})$, we have
$$\sup_{\tau\in[0,1]}[(\dot{U}^a+V_k+\tau\delta V_k,\dot{\Phi}^a+\Psi_k+\tau\delta\Psi_k)]_{7,\lambda,T}\leq C\delta.$$
Taking $\delta$ small enough and using Lemma \ref{perturbation3}, we have
\begin{equation*}
\begin{split}
 [e'_k]_{s,\lambda,T}\leq & C\{[(\dot{U}^a+V_k+\tau\delta V_k,\dot{\Phi}^a+\Psi_k+\tau\delta\Psi_k)]_{s+2,\lambda,T}[(\delta V_k,\delta \Psi_k)]^2_{7,\lambda,T}\\
&+2[(\delta V_k,\delta\Psi_k)]_{s+2,\lambda,T}[(\delta V_k,\delta\Psi_k)]_{7,\lambda,T}\}.
\end{split}
\end{equation*}
If $s+2\neq\alpha$ and $s+2\leq\tilde{\alpha}$ we have
\begin{equation*}
\begin{split}
 [e'_k]_{s,\lambda,T}& \leq C\{(\delta+\delta\theta^{(s+2-\alpha)_+}_k+\delta\theta^{s+2-\alpha-1}_k\Delta_k)\delta^2\theta^{12-2\alpha}_k\Delta^2_k+2\delta^2\Delta^2_k\theta^{s+7-2\alpha}_k\}\\
&\leq C\{\delta^2\theta^{(s+2-\alpha)_++11-2\alpha}\Delta_k+\delta^2\theta^{s+6-2\alpha}\Delta_k\}\leq C\delta^2\Delta_k\theta^{L_1(s)-1}_k,
\end{split}
\end{equation*}
where $L_1(s)=\max\{(s+2-\alpha)_++12-2\alpha;s+7-2\alpha\}.$ For $s+2=\alpha,$ we have
\begin{equation}\nonumber
\begin{split}
 [e'_k]_{s,\lambda,T}&\leq C\{(\delta+\delta \log\theta_k+\delta\theta^{-1}_k\Delta_k)\delta^2\theta^{12-2\alpha}_k\Delta^2_k+2\delta^2\Delta^2_k\theta^{5-\alpha}_k\}\\
&\leq C\{\delta^2\theta^{12-2\alpha}\Delta_k+\delta^2\theta^{4-2\alpha}\Delta_k\}\leq C\delta^2\Delta_k\theta^{L_1(\alpha-2)-1}_k.
\end{split}
\end{equation}
Similarly, we can show that
\begin{equation}\nonumber
\begin{split}
&[\hat{e}'_k]_{s,\lambda,T}\leq C\{[\delta V_k]_{s,\lambda,T}[\delta \Psi_k]_{7,\lambda,T}+[\delta V_k]_{7,\lambda,T}[\delta \Psi_k]_{s+1,\lambda,T}\\
&\quad +[\delta \Psi_k]_{s+1,\lambda,T}[\delta V_k]_{7,\lambda,T}+[\delta \Psi_k]_{7,\lambda,T}[\delta V_k]_{s,\lambda,T}\}\\
&\leq C \delta^2\theta^{s+5-2\alpha}_k\Delta_k.
\end{split}
\end{equation}
Note that $$\tilde{e}'_k=\frac{1}{2}\mathbb{B}''(((\delta V_k)|_{x_2=0},\delta \psi_k),((\delta V_k)|_{x_2=0},\delta \psi_k)).$$
Hence,
\begin{equation}\nonumber
\begin{split}
&||\tilde{e}'_k||_{H^s_{\lambda}(\omega_T)}\leq C\{[\delta V_k]_{s+1,\lambda,T}||\delta\psi_k||_{H^7_{\lambda}(\omega_T)}+[\delta V_k]_{7,\lambda,T}||\delta\psi_k||_{H^{s+1}_{\lambda}(\omega_T)}\}\leq C\delta^2\theta^{L_1(s)-1}_k\Delta_k.\\
\end{split}
\end{equation}
\end{proof}

\subsection{Estimate of the first substitution errors}\label{first}
We can estimate the first substitution errors $e''_k, \hat{e}''_{k}, \tilde{e}''_k$ of the iteration scheme. Define
\begin{equation}\label{e4}
e''_k:=\mathcal{L}'(V_k,\Psi_k)(\delta V_k,\delta\Psi_k)-\mathcal{L}'(S_{\theta_k}V_k,S_{\theta_k}\Psi_k)(\delta V_k,\delta\Psi_k),
\end{equation}
\begin{equation}\label{e5}
\hat{e}''_k:=\mathcal{E}'(V_k,\Psi_k)(\delta V_k,\delta\psi_k)-\mathcal{E}'(S_{\theta_k}V_k,S_{\theta_k}\psi_k)(\delta V_k,\delta\Psi_k),
\end{equation}
\begin{equation}\label{e6}
\tilde{e}''_k:=\mathcal{B}'(V_k|_{x_2=0},\psi_k)(\delta V_k|_{x_2=0},\delta\Psi_k)-\mathcal{B}'(S_{\theta_k}V_k|_{x_2=0},S_{\theta_k}\Psi_k|_{x_2=0})(\delta V_k|_{x_2=0},\delta\psi_k).
\end{equation}
\begin{lemma}\label{se}
Let $\alpha\geq8.$ There exist $\delta>0$ sufficiently small and $\theta_0\geq1$ sufficiently large, such that for all $k=0,\cdots,i-1$ and for all integer $s\in[7,\tilde{\alpha}-2],$ we have
\begin{equation}\label{e6}
[e''_k]_{s,\lambda,T}\leq C\delta^2\theta^{L_2(s)-1}_k\Delta_k,
\end{equation}
\begin{equation}\label{e7}
[\hat{e}''_k]_{s,\lambda,T}\leq C\delta^2\theta^{s+7-2\alpha}_k\Delta_k,
\end{equation}
\begin{equation}\label{e8}
||\tilde{e}''_k||_{H^{s}_{\lambda}(\omega_T)}\leq C\delta^2\theta^{L_2(s)-1}_k\Delta_k,
\end{equation}
where $L_2(s):=\max\{(s+2-\alpha)_++14-2\alpha;s+9-2\alpha\}.$
\end{lemma}
\begin{proof}
We can write
\begin{equation}\nonumber
\begin{split}
 e''_k=\int^1_0 \mathbb{L}''(&U^a+S_{\theta_k}V_k+\tau(I-S_{\theta_k})V_k,\Phi^a+S_{\theta_k}\Psi_k+\tau(I-S_{\theta_k})\Psi_k)\\
&((\delta V_k,\delta\Psi_k)((I-S_{\theta_k})V_k,(I-S_{\theta_k})\Psi_k))d\tau.
\end{split}
\end{equation}
From \eqref{small} and Lemma \ref{estimate5},  we have
$$\sup_{\tau\in[0,1]}[(\dot{U}^a+S_{\theta_k}V_k+\tau(I-S_{\theta_k}) V_k,\dot{\Phi}^a+S_{\theta_k}\Psi_k+\tau(I-S_{\theta_k})\Psi_k)]_{7,\lambda,T}\leq C\delta.$$
Then, we obtain
\begin{equation}\nonumber
\begin{split}
&[e''_k]_{s,\lambda,T}\leq C\{[(\dot{U}^a+S_{\theta_k}V_k+\tau(I-S_{\theta_k}) V_k,\dot{\Phi}^a+S_{\theta_k}\Psi_k+\tau(I-S_{\theta_k})\Psi_k)]_{s+2,\lambda,T}\\
&\quad\times[(\delta V_k,\delta \Phi_k)]_{7,\lambda,T}[((I-S_{\theta_k})V_k,(I-S_{\theta_k})\Psi_k)]_{7,\lambda,T}+[(\delta V_k,\delta \Phi_k)]_{s+2,\lambda,T} \\
&\quad\times[((I-S_{\theta_k})V_k,(I-S_{\theta_k})\Psi_k)]_{7,\lambda,T}+[((I-S_{\theta_k})V_k,(I-S_{\theta_k})\Psi_k)]_{s+2,\lambda,T}[(\delta V_k,\delta\Phi_k)]_{7,\lambda,T}\}.
\end{split}
\end{equation}
From \eqref{small}, Lemma \ref{eestimate}, for $s+2\neq\alpha$ and $s+2\leq \tilde{\alpha},$ we have
\begin{equation}\nonumber
\begin{split}
&[e''_k]_{s,\lambda,T}\leq C\{(\delta+\delta\theta^{(s+2-\alpha)_+}_k+\delta\theta^{s+2-\alpha}_k)\delta^2\Delta_k\theta^{13-2\alpha}_k+\delta^2\Delta_k\theta^{s+8-2\alpha}_k+\delta^2\Delta_k\theta^{s+8-2\alpha}_k\}\\
&\leq C\delta^2\Delta_k\theta^{L_2(s)-1}_k,
\end{split}
\end{equation}
where $L_2(s)=\max\{(s+2-\alpha)_++14-2\alpha;s+9-2\alpha\}.$ For $s+2=\alpha,$ we have
\begin{equation}\nonumber
\begin{split}
&[e''_k]_{s,\lambda,T}\leq C\{(\delta+\delta\log\theta_k+\delta)\delta^2\Delta_k\theta^{13-2\alpha}_k+\delta^2\Delta_k\theta^{6-\alpha}_k+\delta^2\Delta_k\theta^{6-\alpha}_k\}\\
&\leq C\delta^2\Delta_k\theta^{L_2(\alpha-2)-1}_k.
\end{split}
\end{equation}
Similarly, we can obtain the estimates for $\hat{e}''_k$ and $\tilde{e}''_k.$
\end{proof}

\subsection{Estimate of the modified state}\label{modifye}
\begin{lemma}\label{mo}
Let $\alpha\geq8.$ There exist some functions $V_{i+\frac{1}{2}},\Psi_{i+\frac{1}{2}},\psi_{i+\frac{1}{2}}$  vanishing in the past, such that  
$U^a+V_{i+\frac{1}{2}}, \Phi^a+\Psi_{i+\frac{1}{2}}, \varphi^a+\psi_{i+\frac{1}{2}}$ satisfy the constraints \eqref{boundary6} and \eqref{eikonal2}; moreover,
\begin{equation}\label{mo1}
\Psi^{\pm}_{i+\frac{1}{2}}=S_{\theta_i}\Psi^{\pm}_i,\psi_{i+\frac{1}{2}}:=(S_{\theta_i}\Psi^{\pm}_i)|_{x_2=0,}
\end{equation}
\begin{equation}\label{mo2}
v^{\pm}_{i+\frac{1}{2}}=S_{\theta_i}v^{\pm}_i,
\end{equation}
\begin{equation}\label{mo3}
[V_{i+\frac{1}{2}}-S_{\theta_i}V_i]_{s,\lambda,T}\leq C\delta \theta^{s+1-\alpha}_i, \text{ for } s\in[7,\tilde{\alpha}+5].
\end{equation}
\end{lemma}
\begin{proof}
From \eqref{error1}-\eqref{modified22}, we first estimate $\varepsilon^i_1,\varepsilon^i_2$ as follows:
\begin{equation}\nonumber
\begin{split}
&||(m^+_i-m^-_i)|_{x_2=0}||_{H^s_{\lambda}(\omega_T)}\leq||(m^+_{i-1}-m^-_{i-1})|_{x_2=0}||_{H^s_{\lambda}(\omega_T)}+||(\delta m^+_{i-1}-\delta m^-_{i-1})|_{x_2=0}||_{H^s_{\lambda}(\omega_T)}\\
&\leq||\mathcal{B}(V_{i-1}|_{x_2=0},\psi_{i-1})||_{H^s_{\lambda}(\omega_T)}+C[\delta V_{i-1}]_{s+1,\lambda,T}\\
&\leq C \delta \theta^{s-\alpha-1}_i,
\end{split}
\end{equation}
for $s\in[8,\alpha].$
Hence, we have
$$||\varepsilon^i_1||_{H^s_{\lambda}(\omega_T)}\leq C\theta^{s+1-\alpha}_i||(m^+_i-m^-_i)|_{x_2=0}||_{H^{\alpha}_{\lambda}(\omega_T)}\leq C\delta\theta^{s-\alpha}_i,$$
for $s\in[\alpha,\tilde{\alpha}+5].$
As for $s\in[7,\alpha-1],$ we have
$$||\varepsilon^i_1||_{H^s_{\lambda}(\omega_T)}\leq C||(m^+_i-m^-_i)|_{x_2=0}||_{H^{s+1}_{\lambda}(\omega_T)}\leq C\delta\theta^{s-\alpha}_i.$$
Same estimates also hold for $\varepsilon^i_2.$
Therefore, for all $s\in[7,\tilde{\alpha}+5],$ we have
$$[m^{\pm}_{i+\frac{1}{2}}-S_{\theta_i}m^{\pm}_i]_{s,\lambda,T}=\frac{1}{2}[\mathcal{R}_T\varepsilon^i_1]_{s,\lambda,T}\leq C||\varepsilon^i_1||_{H^s_{\lambda}(\omega_T)}\leq C\delta\theta^{s-\alpha}_i.$$
Similarly,
$$[n^{\pm}_{i+\frac{1}{2}}-S_{\theta_i}n^{\pm}_i]_{s,\lambda,T}\leq C\delta\theta^{s-\alpha}_i.$$
Next, we turn to estimate $u_{i+\frac{1}{2}}-S_{\theta_i}u_i$. We first write
\begin{equation}\nonumber
\begin{split}
&u_{i+\frac{1}{2}}-S_{\theta_i}u_i=S_{\theta_i}\varepsilon^i_3+[\partial_t,S_{\theta_i}]\Psi_i+\bar{v}[\partial_1,S_{\theta_i}]\Psi_i\\
&\quad+[(\dot{v}^a+S_{\theta_i}v_i)\partial_1S_{\theta_i}\Psi_i-S_{\theta_i}((\dot{v}^a+v_i)\partial_1\Psi_i)]+(S_{\theta_i}v_i)\partial_1\Phi^a-S_{\theta_i}(v_i\partial_1 \Phi^a).
\end{split}
\end{equation}
Then, we need to estimate the right hand side of the above equality. It is noted that
$$\varepsilon^i_3=\mathcal{E}(V_{i-1},\Psi_{i-1})+\partial_t(\delta\Psi_{i-1})+(v^a+v_{i-1})\partial_1\delta\Psi_{i-1}+\delta V_{i-1}\partial_1(\Phi^a+\Psi_i)-\delta u_{i-1}.$$
Using $(H_{i-1})$, we obtain $[\varepsilon^i_3]_{7,\lambda,T}\leq C\delta\theta^{6-\alpha}_i\leq C\delta\theta_{i}^{8-\alpha},$
hence $$[S_{\theta_i}\varepsilon^i_3]_{s,\lambda,T}\leq C\theta^{s-7}_i[\varepsilon^i_3]_{7,\lambda,T}\leq C\delta\theta^{s-\alpha+1}_i,$$
for $s\in[7,\tilde{\alpha}+5].$
We now estimate the commutators. We take the third commutator as an example. If $s\in[\alpha,\tilde{\alpha}+5],$
we have
\begin{equation}\nonumber
\begin{split}
&[(\dot{v}^a+S_{\theta_i}v_i)\partial_1S_{\theta_i}\Psi_i]_{s,\lambda,T}\leq[\dot{v}^a+S_{\theta_i}v_i]_{7,\lambda,T}[S_{\theta_i}\Psi_i]_{s+1,\lambda,T}
+[\dot{v}^a+S_{\theta_i}v_i]_{s,\lambda,T}[S_{\theta_i}\Psi_i]_{7,\lambda,T}\\
&\leq C\delta^2\theta^{s+1-\alpha}_i,
\end{split}
\end{equation}
and
\begin{equation}\nonumber
\begin{split}
&[S_{\theta_i}(\dot{v}^a+v_i)\partial_1\Psi_i]_{s,\lambda,T}\leq C\theta^{s-\alpha}_i[(\dot{v}^a+v_i)\partial_1\Psi_i]_{\alpha,\lambda,T}\\
&\leq C\theta^{s-\alpha}_i\{[\dot{v}^a+v_i]_{7,\lambda,T}[\Psi_i]_{\alpha+1,\lambda,T}+[\dot{v}^a+v_i]_{\alpha,\lambda,T}[\Psi_i]_{7,\lambda,T}\}
\leq C\delta^2\theta^{s-\alpha+1}_i.
\end{split}
\end{equation}
For $s\in[7,\alpha-1],$ we have
\begin{equation}\nonumber
\begin{split}
&[(\dot{v}^a+S_{\theta_i}v_i)\partial_1S_{\theta_i}\Psi_i-S_{\theta_i}((\dot{v}^a+v_i)\partial_1\Psi_i)]_{s,\lambda,T}\\
&\leq [(v_i-S_{\theta_i}v_i)\partial_1S_{\theta_i}\Psi_i]_{s,\lambda,T}
+[(\dot{v}^a+v_i)\partial_1(\Psi_i-S_{\theta_i}\Psi_i)]_{s,\lambda,T}\\
&\quad +[(I-S_{\theta_i})((\dot{v}^a+v_i)\partial_1{\Psi_i})]_{s,\lambda,T}\\
&\leq C\delta^2\theta^{s+1-\alpha}_i.
\end{split}
\end{equation}
All the remaining terms can be treated similarly. Therefore,
we obtain
$$[u_{i+\frac{1}{2}}-S_{\theta_i}u_i]_{s,\lambda,T}\leq C\delta\theta^{s+1-\alpha}_i.$$
\end{proof}

\subsection{Estimate of the second substitution errors}\label{second3}
We can estimate the second substitution errors $e'''_k,\hat{e}'''_k,\tilde{e}'''_k$ of the iterative scheme. We define
\begin{equation}\label{e9}
e'''_k:=\mathcal{L}'(S_{\theta_k}V_k,S_{\theta_k}\Psi_k)(\delta V_k,\delta\Psi_k)-\mathcal{L}'(V_{k+\frac{1}{2}},\Psi_{k+\frac{1}{2}})(\delta V_k,\delta\Psi_k),
\end{equation}
\begin{equation}\label{e10}
\hat{e}'''_k:=\mathcal{E}'(S_{\theta_k}V_k,S_{\theta_k}\Psi_k)(\delta V_k,\delta\Psi_k)-\mathcal{E}'(V_{k+\frac{1}{2}},\Psi_{k+\frac{1}{2}})(\delta V_k,\delta\Psi_k),
\end{equation}
\begin{equation}\label{e11}
\tilde{e}'''_k:=\mathcal{B}'(S_{\theta_k}V_k|_{x_2=0},S_{\theta_k}\Psi_k|_{x_2=0})((\delta V_k)|_{x_2=0},\delta\psi_k)-\mathcal{B}'(V_{k+\frac{1}{2}}|_{x_2=0},\psi_{k+\frac{1}{2}})((\delta V_k)|_{x_2=0},\delta\psi_k).
\end{equation}
\begin{lemma}\label{seconde}
Let $\alpha\geq8.$ There exist $\delta>0$ sufficiently small and $\theta_0\geq1$ sufficiently large such that, for all $k=0,\cdots,i-1$ and for all integer $s\in[7,\tilde{\alpha}-2],$ we have $\hat{e}'''_k,\tilde{e}'''_k=0$ and
\begin{equation}
[e'''_k]_{s,\lambda,T}\leq C\delta^2\theta^{L_3(s)-1}_k\Delta_k,
\end{equation}
where $L_3(s):=\max\{(s+2-\alpha)_++16-2\alpha;s+10-2\alpha\}.$
\end{lemma}
\begin{proof}
We can write
\begin{equation}\nonumber
\begin{split}
 [e'''_k]=\int^1_0 \mathbb{L}''(&U^a+V_{k+\frac{1}{2}}+\tau(S_{\theta_k}V_k-V_{k+\frac{1}{2}}),\Phi^a+\Psi_{k+\frac{1}{2}})\\
&((\delta V_k,\delta \Psi_k)(S_{\theta_k}V_k-V_{k+\frac{1}{2}},0))d\tau.
\end{split}
\end{equation}
From \eqref{small}, Lemma \ref{eestimate}, Lemma \ref{mo}, we have
$$\sup_{\tau\in[0,1]}[(\dot{U}^a+V_{k+\frac{1}{2}}+\tau(S_{\theta_k}V_k-V_{k+\frac{1}{2}}),\dot{\Phi}^a+\Psi_{k+\frac{1}{2}})]_{s+2,\lambda,T}\leq C\delta\theta^{(s+2-\alpha)_++1}_k.$$
Thus
\begin{equation}\nonumber
\begin{split}
 [e'''_k]_{s,\lambda,T}&\leq \{[ (\dot{U}^a+V_{k+\frac{1}{2}}+\tau(S_{\theta_k}V_k-V_{k+\frac{1}{2}}),\dot{\Phi}^a+\Psi_{k+\frac{1}{2}})]_{s+2,\lambda,T}[(\delta V_k,\delta \Psi_k)]_{7,\lambda,T}\\
&\quad \times[S_{\theta_k}V_k-V_{k+\frac{1}{2}}]_{7,\lambda,T}+[(\delta V_k,\delta\Psi_k)]_{s+2,\lambda,T}[S_{\theta_k}V_k-V_{k+\frac{1}{2}}]_{7,\lambda,T}\\
&\quad +[S_{\theta_k}V_k-V_{k+\frac{1}{2}}]_{s+2,\lambda,T}[(\delta V_k,\delta \Psi_k)]_{7,\lambda,T}\}\\
&\leq C\delta^2\Delta_k\theta^{L_3(s)-1}_k,
\end{split}
\end{equation}
where $L_3(s)=\max\{(s+2-\alpha)_++16-2\alpha;s+10-2\alpha\}.$
It is easy to check that $\hat{e}'''_k$ and $\tilde{e}'''_k$ vanish.
\end{proof}

\subsection{Estimate of the last error term}\label{last}
We now estimate the last error term \eqref{D}:
\begin{equation}\label{D2}
D_{k+\frac{1}{2}}\delta\Psi_k=\frac{\delta\Psi_k}{\partial_2(\Phi^a+\Psi_{k+\frac{1}{2}})}R_k,
\end{equation}
where $R_k:=\partial_2[\mathbb{L}(U^a+V_{k+\frac{1}{2}},\Phi^a+\Psi_{k+\frac{1}{2}})].$
Note that, from \eqref{modified}, \eqref{small} and \eqref{estimate8},
$$|\partial_2(\Phi^a+\Psi_{k+\frac{1}{2}})|=|\pm 1+\partial_2(\dot{\Phi}^a+\Psi_{k+\frac{1}{2}})|\geq\frac{1}{2},$$
provided that $\delta$ is small enough.
Since $U^a$ and $\Phi^a$ do not vanish in the past, $V_{i+\frac{1}{2}}$ and $\Psi_{i+\frac{1}{2}}$ vanish in the past, we cannot expect $R_k$ vanishing in the past. However, since $\delta\Psi_i$ vanishes in the past, we could expect $D_{k+\frac{1}{2}}\delta\Psi_k$ vanishing in the past. In order to take advantage of the existence result in the linear system, it is required that the source term should vanish  in the past. Hence, we need to restrict our analysis on the part of the domain $\Omega_T$ with positive time variable. For the negative time variable, it can be treated similarly. Since we apply the Gagliardo-Nirenberg inequality on the anisotropic Sobolev space with such domains, we will not distinguish the norms of anisotropic Sobolev spaces between $\Omega_T$ and $\Omega^+_T=\{(t,x)\in\Omega_T,t>0\}.$ Simple calculation yields
\begin{equation}\nonumber
\begin{split}
[D_{k+\frac{1}{2}}\delta\Psi_k]_{s,\lambda,T}\leq &C\{[\delta \Psi_k]_{s,\lambda,T}||R_k||_{W^{1,tan}(\Omega^+_T)}||(\partial_2(\Phi^a+\Psi_{k+\frac{1}{2}}))^{-1}||_{W^{1,tan}(\Omega^+_T)}\\
&+||\delta\Psi_k||_{W^{1,tan}(\Omega^+_T)}\times([R_k]_{s,\lambda,T}||(\partial_2(\Phi^a+\Psi_{k+\frac{1}{2}}))^{-1}||_{W^{1,tan}(\Omega^+_T)}\\
&+||R_k||_{W^{1,tan}(\Omega^+_T)}[(\partial_2(\Phi^a+\Psi_{k+\frac{1}{2}}))^{-1}]_{s,\lambda,T})\}.
\end{split}
\end{equation}
Then, we have the following estimate:
\begin{lemma}\label{es} Let $\alpha\geq8$ and $\tilde{\alpha}\geq\alpha+3.$ Then for $\delta>0$ sufficiently small, $\theta_0\geq1$ sufficiently large,  such that, for all $k=0,\cdots,i-1$ and for all integers $s\in[7,\tilde{\alpha}-2],$ we have
\begin{equation}\label{Rk}
[R_k]_{s,\lambda,T}\leq C\delta(\theta_k^{{s+5-\alpha}}+\theta_k^{{(s+4-\alpha)}_++9-\alpha}).
\end{equation}
\end{lemma}
\begin{proof}
Using the definition of $R_k,$ we obtain
$$[R_k]_{s,\lambda,T}=[\mathbb{L}(U^a+V_{k+\frac{1}{2}},\Phi^a+\Psi_{k+\frac{1}{2}})]_{s+2,\lambda,T}.$$
Then, we write
\begin{equation}\nonumber
\begin{split}
&\mathbb{L}(U^a+V_{k+\frac{1}{2}},\Phi^a+\Psi_{k+\frac{1}{2}})\\
=&\mathbb{L}(U^a+V_{k+\frac{1}{2}},\Phi^a+\Psi_{k+\frac{1}{2}})-\mathbb{L}(U^a+V_k,\Phi^a+\Psi_k)+\mathcal{L}(V_k,\Psi_k)-f^a.\\
\end{split}
\end{equation}
If $s+2\leq\tilde{\alpha}-2,$ from $(H_{i-1}),$ we have
$$[\mathcal{L}(V_k,\Psi_k)-f^a]_{s+2,\lambda,T}\leq 2\delta \theta^{s+1-\alpha}_k.$$
Then,
\begin{equation}\nonumber
\begin{split}
&\mathbb{L}(U^a+V_{k+\frac{1}{2}},\Phi^a+\Psi_{k+\frac{1}{2}})-\mathbb{L}(U^a+V_{k},\Phi^a+\Psi_{k})\\
&=\int^1_0\mathbb{L}'(U^a+V_k+\tau(V_{k+\frac{1}{2}}-V_k),\Phi^a+\Psi_k+\tau(\Psi_{k+\frac{1}{2}}-\Psi_k))(V_{k+\frac{1}{2}}-V_k,\Psi_{k+\frac{1}{2}}-\Psi_k)d\tau.\\
\end{split}
\end{equation}
We note that
$$\sup_{\tau\in[0,1]}[(\dot{U}^a+V_{k}+\tau(V_{k+\frac{1}{2}}-V_k),\dot{\Phi}^a+\Psi_{k}+\tau(\Psi_{k+\frac{1}{2}}-\Psi_k))]_{7,\lambda,T}\leq C\delta.$$
Using \eqref{oned}, we can obtain
\begin{equation}\nonumber
\begin{split}
&[\mathbb{L}(U^a+V_{k+\frac{1}{2}},\Phi^a+\Psi_{k+\frac{1}{2}})-\mathbb{L}(U^a+V_{k},\Phi^a+\Psi_{k})]_{s+2,\lambda,T}\\
&\leq C\{[(V_{k+\frac{1}{2}}-V_k,\Psi_{k+\frac{1}{2}}-\Psi_k)]_{s+4,\lambda,T}+[(V_{k+\frac{1}{2}}-V_k,\Psi_{k+\frac{1}{2}}-\Psi_k)]_{7,\lambda,T}\\
&\quad \times[(\dot{U}^a+V_{k}+\tau(V_{k+\frac{1}{2}}-V_k),\dot{\Phi}^a+\Psi_{k}+\tau(\Psi_{k+\frac{1}{2}}-\Psi_k))]_{s+4,\lambda,T}\},\\
&\leq C\delta\{\theta^{s+5-\alpha}_k+\theta^{(s+4-\alpha)_++9-\alpha}_k\}.
\end{split}
\end{equation}
Then, we consider $s=\tilde{\alpha}-2$ and $s=\tilde{\alpha}-3$ separately, we obtain
\begin{equation}\nonumber
\begin{split}
&[R_k]_{s,\lambda,T}=[\mathbb{L}(U^a+V_{k+\frac{1}{2}},\Phi^a+\Psi_{k+\frac{1}{2}})]_{s+2,\lambda,T}\leq [(\dot{U}^a+V_{k+\frac{1}{2}},\dot{\Phi}^a+\Psi_{k+\frac{1}{2}})]_{s+4,\lambda,T}\\
&\leq C\delta \theta^{s+5-\alpha}_k.
\end{split}
\end{equation}
We have proved the above Lemma \ref{es}.
\end{proof}
The following Lemma \ref{DD} can be proved by direct calculations.
\begin{lemma}\label{DD}
 Let $\alpha\geq13,\tilde{\alpha}\geq\alpha+3.$ There exist $\delta>0$ sufficiently small and $\theta_0\geq1$ sufficiently large,  such that,  for all $k=0,\cdots,i-1$ and for all integer $s\in [7,\tilde{\alpha}-2]$, we have
\begin{equation}\label{D_k}
[D_{k+\frac{1}{2}}\delta\Psi_k]_{s,\lambda,T}\leq C\delta^2\theta^{L_4(s)-1}_k\Delta_k,
\end{equation}
where $L_4(s):=\max\{s+13-2\alpha;(s+4-\alpha)_++16-2\alpha;(s+2-\alpha)_++19-2\alpha\}.$
\end{lemma}
\subsection{Convergence of the iteration scheme }\label{convergence}
We first estimate the errors $e_k, \hat{e}_k, \tilde{e}_k.$
\begin{lemma}\label{error5}
Let $\alpha\geq 13.$ There exist $\delta>0$ sufficiently small and $\theta_0\geq1$ sufficiently large, such that for all $k=0,\cdots,i-1$ and for all integer $s\in [7,\tilde{\alpha}-2]$, we have
\begin{equation}
[e_k]_{s,\lambda,T}+[\hat{e}_k]_{s,\lambda,T}+||\tilde{e}_k||_{H^s_{\lambda}(\omega_T)}\leq C\delta^2\theta^{L_4(s)-1}_k\Delta_k,
\end{equation}
where $L_4(s)$ is defined in Lemma \ref{DD}.
\end{lemma}

From Lemma \ref{error5}, we obtain the estimate of the accumulated errors $E_i,\tilde{E}_i,\hat{E}_i.$
\begin{lemma}\label{Eerror}
Let $\alpha\geq15,\tilde{\alpha}\geq\alpha+4.$ There exist $\delta>0$ sufficiently small and $\theta_0\geq1$ sufficiently large, such that
\begin{equation}\label{Eerrore}
[(E_i,\hat{E}_i)]_{\tilde{\alpha}-2,\lambda,T}+||\tilde{E}_i||_{H^{\tilde{\alpha}-2}_{\lambda}(\omega_T)}\leq C\delta^2\theta_i.
\end{equation}
\end{lemma}
We still need to estimate the source terms $f_{i},g_{i},h^{\pm}_i.$
\begin{lemma}\label{serror}
Let $\alpha\geq15$, and $\tilde{\alpha}=\alpha+4.$ There exist $\delta>0$ sufficiently small and $\theta_0\geq1$ sufficiently large, such that for all integers $s\in[7,\tilde{\alpha}+1],$
\begin{equation}\label{ferrore}
[f_i]_{s,\lambda,T}\leq C\Delta_i\{\theta^{s-\alpha-2}_i([f^a]_{\alpha+1,\lambda,T}+\delta^2)+\delta^2\theta^{L_4(s)-1}_i\},
\end{equation}
\begin{equation}\label{gerrore}
||g_i||_{H^{s}_{\lambda}(\omega_T)}\leq C\delta^2\Delta_i\{\theta^{s-\alpha-2}_i+\theta^{L_4(s)-1}_i\},
\end{equation}
 and for all integers $s\in [7,\tilde{\alpha}],$
\begin{equation}\label{h}
[h^{\pm}_i]_{s,\lambda,T}\leq C\delta^2\Delta_i(\theta^{s-\alpha-2}_i+\theta^{L_4(s)-1}_i).
\end{equation}
\end{lemma}
\begin{proof}
Using the definition of $f_i,g_i,h^{\pm}_i,$ we have
$$f_i=(S_{\theta_i}-S_{\theta_i-1})f^a-(S_{\theta_i}-S_{\theta_i-1})E_{i-1}-S_{\theta_i}e_{i-1},$$
$$g_i=-(S_{\theta_i}-S_{\theta_i-1})\tilde{E}_{n-1}-S_{\theta_i}\tilde{e}_{i-1},$$
$$h^+_i=(S_{\theta_i}-S_{\theta_i-1})(\mathcal{R}_T\tilde{E}_{i-1,2}-\hat{E}^+_{i-1})+S_{\theta_i}(\mathcal{R}_T\tilde{e}_{i-1,2}-\hat{e}^-_{i-1}),$$
$$h^-_i=(S_{\theta_i}-S_{\theta_i-1})(\mathcal{R}_T\tilde{E}_{i-1,2}-\mathcal{R}_T\tilde{E}_{i-1,1}-\hat{E}^-_{i-1})+S_{\theta_i}(\mathcal{R}_T\tilde{e}_{i-1,2}-\mathcal{R}_T\tilde{e}_{i-1,1}-\hat{e}^-_{i-1}).$$
Then this Lemma \ref{serror} follows from Lemmas \ref{small}, \ref{error5}, and \ref{Eerror}.
\end{proof}

Now, we consider the estimate of the solutions to problem \eqref{effective2} by using the tame estimate.
\begin{lemma}\label{delta}
Let $\alpha\geq15.$ If $\delta>0,$$[f^a]_{\alpha+1,\lambda,T}/\delta$ are sufficiently small and $\theta_0\geq1$ is sufficiently large, then for all integers $s\in[7,\tilde{\alpha}],$
 \begin{equation}\label{d}
[(\delta V_i,\delta\Psi_i)]_{s,\lambda,T}+||\delta\psi_i||_{H^{s+1}_{\lambda}(\omega_T)}\leq \delta\theta^{s-\alpha-1}_i\Delta_i.
\end{equation}
\end{lemma}
\begin{proof}
Using \eqref{goodunkown}, one has
\begin{equation}\label{EEEE}
[\delta V_i]_{s,\lambda,T}\leq C[(\delta\dot{V}_i,\delta\Psi_i)]_{s,\lambda,T}+[\delta\Psi_i]_{5,\lambda,T}[(\dot{U}^a+V_{i+\frac{1}{2}},\dot{\Phi}^a+\Psi_{i+\frac{1}{2}})]_{s+2,\lambda,T}.
\end{equation}
Using \eqref{EEEE}, \eqref{energye2} and \eqref{energye3}, we obtain
\begin{equation}\nonumber
\begin{split}
&[\delta V_i]_{s,\lambda,T}+[\delta\Psi_i]_{s,\lambda,T}+||\delta\psi_i||_{H^{s+1}_{\lambda}(\omega_T)}\leq C(K)\{
[f_i]_{s+1,\lambda,T}+||g_i||_{H^{s+1}_{\lambda}(\omega_T)}\\
&\qquad +[h_i]_{s,\lambda,T}+([f_i]_{6,\lambda,T}+||g_i||_{H^6_{\lambda}(\omega_T)}+[h_i^{\pm}]_{5,\lambda,T})[(\dot{U}^a+V_{i+\frac{1}{2}},\dot{\Phi}^a+\Psi_{i+\frac{1}{2}})]_{s+4,\lambda,T}
\}.
\end{split}
\end{equation}
By \eqref{per}, Lemma \ref{eestimate}, Lemma \ref{mo} and Lemma \ref{serror}, we have
\begin{equation}\nonumber
\begin{split}
&[\delta V_i]_{s,\lambda,T}+[\delta\Psi_i]_{s,\lambda,T}+||\delta\psi_i||_{H^{s+1}_{\lambda}(\omega_T)}\\
&\leq C\{\Delta_i(\theta^{5-\alpha}_i([f^a]_{\alpha+1,\lambda,T}+\delta^2\theta^{20-2\alpha}_i))(\delta+\delta\theta^{(s+4-\alpha)_+}_i+\delta\theta^{s+5-\alpha}_i)\\
&\quad +\Delta_i(\theta^{s-\alpha-1}_i([f^a]_{\alpha+1,\lambda,T}+\delta^2)+\delta^2\theta^{L_4(s+1)-1}_i)\}.
\end{split}
\end{equation}
We want each term on the right hand side of the above inequality to be less than $\delta\theta^{s-\alpha-1}_i\Delta_i$ for all $s\in[7,\tilde{\alpha}].$ Since $\alpha\geq15,$ we have
\begin{eqnarray}\label{alpha}
\left\{ \begin{split}
\displaystyle &L_4(s+1)\leq s-\alpha,\\
\displaystyle &(s+4-\alpha)_++5-\alpha\leq s-\alpha-1,\\
\displaystyle &s+10-2\alpha\leq s-\alpha-1,\\
\displaystyle &(s+4-\alpha)_++20-2\alpha\leq s-\alpha-1,\\
\displaystyle &s+25-3\alpha\leq s-\alpha-1.\\
\end{split}
\right.
\end{eqnarray}
Lemma \ref{delta} is proved after taking $\delta$ and $[f^a]_{\alpha+1,\lambda,T}/\delta$ small enough.
\end{proof}

We now prove the remaining inequalities in $(H_i).$
\begin{lemma}\label{Hn}
Let $\alpha\geq15.$ If $\delta>0,$ $[f^a]_{\alpha+1,\lambda,T}/\delta$ are sufficiently small and if $\theta_0\geq1$ is sufficiently large, then for all integers $s\in[7,\tilde{\alpha}],$
 \begin{equation}\label{Le}
[\mathcal{L}(V_i,\Psi_i)-f^a]_{s,\lambda,T}\leq 2\delta\theta^{s-\alpha-1}_i.
\end{equation}
Moreover, for all integers $s\in[7,\tilde{\alpha}-2],$
 \begin{equation}\label{B2}
||\mathcal{B}(V_i|_{x_2=0},\Psi_i)||_{H^{s}_{\lambda}(\omega_T)}\leq \delta\theta^{s-\alpha-1}_i.
\end{equation}
and
\begin{equation}\label{E}
||\mathcal{E}(V_i,\Psi_i)||_{H^{7}_{\lambda}(\Omega_T)}\leq \delta\theta^{6-\alpha}_i.
\end{equation}
\end{lemma}
\begin{proof}
From the iteration step, we can decompose
$$\mathcal{L}(V_i,\Phi_i)-f^a=(S_{\theta_{i-1}}-I)f^a+(I-S_{\theta_{i-1}})E_{i-1}+e_{i-1}.$$
First, we estimate $[(S_{\theta_{i-1}}-I)f^a]_{s,\lambda,T}.$ If $s\in[\alpha+1,\tilde{\alpha}-2],$
\begin{equation}\nonumber
\begin{split}
&[(S_{\theta_{i-1}}-I)f^a]_{s,\lambda,T}\leq[S_{\theta_{i-1}}f^a]_{s,\lambda,T}+[f^a]_{s,\lambda,T}\\
&\leq C\theta^{s-\alpha-1}_{i-1}[f^a]_{\alpha+1,\lambda,T}+[f^a]_{\tilde{\alpha}-2,\lambda,T}\leq C\theta^{s-\alpha-1}_i([f^a]_{\alpha+1,\lambda,T}+\delta).\\
\end{split}
\end{equation}
If $s\in[7,\alpha+1],$ we have
$$[(S_{\theta_{i-1}}-I)f^a]_{s,\lambda,T}\leq C\theta^{s-\alpha-1}_{i-1}[f^a]_{\alpha+1,\lambda,T}.$$
For the remaining terms, we have
$$[(I-S_{\theta_{i-1}})E_{i-1}]_{s,\lambda,T}\leq C\theta^{s-\tilde{\alpha}+2}_{i-1}[E_{i-1}]_{\tilde{\alpha}-2,\lambda,T}\leq C\theta^{s-\tilde{\alpha}+2}_{i-1}\delta^2\theta_{i-1}\leq C\delta^2\theta^{s-\alpha-1}_{i}.$$
$$[e_{i-1}]_{s,\lambda,T}\leq C\delta^2\theta^{L(s)-1}_i\Delta_i\leq C\delta^2\theta^{L(s)-2}_i\leq C\delta^2\theta^{s-\alpha-1}_i.$$
Combining all the terms and taking $\delta$ and $[f^a]_{\alpha+1,\lambda,T}/\delta$ small enough, we obtain \eqref{Le}. \eqref{B2} and \eqref{E} can be proved similarly and hence we have proved Lemma \ref{Hn}.
\end{proof}

\smallskip


\noindent
{\it Proof of Theorem \ref{stability}}: \;
Given the initial data $(U^{\pm}_0,\varphi_0)$ satisfying all the assumptions of Theorem \ref{stability},  $\alpha\geq15$ and let $\tilde{\alpha}=\alpha+4$ and $\mu=\tilde{\alpha}+3.$ Then the initial data $U^{\pm}_0$ and $\varphi_0$ are compatible up to order $\mu=\tilde{\alpha}+3.$ From \eqref{estimatea} and \eqref{fa}, we can obtain \eqref{small} and all the requirements of Lemma \ref{delta}, Lemma \ref{Hn} and Lemma \ref{H0} provided that $(\dot{U}^{\pm}_0,\varphi_0)$ is sufficiently small in $H^{2\mu+1}_{\ast}(\R^2_+)\times H^{2\mu+2}(\R)$ with $\dot{U}^{\pm}_0:=U^{\pm}_0-\bar{U}^{\pm}.$ Hence, for small initial data, property $(H_i)$ holds for all integers $i$. In particular, we have
$$\sum^{\infty}_{i=0}([(\delta V_i,\delta\Psi_i)]_{\alpha-1,\lambda,T}+||\delta\psi_i||_{H^{\alpha}_{\lambda}(\omega_T)})\leq C\sum^{\infty}_{i=0}\theta^{-2}_i\Delta_i<\infty.$$
Thus, sequence $(V_i,\Psi_i)$ converges to some limit $(V,\Psi)$ in $H^{\alpha-1}_{\lambda}(\Omega_T),$ and sequence $\psi_i$ converges to some limit $\psi$ in $H^{\alpha}_{\lambda}(\omega_T).$ Passing to the limit in \eqref{Le} and \eqref{B2} for $s=\alpha-1$ and in \eqref{E}, we obtain \eqref{system}. Therefore, $(U,\Phi)=(U^a+V,\Phi^a+\Psi)$ is a solution on $\Omega_T$ of nonlinear systems \eqref{equation2}-\eqref{eikonal}.


\appendix
\section{Some Basic Estimates}\label{Appendix}
We present some nonlinear tame estimates in the weighted Sobolev space $H^s_{\lambda}(\omega_T)$ and anisotropic Sobolev space $H^{s,\lambda}_{\ast}(\Omega_T).$ For more details, see \cite{Secchi2000}. First, we consider the weighted Sobolev space $H^s_{\lambda}(\omega_T)$.
\begin{theorem}[Gagliardo-Nirenberg] \label{GN}
 Let $s>1$ be an integer, $\lambda\geq1$ and $T\in\R.$ There is a constant $C$ which is independent of $\lambda$ and $T$ such that for all $u\in H^{s}_{\lambda}(\omega_T)\cap L^{\infty}(\omega_T)$ and
all multi-index $\alpha\in \mathbb{N}^3$ with $|\alpha|\leq s$, we have
$$||\partial^{\alpha}u||_{L^{2p}_{\lambda}(\omega_T)}\leq C||u||^{1-\frac{1}{p}}_{L^{\infty}(\omega_T)}||u||^{\frac{1}{p}}_{H^s_{\lambda}(\omega_T)},\quad \frac{1}{p}=\frac{|\alpha|}{s}.$$
\end{theorem}
This result can be used to prove the following tame estimates for products of functions in $H^s_{\lambda}(\omega_T).$
\begin{theorem}\label{product}
Let $s\geq1$ be an integer, $\lambda\geq1$ and $T\in\R.$ There exists a constant $C$ which is independent of $\lambda$ and $T$ such that for all functions $u,v\in H^s_{\lambda}(\omega_T)\cap L^{\infty}(\omega_T),$ the product $uv$ belongs to $H^s_{\lambda}(\omega_T)$ and satisfies the estimate
$$||uv||_{H^s_{\lambda}(\omega_T)}\leq C(||u||_{L^{\infty}(\omega_T)}||v||_{H^s_{\lambda}(\omega_T)}+||v||_{L^{\infty}(\omega_T)}||u||_{H^s_{\lambda}(\omega_T)}).$$
\end{theorem}
Furthermore, there is a tame estimate for the composed functions.
\begin{theorem}\label{composed}
Let $s\geq1$ be an integer, $\lambda\geq1$ and $T\in\R.$ Assume that $F$ is a $C^{\infty}$ function such that $F(0)=0.$ Then, there is an increasing function $C(\cdot)$ which is independent of $\lambda$ and $T$ such that for all $u\in H^s_{\lambda}(\omega_T)\cap L^{\infty}(\omega_T),$ the composed function $F(u)$ belongs to $H^s_{\lambda}(\omega_T)$ and satisfies
$$||F(u)||_{H^s_{\lambda}(\omega_T)}\leq C(||u||_{L^{\infty}(\omega_T)})||u||_{H^s_{\lambda}(\omega_T)}.$$
\end{theorem}
Then, we introduce the following Sobolev embedding estimates in $H^s_{\lambda}(\omega_T).$
\begin{theorem}\label{sb}
The following inequalities hold with constants $C$ which are independent of $\lambda\geq1,$
$$||e^{-\lambda T}u||_{L^{\infty}(\omega_T)}\leq\frac{C}{\lambda}||u||_{H^2_{\lambda}(\omega_T)}, \quad \forall u\in H^2_{\lambda}(\omega_T),$$
$$||e^{-\lambda T}u||_{W^{1,\infty}(\omega_T)}\leq C||u||_{H^3_{\lambda}(\omega_T)}, \quad \forall u\in H^3_{\lambda}(\omega_T).$$
\end{theorem}
Next, we focus on the weighted anisotropic Sobolev space $H^{s,\lambda}_{\ast}(\Omega_T).$ We note that the estimate in $H^{s,\lambda}_{\ast}(\Omega_T)$ are different for $s$ to be odd and even. First, we present the estimates when $s$ is even.
\begin{theorem}[Gagliardo-Nirenberg]\label{GN2}
 Let $s>1$ be an even integer, $\lambda\geq1$ and $T\in\R.$ There is a constant $C$ which is independent of $\lambda$ and $T$ such that for all $u\in H^{s,\lambda}_{\ast}(\Omega_T)\cap L^{\infty}(\Omega_T)$ and
all multi-index $\alpha\in \mathbb{N}^3,k\in \mathbb{N}$ with $|\alpha|+2k\leq s$, we have
$$||\partial^{\alpha}_{\ast}\partial^k_2u||_{L^{2p}_{\lambda}(\Omega_T)}\leq C||u||^{1-\frac{1}{p}}_{L^{\infty}(\Omega_T)}[u]^{\frac{1}{p}}_{s,\lambda,T},\frac{1}{p}=\frac{|\alpha|+2k}{s}.$$
\end{theorem}
Similarly  to  $H^s_{\lambda}(\omega_T),$ we also have the following estimates for the products and composed functions.
\begin{theorem}\label{product2}
Let $s\geq1$ be an even integer, $\lambda\geq1$ and $T\in\R.$ Then, for all functions $u,v\in H^{s,\lambda}_{\ast}(\Omega_T)\cap L^{\infty}(\Omega_T)$ and $C^{\infty}$ function $F$ of $u,$ $F(0)=0,$ we have
$$[uv]_{s,\lambda,T}\leq C_1(||u||_{L^{\infty}(\Omega_T)}[v]_{s,\lambda,T}+||v||_{L^{\infty}(\Omega_T)}[u]_{s,\lambda,T}).$$
$$[F(u)]_{s,\lambda,T}\leq C_2(||u||_{L^{\infty}(\Omega_T)})[u]_{s,\lambda,T},$$
where $C_1$ is a constant and $C_2$ is an increasing function. They are both independent of $\lambda$ and $T.$
\end{theorem}
For the Sobolev embedding theorem, we have
\begin{theorem}\label{sb2}
The following inequalities hold with constants $C$ which are independent of $\lambda\geq1,$
$$||e^{-\lambda T}u||_{L^{\infty}(\omega_T)}\leq C[u]_{4,\lambda,T}, \forall u\in H^{4,\lambda}_{\ast}(\Omega_T),$$
$$||e^{-\lambda T}u||_{W^{1,\infty}(\Omega_T)}\leq C[u]_{6,\lambda,T}, \forall u\in H^{6,\lambda}_{\ast}(\Omega_T).$$
\end{theorem}
For the case when $s$ is odd, we note that
$$[u]_{s,\lambda,T}\leq C\left([u]_{s-1,\lambda,T}+\sum_{|\alpha|=1}[\partial^{\alpha}_{\ast}u]_{s-1,\lambda,T}\right).$$
Thus, we have
\begin{theorem}\label{product3}
Let $s\geq1$ be an odd integer, $\lambda\geq1$ and $T\in\R.$ Then, for all functions $u,v\in H^{s,\lambda}_{\ast}(\Omega_T)\cap L^{\infty}(\Omega_T)$ and $C^{\infty}$ function $F$ of $u,$ $F(0)=0,$ we have
$$[uv]_{s,\lambda,T}\leq C_1(||u||_{W^{1,tan}(\Omega_T)}[v]_{s,\lambda,T}+||v||_{W^{1,tan}(\Omega_T)}[u]_{s,\lambda,T}),$$
$$[F(u)]_{s,\lambda,T}\leq C_2(||u||_{W^{1,tan}(\Omega_T)})[u]_{s,\lambda,T},$$
where $C_1$ is a constant and $C_2$ is an increasing function, and  both are independent of $\lambda$ and $T.$
\end{theorem}
For the Sobolev embedding theorem, we have
\begin{theorem}\label{sb3}
The following inequalities hold with constants $C$ which are independent of $\lambda\geq1,$
$$||e^{-\lambda T}u||_{W^{1,tan}(\Omega_T)}\leq C[u]_{5,\lambda,T}, \forall u\in H^{5,\lambda}_{\ast}(\Omega_T),$$
$$||e^{-\lambda T}u||_{W^{2,tan}(\Omega_T)}\leq C[u]_{7,\lambda,T}, \forall u\in H^{7,\lambda}_{\ast}(\Omega_T).$$
\end{theorem}


\section{A \textit{priori} Estimate of Tangential Derivatives}\label{Appendix2}

We prove Lemma \ref{tangentialestimate} here. We differentiate \eqref{new} with the $l$th order tangential derivatives $D^{\beta}=\partial^{\alpha_0}_t\partial^{\alpha_1}_1,$ with the multi-index $\beta=(\alpha_0,\alpha_1,0,0),$ $l=\alpha_0+\alpha_1$ and $0\leq l\leq s, $
and obtain
\begin{equation}\label{ddd}
\begin{split}
&A_0\partial_tD^{\beta}W+A_1\partial_1D^{\beta} W+I_2\partial_2D^{\beta}W+CD^{\beta}W\\
&+\sum_{\langle\beta'\rangle=1,\beta'\leq\beta}c_{\beta'}\left[D^{\beta'}A_0D^{\beta'-\beta}\partial_tW+D^{\beta'}A_1D^{\beta-\beta'}\partial_1W\right]=D^{\beta}F\\
&+ \sum_{\langle \beta'\rangle\geq2,\beta'\leq\beta}c_{\beta'}\left[D^{\beta'}A_0D^{\beta-\beta'}\partial_tW+D^{\beta'}A_1D^{\beta-\beta'}\partial_1W\right]+\sum_{\langle \beta'\rangle\geq1,\beta'\leq\beta}c_{\beta'}\left[D^{\beta'}CD^{\beta-\beta'}W\right],
\end{split}
\end{equation}
where $c_{\beta'}$ is a fixed constant that only depends on $\beta'$.
All the terms on the left hand side of \eqref{ddd} are the $l$ and $l+1$th order derivatives, and all the terms on the right hand side of \eqref{ddd} are no more than $l$th derivatives.
Then, we define the vector
$$W^{(l)}=\{\partial^{\alpha_0}_t\partial^{\alpha_1}_1W,\alpha_0+\alpha_1=l\}.$$
Hence, we can combine all the $l$th order derivatives of the interior equations \eqref{new} into system
\begin{equation}\label{ddd2}
\mathcal{A}_0\partial_tW^{(l)}+\mathcal{A}_1\partial_1W^{(l)}+\mathcal{I}\partial_2W^{(l)}+\mathcal{C}W^{(l)}=\mathcal{F}^{(l)},
\end{equation}
where $\mathcal{A}_0$,$\mathcal{A}_1$ and $\mathcal{I}$ are block diagonal matrices with blocks $A_0,A_1$ and $I_2$ respectively and $\mathcal{F}^{(l)}$
denotes all the terms on the right hand side of \eqref{ddd}. In terms of the boundary conditions, we can also apply the $l$th order tangential derivatives to \eqref{boundary8} to obtain
\begin{equation}\label{ddd22}
\begin{split}
&b\nabla D^{\beta}\psi+b_{\sharp}D^{\beta}\psi+\mathbf{M}D^{\beta} W^{nc}|_{x_2=0}=D^{\beta}g+\\
&+\sum_{\langle\beta'\rangle\geq1,\beta'\leq\beta}c_{\beta'}\left[D^{\beta'}\mathbf{M}D^{\beta-\beta'}W^{nc}|_{x_2=0}+D^{\beta'}b\nabla D^{\beta-\beta'}\psi+D^{\beta'}b_{\sharp}D^{\beta-\beta'}\psi\right].
\end{split}
\end{equation}
Then, combining all the derivatives of the boundary conditions, we can rewrite \eqref{ddd22} as
\begin{equation}\label{ddd3}
\mathcal{B}\nabla\psi^{(l)}+\mathcal{B}_{\sharp}\psi^{(l)}+\mathcal{M}W^{(l),nc}|_{x_2=0}={\mathcal{G}}^{(l)}.
\end{equation}
Adopting the same idea in \cite{Coulombel2008}, we can easily verify that the new system has the same leading order derivatives as \eqref{new} and \eqref{boundary8}. Hence, it is straightforward to obtain similar estimate to \eqref{estimate3}:
\begin{equation}\label{ddd4}
\begin{split}
&\sqrt{\lambda}||W^{(l)}||_{L^2_{\lambda}(\Omega_T)}+||W^{(l),nc}|_{x_2=0}||_{L^2_{\lambda}(\omega_T)}+||\psi^{(l)}||_{H^1_{\lambda}(\omega_T)}\\
&\leq C_l(\lambda^{-\frac{3}{2}}||\mathcal{F}^{(l)}||_{L^2(H^1_{\lambda}(\omega_T))}+\lambda^{-1}||\mathcal{G}^{(l)}||_{H^1_{\lambda}(\omega_T)}).
\end{split}
\end{equation}
Now, we are ready to estimate $\mathcal{F}^{(l)}$ and $\mathcal{G}^{(l)}$ which are the right hand side of \eqref{ddd} and \eqref{ddd3}.
First, we have
$$||D^{\beta}F||_{L^2(H^1_{\lambda}(\omega_T))}\leq\lambda||D^{\beta}F||_{L^2_{\lambda}(\Omega_T)}+||\nabla^{tan}D^{\beta}F||_{L^2_{\lambda}(\Omega_T)}\leq||F||_{L^2(H^{l+1}_{\lambda}(\omega_T))},$$
$$||D^{\beta}g||_{H^1_{\lambda}(\omega_T)}\leq||g||_{H^{l+1}_{\lambda}(\omega_T)}.$$
For simplicity, we omit the subscripts $r,l,\pm$.
We focus on the term $D^{\beta'}A_0\partial_tD^{\beta-\beta'}W$ in $L^2(H^1_{\lambda}(\omega_T)),$ where $\langle\beta'\rangle\geq2$. For fixed $x_2>0,$ we apply the Gagliardo-Nirenberg inequalities (Theorem \ref{GN} and Theorem \ref{composed}).
Then, we integrate with respect to $x_2.$ Decomposing $\beta'=\beta''+\beta_1,\langle\beta_1\rangle=1.$
 Noticing that $D^{\beta_1}A_0$ is a $C^{\infty}$ function of
$(\dot{U},\nabla\dot{\Phi},\nabla^{tan}\dot{U},\nabla^{tan}\nabla\dot{\Phi})$ that vanishes at the origin and
$$||(\dot{U}_r,\dot{U}_l)||_{W^{2,\infty}(\Omega_T)}+||(\dot{\Phi}_r,\dot{\Phi}_l)||_{W^{3,\infty}(\Omega_T)}\leq[(\dot{U},\dot{\Phi})]_{10,\lambda,T}\leq K,$$ we obtain
$$||D^{\beta'}A_0(x_2)||_{L^p_{\lambda}(\omega_T)}\leq C(K)||(\dot{U},\nabla\dot{\Phi},\nabla^{tan}\dot{U},\nabla^{tan}\nabla\dot{\Phi})(x_2)||^{\frac{2}{p}}_{H^{l-1}_{\lambda}(\omega_T)},\; \frac{2}{p}=\frac{\langle\beta'\rangle-1}{\langle \beta\rangle-1},$$
$$||D^{\beta-\beta'}\partial_tW(x_2)||_{L^q_{\lambda}(\omega_T)}\leq C||\partial_tW(x_2)||^{\frac{2}{p}}_{L^{\infty}(\omega_T)}||\partial_tW(x_2)||^{\frac{2}{q}}_{H^{l-1}_{\lambda}(\omega_T)},\; \frac{2}{q}=\frac{\langle\beta\rangle-\langle\beta'\rangle}{\langle\beta\rangle-1}.$$
Using H$\ddot{\rm o}$lder's inequality and integrating with the respect to $x_2$, we obtain:
\begin{equation}
||D^{\beta'}A_0\partial_tD^{\beta-\beta'}W||_{L^2_{\lambda}(\Omega_T)}\leq C(K)\{||W||_{L^2(H^l_{\lambda}(\omega_T))}+||W||_{W^{1,tan}(\Omega_T)}||(\dot{U},\nabla\dot{\Phi})||_{L^2(H^l_{\lambda}(\omega_T))}\}.\\
\end{equation}
Decomposing $\beta'=\beta''+\beta_2,\langle\beta_2\rangle=2,$ we obtain
\begin{equation}
\begin{split}
&||\partial_t(D^{\beta'}A_0\partial_tD^{\beta-\beta'}W)||_{L^2_{\lambda}(\Omega_T)}\\
&\leq C(K)\{||W||_{L^2(H^l_{\lambda}(\omega_T))}+||W||_{W^{1,tan}(\Omega_T)}||(\dot{U},\nabla\dot{\Phi})||_{L^2(H^{l+1}_{\lambda}(\omega_T))}\},\\
\end{split}
\end{equation}
\begin{equation}
\begin{split}
&||\partial_1(D^{\beta'}A_0\partial_tD^{\beta-\beta'}W)||_{L^2_{\lambda}(\Omega_T)}\\
&\leq C(K)\{||W||_{L^2(H^l_{\lambda}(\omega_T))}+||W||_{W^{1,tan}(\Omega_T)}||(\dot{U},\nabla\dot{\Phi})||_{L^2(H^{l+1}_{\lambda}(\omega_T))}\}.\\
\end{split}
\end{equation}
Thus, we have
\begin{equation}
\begin{split}
&||D^{\beta'}A_0D^{\beta-\beta'}\partial_tW||_{L^2(H^1_{\lambda}(\omega_T))}\\
&\leq C(K)\{\lambda||W||_{L^2(H^l_{\lambda}(\omega_T))}+||W||_{W^{1,tan}(\Omega_T)}||(\dot{U},\nabla\dot{\Phi})||_{L^2(H^{l+1}_{\lambda}(\omega_T))}\}.\\
\end{split}
\end{equation}
Similarly,
\begin{equation}
\begin{split}
&||D^{\beta'}A_1D^{\beta-\beta'}\partial_1W||_{L^2(H^1_{\lambda}(\omega_T))}\\
&\leq C(K)\{\lambda||W||_{L^2(H^l_{\lambda}(\omega_T))}+||W||_{W^{1,tan}(\Omega_T)}||(\dot{U},\nabla\dot{\Phi})||_{L^2(H^{l+1}_{\lambda}(\omega_T))}\},\\
&||D^{\beta'}CD^{\beta-\beta'}W||_{L^2(H^1_{\lambda}(\omega_T))},\\
&\leq C(K)\{\lambda||W||_{L^2(H^l_{\lambda}(\omega_T))}+||W||_{W^{1,tan}(\Omega_T)}||(\dot{U},\nabla\dot{U},\nabla\dot{\Phi},\nabla^{tan}\nabla\dot{\Phi})||_{L^2(H^{l+1}_{\lambda}(\omega_T))}\}.\\
\end{split}
\end{equation}
From all the above estimates, 
\begin{equation}\label{Fe}
\begin{split}
&||\mathcal{F}^{(l)}||_{L^2(H^1_{\lambda}(\omega_T))}\\
&\leq C(K)\{||F||_{L^2(H^{l+1}_{\lambda}(\omega_T))}+\lambda||W||_{L^2(H^l_{\lambda}(\omega_T))}+||W||_{W^{1,tan}(\Omega_T)}||(\dot{U},\nabla\dot{\Phi})||_{L^2(H^{l+2}_{\lambda}(\omega_T))}\}.\\
\end{split}
\end{equation}
Similarly, we estimate the right hand side of \eqref{ddd22} and obtain
 \begin{equation}\label{Ge}
\begin{split}
&||\mathcal{G}^{(l)}||_{H^1_{\lambda}(\omega_T)}\leq C(K)\{||g||_{H^{l+1}_{\lambda}(\omega_T)}\\
&\quad+||W^{nc}|_{x_2=0}||_{H^l_{\lambda}(\omega_T)}+||\psi||_{H^{l+1}_{\lambda}(\omega_T)}+||W^{nc}|_{x_2=0}||_{L^{\infty}(\omega_T)}||(\dot{U}|_{x_2=0},\nabla\psi)||_{H^{l+1}_{\lambda}(\omega_T)}\\
&\quad+||\psi||_{W^{1,\infty}(\omega_T)}||(\dot{U},\partial_2\dot{U},\nabla\dot{\Phi})|_{x_2=0}||_{H^{l+1}_{\lambda}(\omega_T)}\}.
\end{split}
\end{equation}
Substituting \eqref{Fe} and \eqref{Ge} into \eqref{ddd4}, multiplying the obtained inequality by $\lambda^{s-l}$, summing over $l=0,\cdots,s,$ and choosing $\lambda$ large enough, we can complete the proof of Lemma \ref{tangentialestimate}.

\section*{Acknowledgments}
 F. Huang was supported in part by National Center for Mathematics and Interdisciplinary Sciences, AMSS, CAS and NSFC Grant No. 11371349 and 11688101.
 D. Wang was supported in part by NSF grants DMS-1312800 and DMS-1613213.
 D.  Yuan was supported by China Scholarship Council No.201704910503.
The authors would like to thank Professor Robin Ming Chen and Jilong Hu of  the University of Pittsburgh for   valuable discussions.
We are also grateful to  the anonymous referee for valuable comments and suggestions.
%

\end{document}